\documentclass{amsart}

\usepackage{ifxetex,ifluatex}
\newif\ifxetexorluatex
\ifxetex
  \xetexorluatextrue
\else
  \ifluatex
    \xetexorluatextrue
  \else
    \xetexorluatexfalse
  \fi
\fi

\ifxetexorluatex
  \usepackage{fontspec}
\else
  \usepackage[T1]{fontenc}
  \usepackage[utf8]{inputenc}
  \usepackage{lmodern}
\fi

\usepackage[left=1in,top=2in,right=1in,bottom=2in,letterpaper]{geometry}
\usepackage{amsmath}
\usepackage{rotating}
\usepackage{amsthm,amssymb,amsfonts}
\usepackage{graphicx}
\usepackage{xcolor}
\usepackage{verbatim}
\usepackage{lscape}
\usepackage{tikz-cd}

\usepackage{hyperref}

\theoremstyle{plain}
\newtheorem{theorem}{Theorem}[section]
\newtheorem{lemma}[theorem]{Lemma}
\newtheorem{proposition}[theorem]{Proposition}
\newtheorem{corollary}[theorem]{Corollary}

\newtheorem*{definition}{Definition}
\newtheorem{example}[theorem]{Example}

\newtheorem{remark}[theorem]{Remark}

\numberwithin{equation}{section}
\makeatletter
\let\@wraptoccontribs\wraptoccontribs
\makeatother

\title{Platonic solids and high genus covers of lattice surfaces}
\author{Jayadev S. Athreya}
\email{jathreya@uw.edu}
\address{Department of Mathematics, University of Washington, Padelford Hall, Seattle, WA 98195, USA}

\author{David Aulicino}
\email{david.aulicino@brooklyn.cuny.edu}
\address{Department of Mathematics, Brooklyn College,
Room 1156, Ingersoll Hall 
2900 Bedford Avenue 
Brooklyn, NY 11210-2889 
USA}
\address{Department of Mathematics, CUNY Graduate Center, 365 5th Ave, New York, NY 10016, USA}

\author{W. Patrick Hooper}
\email{whooper@ccny.cuny.edu}
\address{Department of Mathematics, The City College of New York, 160 Convent Ave, New York, NY 10031, USA}
\address{Department of Mathematics, CUNY Graduate Center, 365 5th Ave, New York, NY 10016, USA}

\date{}

\contrib[With an Appendix by]{Anja Randecker}
\email{anja@math.toronto.edu}
\address{Department of Mathematics 
University of Toronto 
Room 6290, 40 St. George Street 
Toronto, ON, M5S 2E4 Canada.
}

    \thanks{J.S.A. was partially supported by NSF CAREER grant DMS 1559860}
    \thanks{D.A. was partially supported by NSF DMS - 1738381, DMS - 1600360 and PSC-CUNY grants 60571-00 48 and 61639-00 49}
    \thanks{W.P.H. was partially supported by  NSF DMS 1500965 and PSC-CUNY grant 60708-00 48.}

\newcommand{\splin}{SL(2,\mathbb{R})}

\newcommand{\nc}{\newcommand}

\nc\bB{\mathbb{B}}
\nc\bC{\mathbb{C}}
\nc\bD{\mathbb{D}}
\nc\bE{\mathbb{E}}
\nc\bF{\mathbb{F}}
\nc\bG{\mathbb{G}}
\nc\bH{\mathbb{H}}
\nc\bI{\mathbb{I}}
\nc{\bJ}{\mathbb{J}}
\nc\bK{\mathbb{K}}
\nc\bL{\mathbb{L}}
\nc\bM{\mathbb{M}}
\nc\bN{\mathbb{N}}
\nc\bO{\mathbb{O}}
\nc\bP{\mathbb{P}}
\nc\bQ{\mathbb{Q}}
\nc\bR{\mathbb{R}}
\nc\bS{\mathbb{S}}
\nc\bT{\mathbb{T}}
\nc\bU{\mathbb{U}}
\nc\bV{\mathbb{V}}
\nc\bW{\mathbb{W}}
\nc\bY{\mathbb{Y}}
\nc\bX{\mathbb{X}}
\nc\bZ{\mathbb{Z}}

\nc\cA{\mathcal{A}}
\nc\cB{\mathcal{B}}
\nc\cC{\mathcal{C}}
\nc\cD{\mathcal{D}}
\nc\cE{\mathcal{E}}
\nc\cF{\mathcal{F}}
\nc\cG{\mathcal{G}}
\nc\cH{\mathcal{H}}
\nc\cI{\mathcal{I}}
\nc{\cJ}{\mathcal{J}}
\nc\cK{\mathcal{K}}
\nc\cM{\mathcal{M}}
\nc\cN{\mathcal{N}}
\nc\cO{\mathcal{O}}
\nc\cP{\mathcal{P}}
\nc\cQ{\mathcal{Q}}
\nc\cS{\mathcal{S}}
\nc\cT{\mathcal{T}}
\nc\cU{\mathcal{U}}
\nc\cV{\mathcal{V}}
\nc\cW{\mathcal{W}}
\nc\cY{\mathcal{Y}}
\nc\cX{\mathcal{X}}
\nc\cZ{\mathcal{Z}}

\nc\hol{\mathbf{hol}}
\nc\PSL{\mathit{PSL}}
\nc\SL{\mathit{SL}}
\nc\PSO{\mathit{PSO}}
\nc\SO{\mathit{SO}}
\nc\GL{\mathit{GL}}
\nc\PGL{\mathit{PGL}}
\nc\bs{\backslash}
\nc\Area{\mathit{Area}}
\nc\Isom{\mathit{Isom}}
\nc\Perm{\mathit{Perm}}
\nc\Aff{\mathit{Aff}}
\nc\bu{\mathbf{u}}
\renewcommand\Im{\mathit{Im}}

\begin{document}

\begin{abstract} We study the translation surfaces obtained by considering the unfoldings of the surfaces of Platonic solids. We show that they are all lattice surfaces and we compute the topology of the associated Teichm\"uller curves. Using an algorithm that can be used generally to compute Teichm\"uller curves of translation covers of primitive lattice surfaces, we show that the Teichm\"uller curve
of the unfolded dodecahedron has genus $131$ with $19$ cone singularities and $362$ cusps. We provide both theoretical and rigorous computer-assisted proofs that there are no closed saddle connections on the surfaces associated to the tetrahedron, octahedron, cube, and icosahedron. We show that there are exactly $31$ equivalence classes of closed saddle connections on the dodecahedron, where equivalence is defined up to affine automorphisms of the translation cover. Techniques established here apply more generally to Platonic surfaces and even more generally to translation covers of primitive lattice surfaces and their Euclidean cone surface and billiard table quotients. 
\end{abstract}

\maketitle

\tableofcontents

\section{Introduction}

In this paper, we study geodesics on the surfaces of Platonic solids, with a particular focus on the case of the dodecahedron.   We study the affine symmetry groups of natural covers, and as an application of our results, we give a unified proof that the only Platonic solid with a closed geodesic passing through exactly one vertex is the dodecahedron.  The existence of such a trajectory on the dodecahedron was shown in~\cite{AA}.  We classify the natural equivalence classes of such closed singular geodesic trajectories on the dodecahedron and show that there are $31$ such equivalence classes.

The geometry of Platonic solids and polyhedra has been studied for thousands of years, at least from the time of Euclid. Martin Gardner~\cite{Gardner} wrote a beautiful history of some of this study, and more recently Atiyah-Sutcliffe~\cite{AS} surveyed the role of polyhedra in various problems in physics and chemistry.  In 1906, the German mathematicians (colleagues at Hannover) Paul St\"ackel~\cite{Stackel} and Carl Rodenberg~\cite{Rodenberg} wrote papers introducing the study of straight-line trajectories on the surfaces of polyhedra.  

St\"ackel's work (published in May of 1906) explains that a straight line on a face of a polyhedron can be uniquely continued over an edge, giving a notion of geodesic trajectory on the surface of a polyhedron. Rodenberg's paper (published in September of 1906) studies these geodesics for many important examples, focusing on understanding closed billiard trajectories on the regular $n$-gon and closed geodesics on the (regular) dodecahedron.

Thirty years later, Ralph Fox and Richard Kershner, 23-year-old graduate students at Johns Hopkins, described~\cite{FK36} an unfolding procedure which, in the setting where the angles of a polyhedron are rational, yields a compact \emph{translation surface}.  As a consequence, Fox and Kershner proved that on unfoldings of rational polyhedra, topological closures of non-closed geodesics consist of regions bounded by finitely many \emph{saddle connections}, i.e., geodesics connecting vertices. This includes the possibility that the region may be the entire unfolded surface.

The final paragraph of their paper proved that the unfoldings of the tetrahedron, octahedron, cube, and icosahedron have the property that the existence of a closed non-singular trajectory implies that all parallel regular trajectories are closed because these polyhedra are constructed from polygons that tile the plane. They observed that the translation surfaces that arise as unfoldings of these polyhedra are completely periodic because they are {\em square-tiled} (or {\em arithmetic}): the unfoldings are covers of flat tori branched over one point. Fox-Kershner's work left open the problem of finding proper minimal components on the dodecahedron and rational triangular billiards.  

To the best of the authors' knowledge, the field lay dormant until the work of Zelmyakov-Katok~\cite{ZK75}, which independently described the unfolding procedure in the context of billiards, and sparked many new developments in the dynamics of polygonal billiard flows.

In this paper, we will show how the problem of understanding the dodecahedron can be answered using the fundamental work of W.~Veech's result \cite{Veech92} on the translation surface defined by the double pentagon. Veech showed (among other things) that for any choice of direction on the double pentagon (Figure \ref{fig:pi}), either the surface decomposed into cylinders of periodic trajectories with singular trajectories on their boundaries; or every trajectory became uniformly distributed and therefore dense on the surface.

A complete answer to the question for rational triangular billiards remains unsolved, and it has been a significant motivation in the study of orbit closures on moduli spaces of translation surfaces. Partial answers are due to Veech \cite{Veech89}, Kenyon-Smillie \cite{KS00}, and Puchta \cite{Puchta}.  Recent significant progress was made by Mirzakhani and Wright \cite{MWFullRank}.

The study of geodesics on surfaces of rational polyhedra attracted relatively sporadic attention, with some interesting contributions by Galperin~\cite{Galperin}, who showed a generic polyhedron does not have closed non-self-intersecting geodesics at all.  The paper of Fuchs and Fuchs \cite{FuchsFuchs} generated much subsequent activity: in their paper, regular closed trajectories were considered on some Platonic solids, and this investigation continued in \cite{FuchsDedicata}. The Veech group and cusp widths of the Teichm\"uller curves of the cube and icosahedron were determined in \cite{FuchsDedicata}, though not using this language.  In Sections \ref{sect:unfoldings} and \ref{sect:nondodec}, we 
demonstrate how to compute these objects using \verb|SageMath| \cite{sagemath} and the \verb|surface_dynamics| package \cite{SurfaceDynamics}.


Davis, Dodds, Traub, and Yang \cite{DDTY} posed the question of the existence of a closed saddle connection, i.e. a geodesic trajectory that starts and ends at the same vertex (without passing through any other vertex), and gave a negative answer for the tetrahedron and cube.

Fuchs \cite{FuchsArnold} continued this investigation by reproving the negative result of \cite{DDTY} and established the non-existence of such a trajectory on the octahedron and icosahedron, and giving a plausible conjectural example of a closed (but not simple) saddle connection.  The existence of a closed simple saddle connection is first rigorously proved in \cite{AA}, and it was found by methods that we elaborate open here. It was recently brought to our attention that a very similar picture (without a proof) appears in Petrunin's beautiful book~\cite{Petrunin}.

In this paper, we classify \emph{all} closed saddle connections on the dodecahedron, up to a natural affine equivalence. By a closed saddle connection, we mean a closed straight-line trajectory on the surface of a polyhedron determined by the data of a vertex of the polyhedron and a tangent vector.  The trajectory is not allowed to hit any other vertex. We call two saddle connections (or cylinders) on a cone surface {\em unfolding-symmetric} if their lifts to the unfolding differ by an affine automorphism (see \S\ref{sect:cone lattice} for a precise definition).  Our main result is that under this equivalence, there are 31 saddle connections on the dodecahedron.

\begin{theorem}
\label{PlatVertDodecRough}
There are $31$ unfolding-symmetry equivalence classes of closed saddle connections on the dodecahedron.
\end{theorem}

\begin{remark}
The closed saddle connection found in \cite{AA} corresponds to a trajectory in equivalence class number 1 in the table in Appendix \ref{appendix:list}. The vector in \cite{AA} is the vector in this paper multiplied by $\sin(\pi/5)$ due to the fact that in the current paper our dodecahedron is built out of pentagons whose side length is 2, whereas in ~\cite{AA}, the dodecahedron is built out of pentagons inscribed in the unit circle.
\end{remark}

We use the general framework here to give a uniform proof of the following:

\begin{theorem}
\label{PlatVertSBNonDodec}
There are no closed saddle connections on the tetrahedron, octahedron, cube, or icosahedron.
\end{theorem}

The first key observation underlying these results is that the unfolding of every Platonic solid is a \emph{lattice surface} (that is, a surface whose stabilizer in $\SL(2,\bR)$,
known as the \emph{Veech group}, is a lattice) (see Proposition \ref{prop: natural covering}).  To each such surface is associated its \emph{Teichm\"uller curve}, the collection of area-preserving affine deformations of the surface, which is parametrized by the hyperbolic plane modulo the Veech group.  Two proofs of Theorem \ref{PlatVertSBNonDodec} are provided: one computational $\S$\ref{sect:blocking}, and one theoretical $\S$\ref{sect:weierstrass}.  The new concept underpinning the theoretical proof we call a \emph{virtual Weierstrass point (and saddle connection)}.  This concept isolates the exact mechanism for why no closed saddle connections can exist on any Platonic solid except for the dodecahedron.  Finally, while the Teichm\"uller curves of the unfoldings of the other Platonic solids are relatively easy to compute (see \S\ref{sect:topology}), the Teichm\"uller curve of the unfolded dodecahedron is much more challenging to understand due to its large size:

\begin{theorem}
\label{TeichDodec}
The unfolding $\tilde D$ of the dodecahedron lies in the stratum $\cH(8^{20})$. The associated Teichm\"uller curve has genus $131$; $362$ cusps; $18$ cone singularities of cone angle $\pi$; and a single cone singularity with cone angle $\frac{2\pi}{5}$, which represents the unfolded dodecahedron. 
\end{theorem}

The complexity of the Teichm\"uller curve makes it challenging to count objects on the dodecahedron. For example:

\begin{theorem}
\label{thm:422}
There are $422$ unfolding-symmetry equivalence classes of maximal cylinders on the dodecahedron. 
\end{theorem}
This complexity necessitates the development of a computational approach to fully understand geodesics on the dodecahedron and its unfolding. Theorem \ref{NormCov} of this paper notes that $\tilde D$ is a degree 60 regular cover of the double pentagon, which is a primitive lattice surface in ${\mathcal H}(2)$. The computational techniques we develop can be applied more generally to the case of finite translation covers of lattice surfaces. In particular:
\begin{itemize}
\item \S \ref{sect:computing veech groups} gives an algorithm for computing the Veech group and combinatorial data describing the Teichm\"uller curve of a finite translation cover of a primitive lattice surface. (Such an algorithm was already known for square tiled surfaces \cite{S04}.  Furthermore, it was extended to all $n$-gons and double $n$-gons in \cite{Finster, FinsterMAThesis}).
\item The techniques of \S \ref{sect:closed} can be used to classify closed geodesics and saddle connections on Euclidean cone surfaces (given by $k$-differentials) whose unfoldings cover a lattice surface.
\end{itemize}



An important aspect of this work is the fact that it provides a rich infinite class of highly symmetric translation surfaces.  The Platonic solids generalize to objects called Platonic surfaces, which realize a singular flat geometry.  We prove in Theorem \ref{NormCov} that unfoldings of Platonic surfaces are regular covers of Veech's $p$-gons and double $p$-gons.  As a result, computational experiments show that the Lyapunov exponents of the Kontsevich-Zorich cocycle are very exceptional in the sense that the Lyapunov spectrum is very far from being simple.  These surfaces uniformize to Riemann surfaces with large automorphism group, and include all of the Hurwitz surfaces.

In a sequel to this paper, the second named author uses this result to obtain a new computation of the rotation group of the dodecahedron and to produce a simple algorithm in order to find a large normal subgroup of the rotation group of any Platonic surface \cite{Aulicino}.  In certain cases, this provides uniform elementary computations of the automorphism group of several classical Riemann surfaces that are known to have large automorphism group.

Finally, the authors feel that the work of St\"ackel and Rodenberg deserves recognition considering that they were the first to consider the objects that have become so central to an important field of modern mathematics.  Anja Randecker has summarized the content of these papers in Appendix \ref{appendix:History}.




\subsubsection*{Organization} Our paper is organized as follows. In \S\ref{sect:preliminaries}, we introduce the notions of Euclidean cone structures on surfaces, translation surfaces, and a description of the unfolding procedure of Fox and Kershner using both flat geometry and complex analysis. In \S\ref{sect:decompositions}, we describe the structure of unfoldings of surfaces tiled by regular $n$-gons, and in \S\ref{sect:unfoldings}, we apply this in order to understand unfoldings of the Platonic solids. In \S\ref{sect:weierstrass}, we use some basic facts about hyperelliptic surfaces and Weierstrass points to prove that the Platonic solids, excluding the dodecahedron, do not admit closed saddle connections. In \S\ref{sect:nondodec}, we compute the (arithmetic) unfoldings of the Platonic solids (excluding the dodecahedron), compute their Veech groups, and give an alternate (computational) proof that they do not admit closed saddle connections. In \S\ref{sect:computing veech groups}, we give a general algorithm to compute the affine symmetry group of a translation cover when the base of the covering has a large affine symmetry group. Using this, in \S\ref{sect:dodecunfolding}, we compute the unfolding of the dodecahedron, its affine symmetry group, and the associated Teichm\"uller curve. Finally, in \S\ref{sect:closed}, we compute the collection of affine symmetry group orbits of closed saddle connections on the dodecahedron, and we give a complete list of them (ordered according to a certain word length) in Appendix \ref{appendix:list}. Appendix \ref{appendix:History}, by Anja Randecker, summarizes some of the early work on this problem. Appendix \ref{appendix:flatsurf} describes the inner workings of the Sage package FlatSurf, which we use extensively in our computations involving the dodecahedron. Appendix \ref{appendix:shortest} gives a geometrically shortest representation for each of the 31 equivalence classes of closed saddle connections on the dodecahedron.

\subsubsection*{Auxiliary Files}

For the convenience of the reader, Sage notebooks are made available with all of the code in this paper.  The Sage notebook \verb|code_from_the_article.ipynb| contains all of the code used to generate the data in this paper.  The Sage notebook \verb|Figures.ipynb| contains all of the code used for generating the figures that appear in this paper.  Both of these notebooks are available as supplemental files on ArXiv.  More detailed instructions and descriptions are available at the website below.
\begin{center}
\url{http://userhome.brooklyn.cuny.edu/aulicino/dodecahedron/}
\end{center}
On this website, representatives of each equivalence class can be found for the printing and constructing of dodecahedrons.

\subsubsection*{Acknowledgments} We would like to thank Joshua Bowman, Diana Davis, Myriam Finster, Dmitri Fuchs, Samuel Leli\`evre, Anja Randecker, and Gabriela Weitze-Schmith\"usen for helpful discussions. We would like to thank the anonymous referee for their careful reading of the paper and clarifying remarks, and for pointing us to the references~\cite{Galperin, Petrunin}.



\section{Flat metrics and unfolding} 
\label{sect:preliminaries}

In this section, we recall the basics of flat metrics on surfaces (\S\ref{sect:flat}) and describe (\S\ref{sect:unfolding}) the unfolding procedure of Fox-Kershner. We show how to view this construction complex analytically in \S\ref{sect:complex}. In \S\ref{sect:additional}, we define the Veech group of a translation surface and describe how this group behaves under coverings. Section \ref{sect:hyperbolic} describes the hyperbolic geometry of the space of affine deformations of a translation surface.

\subsection{Flat metrics on surfaces}\label{sect:flat}

\subsubsection*{Cone surfaces}
A \emph{Euclidean cone surface}, or \emph{cone surface} for short, $S$ is a closed topological surface together with the choice of a finite subset $\Sigma \subset S$ and an 
atlas of charts from $S^\ast = S \smallsetminus \Sigma$ to
the plane (i.e., a collection of local homeomorphisms $\phi_i: U_i \to \bR^2$ with $\bigcup U_i=S^\ast$)
so that the transition functions (maps of the form $\phi_j \circ \phi_i^{-1}: \phi_i(U_i \cap U_j) \to \phi_j(U_i \cap U_j)$)
are restrictions of orientation preserving isometries of $\bR^2$
and so that we can continuously extend the pullback metric to $\Sigma$. This extension to $\Sigma$ makes the points in $\Sigma$ 
cone singularities, so we call the points in $\Sigma$ {\em singularities}. We call a point of $\Sigma$ with cone angle $2\pi$ a {\em removable singularity} or a {\em marked point}.

Let $S_1$ and $S_2$ be cone surfaces determined by atlases $\{\phi_i:U_i \to \bR^2:~i \in \Lambda_1\}$
and  $\{\psi_j:V_j \to \bR^2:~j \in \Lambda_2\}$, respectively. We consider $S_1$ and $S_2$ to be the same 
cone surface if there is an orientation preserving homeomorphism $h:S_1^\ast \to S_2^\ast$ so that the atlas
\begin{equation}
\label{eq:combine}
\{\phi_i:U_i \to \bR^2:~i \in \Lambda_1\} \cup \{\psi_j\circ h:h^{-1}(V_j) \to \bR^2:~j \in \Lambda_2\}
\end{equation}
also determines a cone surface structure.

\subsubsection*{Geometric objects}
A \emph{saddle connection} on a Euclidean cone surface is a geodesic segment in the flat metric that connects
two singular points (possibly the same point) with no singular points in its interior. If it connects a singular point to itself, we call it a \emph{closed saddle connection}.

An {\em immersed cylinder} (or simply {\em cylinder}) on a cone surface $S$ is an immersion from a Euclidean cylinder $C=(\bR/c\bZ) \times [0,h]$ for some circumference $c>0$ and height $h>0$ into $S$ which is a local isometry on the interior of $C$. We say cylinder $f_1:C_1 \to S$ is contained in cylinder $f_2:C_2 \to S$ if there is an isometric embedding $g:C_1 \to C_2$ so that $g \circ f_1=f_2$, and consider two cylinders to be the same if they are contained in each other. Cylinder $f_1:C_1 \to S$ covers $f_2:C_2 \to S$ if there is a covering map $g:C_1 \to C_2$ so that $g \circ f_1=f_2$. We say a cylinder on $S$ is {\em maximal} if it is not contained in another cylinder on $S$ and if it is does not cover another cylinder on $S$.

\subsubsection*{Parallel transport}

Select a base point $x_0$ for $S^\ast=S \smallsetminus \Sigma$, where $S$ is a cone surface as above. 
Since the surface is flat, parallel transport around loops based at $x_0$ gives a group homomorphism from $\pi_1(S^*,x_0)$ to the rotation group $\SO(2, \bR)$ acting on the tangent space to $x_0$.
Since $\SO(2, \bR)$ is abelian, this homomorphism factors through integral homology, i.e. we have a well defined homomorphism
\begin{equation}
\label{eq:parallel transport}
PT: H_1(S^*; \bZ) \to \SO(2, \bR).
\end{equation}

\subsubsection*{Translation surfaces}
A {\em translation surface} $S$ is a cone surface where transition functions between pairs of overlapping charts are translations.
Two translation surfaces $S_1$ and $S_2$ are considered equivalent if up to the action of an orientation preserving homeomorphism 
$S_1^\ast \to S_2^\ast$,
the atlases combine as in \eqref{eq:combine} to form a translation surface atlas. As a consequence of this definition, the cone angles at cone singularities are \emph{integer} multiples of $2\pi$.

The group $GL(2,\bR)$ acts linearly on the plane and induces an action on translation surfaces. If $S$ is a translation surface determined by charts $\{\phi_i:U_i \to \bR^2\}$ and $A \in \GL(2,\bR)$, then $A(S)$ is the translation surface determined by 
the charts $\{A \circ \phi_i:U_i \to \bR^2\}$. 

An {\em affine map} $h:S_1 \to S_2$ between translation surfaces is a homeomorphism which is affine in local coordinate charts.
Tangent spaces to regular points of $S_1$ can be naturally identified with $\bR^2$ by differentiating a chart containing the point. 
Let $p_1 \in S_1 \smallsetminus \Sigma$ and define $p_2=h(p_1)$. Then the derivative of $h$ gives a linear map from the tangent space $T_{p_1} S_1 \to T_{p_2} S_2$.
Using the identification between these tangent spaces and $\bR^2$ we can view this map as a linear map $Dh:\bR^2 \to \bR^2$.
We have $Dh \in \GL(2,\bR)$ and its value is independent of the choice of $p_1$. We call $Dh$
{\em the derivative} of $h$.

The {\em affine automorphism group} $\Aff_\pm(S)$ of a translation surface $S$ is the group of affine homeomorphisms $S \to S$. The subgroup of orientation preserving affine automorphisms is $\Aff(S)$.

If $S$ is a translation surface, then the parallel transport homomorphism $PT$ of \eqref{eq:parallel transport} has trivial image. Conversely,
if the parallel transport homomorphism of a cone surface $C$ has trivial image, then $C$ is the same as a cone surface as some translation surface $S$. Note however that $C$ does not uniquely determine $S$ because on a translation surface one has a notion of a horizontal direction while on a cone surface there is no natural choice of direction. Concretely:

\begin{proposition}
\label{prop:same cone surfaces}
Two translation surfaces $S$ and $S'$ are the same as cone surfaces if and only if $S'=R(S)$ for some $R \in SO(2,\bR)$.
\end{proposition}

\subsubsection*{Polyhedra}

A {\em polyhedron} is a finite union of convex polygons ({\em faces}) arranged in Euclidean 3-space so that the intersection of any two polygons is a shared vertex, edge, or the empty set; and so that their union is homeomorphic to the $2$-sphere. A polyhedron has a natural cone surface structure whose metric agrees with the intrinsic metric of the polyhedron.

The {\em angle at a vertex} of a polyhedron is the sum of the angles made at that vertex of the incident faces. 
Following \cite{FK36} we say that a polyhedron is \emph{rational} if the angle at every vertex (that is, the sum of the interior angles of the polygons meeting at that vertex) is of the form $\frac{p}{q}\pi$, where $\frac{p}{q} \in \bQ$.

\begin{remark}
\label{rem:PlatArchJohnAreRat}
Polyhedra whose faces are all regular polygons (e.g., the Platonic, Archimedean, and Johnson solids) are all rational, since the angles of regular polygons are all rational multiples of $\pi$.
\end{remark}

The dodecahedral surface is shown in Figure \ref{fig:dodecahedron}.

\subsection{Unfolding}\label{sect:unfolding}

The {\em unfolding} of a cone surface $S$ is the smallest cover of $S$ branched over $\Sigma$ for which
the parallel transport homomorphism $PT$ has trivial image. 
This unfolding $\tilde S$ is the completion of the cover of $S^\ast$ associated to $\ker PT$. Let $\tilde \Sigma$ be the points added in the completion
and let $\tilde S^*=\tilde S \smallsetminus \tilde \Sigma$. From standard covering space theory and the fact that the finite subgroups of $SO(2,\bR)$ are isomorphic to $\bZ/k\bZ$, we observe:

\begin{proposition}
The covering $\tilde S^\ast \to S^\ast$ is regular with deck group isomorphic to $PT\big(H_1(S^*; \bZ)\big)$.
In particular, the cover is finite if and only if $PT\big(H_1(S^*; \bZ)\big) \cong \bZ/k\bZ$ for some integer $k \geq 1$.
\end{proposition}

\begin{proposition}
The associated parallel transport map of a cone surface coming from a polyhedron has
finite image if and only if the polyhedron is rational. Furthermore if $P$ is a rational polyhedron with vertex angles $\frac{p_i}{q_i}2\pi$, where $\frac{p_i}{q_i}$ is in lowest terms, and $k=\mathit{lcm}(q_i)$, then the image of $PT$ is generated by a rotation by $\frac{2 \pi}{k}$. 
\end{proposition}
\begin{proof}
Since $P$ is a $2$-sphere, $H_1(P^*; \bZ)$ is generated by loops around the cone singularities,
and the image of each such loop under $PT$ gives rotation by the angle at the surrounded vertex. 
\end{proof}
We say that a cone surface is {\em rational} if $PT$ has finite image. From the above, this extends the definition of a rational polyhedron. Also the unfolding of a cone surface is a finite cover if and only if the cone surface is rational.

\begin{figure}
\centerline{\includegraphics[width=6in]{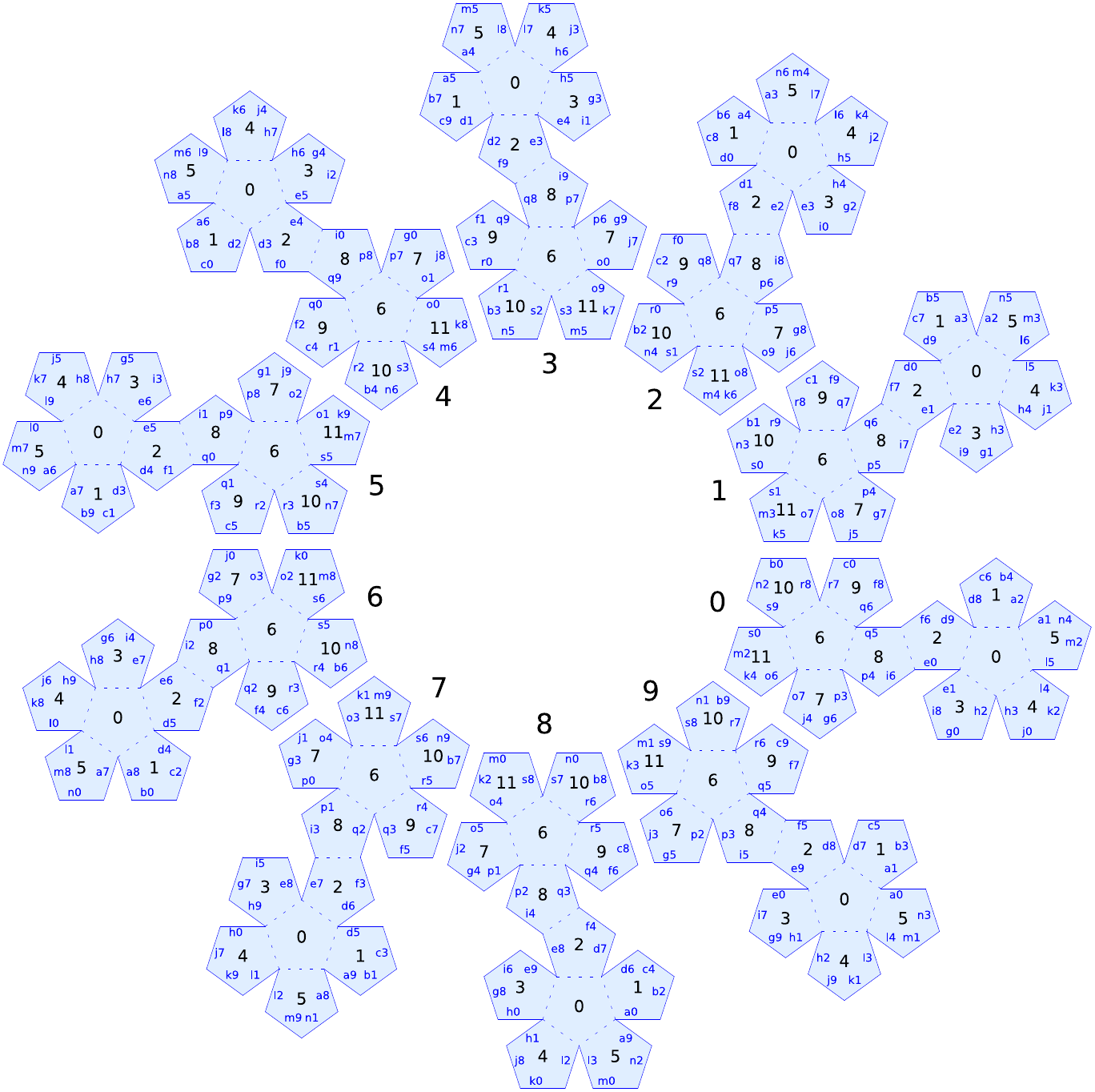}}
\caption{The unfolding of the dodecahedron. Matching labels indicate edge gluings. Letters indicate the covering map to the dodecahedron as in Figure \ref{fig:dodecahedron}.  The numbering on the pentagons and sheets is used in Example \ref{ex:dodecahedron}.}
\label{fig:unfolding}
\end{figure}

\subsection{Complex analysis and $k$-differentials}\label{sect:complex}

\subsubsection*{$k$-differentials}

Let $k$ be a positive integer and $X$ be a Riemann surface. A {meromorphic \em $k$-differential} on $X$ is a differential $m$ on $X$ which can be expressed in local coordinates in the form $m=f(z)~(dz)^k$ where $f$ is meromorphic. 
While $f$ is not globally defined, the zeros and poles of $m$ are. We restrict $m$ to have poles of order larger than $-k$, which ensures that the surface has finite area in the measure determined by the differential. We will also assume that our differentials are not identically zero.

Integrating $\sqrt[k]{m}$ gives charts to the plane so that the transition functions are given by orientation preserving isometries whose rotational components have finite order dividing $k$. This endows $X$ with a Euclidean cone structure so that parallel transport around any loop induces a rotation by a multiple of $\frac{2 \pi}{k}$. In particular, a zero or pole of order $j>-k$ corresponds to a cone point of cone angle $\frac{2\pi(j+k)}{k}$. 

\subsubsection*{Translation surfaces}

Let $S$ be a translation surface. Pulling back the complex structure from the complex plane $\bC$ along the charts, $S$ becomes a Riemann surface $X$; pulling back the $1$-form $d z$ on $\bC$ yields a non-trivial $1$-differential (or {\em holomorphic $1$-form}) $\omega$ on $X$. Integrating the $1$-form gives an atlas of charts determining $S$, so an equivalent definition of a translation surface is a triple $(X, \omega, \Sigma)$ consisting of a Riemann surface $X$, a holomorphic $1$-form $\omega$ on $X$, and a finite subset $\Sigma$ of $X$ containing the zeros of $\omega$.

From Proposition \ref{prop:same cone surfaces}, collections of translation surfaces which are the same as cone surfaces have the form $\{(X,e^{i \theta} \omega,\Sigma):~0 \leq \theta< 2\pi\}$, i.e., such surfaces differ by rotation.

If $(X,\omega)$ is a translation surface and $\gamma:[0,1] \to X$ is a path, the {\em holonomy} of $\gamma$ is 
$$\hol(\gamma)=\int_\gamma~\omega.$$
This notion of holonomy induces a well-defined map $\hol:H_1(X, \Sigma; \bZ) \to \bC$, the image of which is a subgroup of $\bC$ known as the {\em (group of) relative periods}
of the translation surface. A surface is called {\em arithmetic}
or {\em square-tiled} if the relative
periods form a discrete subgroup of $\bC$ (and thus are isomorphic as a group to $\bZ^2$).

\subsubsection*{Rational cone surfaces}

Now let $S$ be a rational cone surface with singular set $\Sigma$. Let $k \geq 1$ be the size of the the image of $PT$. 
Let $\tilde S$ be its unfolding and $\tilde \Sigma$ be the preimage of $\Sigma$ under the covering map.

As above $\tilde S$ determines a family $\{(\tilde X,e^{i \theta} \omega,\tilde \Sigma)\}$. Fix one choice of $\omega$.
Let $\Delta$ denote the deck group of the branched cover $\tilde S \to S$. If $h \in \Delta$, then $h$ acts in local coordinates 
on the translation surface $(\tilde X,\omega)$ preserving $\tilde \Sigma$, so $h$ is an automorphism of the Riemann surface $\tilde X$ and 
$h_\ast(\omega)=e^{\frac{j}{k} 2 \pi i} \omega$. The $k$-th power $\omega^k$ is thus $\Delta$-invariant.  Identifying $S$ with $\tilde S/\Delta$ determines a Riemann surface $X=\tilde X/\Delta$ and a $k$-differential
$m=\pi_\ast(\omega^k)$. The singular set satisfies $\Sigma=\pi(\tilde \Sigma)$.
In analogy to the case of translation surfaces, $S$ corresponds to the family 
$\{(X,e^{i \theta}m, \Sigma)\}$. Given $(X, m, \Sigma)$ we can recover $S$ by integrating $\sqrt[k]{m}$
to obtain local coordinate charts.

\subsubsection*{Strata}

Fix $k \geq 1$ and $g \geq 0$ and let $\alpha=(m_1, \ldots, m_n)$ be a tuple of integers larger than $-k$ which sum to $k(2g-2)$. 
The {\em stratum} $\cH_k(\alpha)$ is the collection of equivalence classes of $(X,m,\Sigma)$ where $X$ is a compact Riemann surface of genus $g$, $m$ is a meromorphic $k$-differential on $X$, and $\Sigma=(x_1, \ldots, x_n)$ is a collection of points on $X$ containing the zeros and poles of $m$
and satisfying the condition that the order of $x_i$ is $m_i$ for $i \in \{1, \ldots, n\}$. Equivalence is given by biholomorphisms mapping differentials and singular sets to each other. Such biholomorphisms are allowed to permute the singlular points. The fact that $\sum m_i=k(2g-2)$ is a consequence of the orbifold Gauss-Bonnet formula (see, e.g., \cite[Proposition 7.9]{FarbMargalit11}). See \cite{BCGGM} for more information on the structure of strata of $k$-differentials.

The $k$-differentials for the Platonic solids (where the underlying surface is $\mathbb C P^1$) and their canonical $k$-covers lie in the strata listed in Table \ref{PlatSolidsTable} by the following elementary result. Note that we commonly use the shorthand exponential notation for strata, e.g. $\cH(4,2,2) = \cH(4,2^2)$.

\begin{lemma}
\label{lem:strata computation}
Let $m$ be a $k$-differential on $\mathbb{C} P^1$ with no zeros and $2k$-simple poles. Then the canonical $k$-cover lies in the stratum $\cH\big((k-2)^{2k}\big)$, and has genus $(k-1)^2$. 
\end{lemma}

\begin{proof} The cone angle of the flat metric at simple pole of a $k$-differential is $2\pi\frac{k-1}{k}$. Taking the $k$-cover ramified over these $2k$ points, we obtain a surface with $2k$ points with cone angle $2\pi(k-1)$, that is a holomorphic $1$-form with $2k$ zeros of order $(k-2)$. The genus of the surface satisfies $2g- 2 = 2k(k-2)$, so $$g = k(k-2) + 1 = (k-1)^2.$$
\end{proof}

\begin{table}
  \centering
\begin{tabular}{c|c|c|c} 
Polyhedron & Stratum of & Stratum of & Genus of \\
 & $k$-differentials & unfolding & unfolding \\
 \hline
 Tetrahedron & $\cH_2(-1^4)$ & $\cH_1(0^4)$ & $1$ \\ 
 Octahedron & $\cH_3(-1^6)$ & $\cH_1(1^6)$ & $4$ \\ 
 Cube & $\cH_4(-1^8)$ & $\cH_1(2^8)$ & $9$ \\ 
 Icosahedron & $\cH_6(-1^{12})$ & $\cH_1(4^{12})$ & $25$ \\ 
 Dodecahedron & $\cH_{10}(-1^{20})$ & $\cH_1(8^{20})$ & $81$ \\ 
\end{tabular}
\caption{Strata of the Platonic solids and their unfoldings.
}
\label{PlatSolidsTable}
\end{table}

\subsection{Affine symmetries and covers}
\label{sect:additional}

\subsubsection*{Veech groups}
The {\em Veech group} $V(S)$ of a translation surface $S$ is the set of derivatives of orientation preserving affine automorphisms of the surface,
$$V(S) = \{A \in \SL(2,\bR):~A(S)=S\}=D\big(\Aff(S)\big).$$
The {\em full Veech group} $V_{\pm}(S)$ of $S$ is the set of all derivatives of affine automorphisms of the surface,
$$V_{\pm}(S) = \{A \in \GL(2,\bR):~A(S) = S\}=D\big(\Aff_\pm(S)\big),$$
which is contained in the subgroup of matrices with determinant $\pm 1$.
Note that these definitions implicitly depends on selecting a singular set $\Sigma$,
and if $S=(X,\omega)$ we will by default take $\Sigma$ to be the zeros of $\omega$.

We say $S$ is a {\em lattice surface} (also known in the literature as a \emph{Veech surface}) if $V(S)$ is a lattice subgroup of $\splin$. 
If $S$ is a lattice surface, then the $\splin$-orbit of $S$ projected from the stratum into the moduli space $\mathcal M_g$ of genus $g$ Riemann surfaces is called
a {\em Teichm\"uller curve}. When $\mathcal M_g$ is equipped with the Teichm\"uller metric, the Teichm\"uller curve
is totally geodesic and isometric to 
$$PSO(2,\bR) \backslash P\SL(2,\bR)/PV(S)$$
endowed with its hyperbolic metric coming from identifying the hyperbolic plane with $PSO(2,\bR) \backslash P\SL(2,\bR)$.
Here $PV(S)$ denotes the image of $V(S)$ in $P\SL(2,\bR)$. We say a lattice surface $S$ is \emph{arithmetic} if $V(S)$ is commensurable to $SL(2, \bZ)$. The \emph{Veech dichotomy} is a celebrated result about lattice surfaces:
\begin{theorem}[Veech Dichotomy \cite{Veech89}]\label{theorem:veech}
A lattice surface is {\em completely periodic}: the existence of a saddle connection
(or a closed geodesic) implies that the surface decomposes into parallel cylinders. Also, any infinite trajectory is dense and equidistributed. Furthermore a lattice surface $S$ decomposes into cylinders parallel to a non-zero vector ${\mathbf v}$ if and only if ${\mathbf v}$ is stabilized by a parabolic element in $V(S)$.
\end{theorem}

\subsubsection*{Coverings and primitivity}
A {\em covering} from $S=(X,\omega)$ and $S'=(Y,\eta)$ is a branched covering map $\pi:S \to S'$ such that $\pi^\ast(\omega)=\eta$.
Following \cite{Moller06}, we say a translation surface $S$ is {\em imprimitive} if there is another translation surface $S'$
of lower genus and a covering $\pi:S \to S'$. 
We call $S$ {\em primitive} if it is not imprimitive. 
Note that the notion of primitivity of $S$ is independent of the choice of the singularity set $\Sigma$ of $S$. From \cite[Theorem 2.6]{Moller06} or \cite[Theorem 2.1]{McM06} we have:

\begin{theorem}
\label{thm:unique}
Any translation surface $S$ covers a primitive translation surface.
If the primitive surface being covered has genus larger than one, then the primitive surface is uniquely determined by $S$.
\end{theorem}

In case $S$ is imprimitive, the Veech group of the corresponding primitive surface has bearing on the Veech group of $S$:

\begin{corollary}
\label{cor:subgroup}
Suppose $S$ is a translation surface with singular set $\Sigma$, and suppose
$S$ covers a primitive translation surface $(Y,\eta)$ of genus larger than one.
Then $V_{\pm}(S) \subset V_{\pm}(Y,\eta)$.
\end{corollary}
\begin{proof}
Fix $S$ and $(Y,\eta)$ as above. Letting $A \in \GL(2,\bR)$, observe that $A(S)$ covers $A(Y,\eta)$ and
$A(Y,\eta)$ is primitive since $(Y,\eta)$ is primitive. If $A(S)=S$, then we must have $A(Y,\eta)=(Y,\eta)$ by the uniqueness claimed in Theorem \ref{thm:unique}. This says $A \in V_{\pm}(S)$ implies $A \in V_{\pm}(Y,\eta)$ as desired.
\end{proof}

In particular, we see that if a lattice surface $S$ covers a primitive translation surface 
$P$ of genus larger than one, then $P$ is also a lattice surface. A parallel observation holds in the case when $S$ is square-tiled, see, for example Hubert-Schmidt~\cite{HS01}.

\begin{theorem}
\label{thm:square tiled coverings}
Let $S$ be a square-tiled translation surface. Let $P$ be the torus with one marked point given by $\bC/\Lambda$ where $\Lambda$ is the (discrete) group of relative periods of $S$. Then there is a unique covering $S \to P$ sending the singular set of $S$ to the marked point of $P$. Furthermore, $V_\pm(S)$ is a finite index subgroup of $V_\pm(P)$.


\end{theorem}


\subsubsection*{Translation coverings and Veech groups}
Let $S=(X,\omega)$ and $S' = (X', \omega')$ be translation surfaces with respective singular sets $\Sigma$ and $\Sigma'$. 
A {\em translation covering} is a covering $\pi:S \to S'$ so that all branched points lie in $\Sigma'$
and $\pi^{-1}(\Sigma')=\Sigma$.
If $\pi: S \to S'$ is a translation covering, there is a close relationship between the Veech groups.
Gutkin and Judge \cite[Theorem 4.9]{GutkinJudge00} showed:
\begin{theorem}
If $\pi:S \to S'$ is a translation covering, then the Veech groups $V_\pm(S)$ and $V_\pm(S')$ are {\em commensurate} (i.e.,
$V_\pm(S) \cap V_\pm(S')$ is a finite index subgroup of both $V_\pm(S)$ and $V_\pm(S')$).
In particular, $S$ is a lattice surface if and only if $S'$ is.
\end{theorem}

\begin{corollary}
\label{cor:finite index}
Let $\pi:S \to P$ is a translation covering, where $P$ is a primitive lattice surface with genus greater than $1$ and with singular set containing no removable singularities. Then $S$ is a lattice surface and $V_\pm(S)$ is a finite index subgroup of $V_\pm(P)$.
\end{corollary}

\subsection{Hyperbolic Geometry}
\label{sect:hyperbolic}

The group of M\"obius transformations preserving the upper half plane $\bH^2=\{z \in \bC:~\Im(z)>0\}$ is given by 
$\PSL(2,\bR)$. This group acts isometrically when the upper half plane is endowed with the metric making it isometric to the hyperbolic plane. The group $\PSL(2,\bR)$ is the orientation-preserving isometry group of $\bH^2$ and the full isometry group is
$\PGL(2,\bR)$. The elements of $\PGL(2,\bR) \smallsetminus \PSL(2,\bR)$ are M\"obius transformations carrying the upper half plane to the lower half plane, which then can be composed with complex conjugation to get an orientation reversing isometry of $\bH^2$. 

Let $U \bH^2$ be the unit tangent bundle of $\bH^2$, and observe that $\PSL(2,\bR)$ acts simply transitively on $U \bH^2$.
Fixing a $\bu_0 \in U \bH^2$, the map sending $A \in \PSL(2,\bR)$ to $A(\bu_0)$ is a homeomorphism $\PSL(2,\bR) \to U \bH^2$.
Thus $\PSL(2,\bR)$ is homeomorphic to the product $\bH^2 \times \bS^1$ and so its universal cover $\widetilde \PSL(2,\bR)$ forms part of a short exact sequence
\begin{equation}
\label{eq:short exact2}
0 \to \bZ \xrightarrow{i} \widetilde \PSL(2,\bR) \xrightarrow{\upsilon} \PSL(2,\bR) \to 0.
\end{equation}

For each $\tilde g \in \widetilde \PSL(2,\bR)$ let $\tilde \gamma_{\tilde g}:[0,1] \to \widetilde \PSL(2,\bR)$ be a path joining the identity
to $\tilde g$. The homotopy class relative to the endpoints $[\tilde \gamma_{\tilde g}]$ is independent of the choice of this path.
Observe 
\begin{equation}
\label{eq:concatenation2}
[ (\tilde \gamma_{\tilde h} \cdot \tilde g) \ast \tilde \gamma_{\tilde g}] = [\tilde \gamma_{\tilde h \tilde g}]
\quad
\text{for all $\tilde g, \tilde h \in \widetilde \PSL(2,\bR)$,}
\end{equation}
where $\alpha \ast \beta$ denotes path concatenation with $\alpha$ following $\beta$ 
and $\tilde \gamma_{\tilde h} \cdot \tilde g$ denotes the path $t \mapsto \tilde \gamma_{\tilde h}(t) \tilde g$.
(Note that many authors use the opposite convention for path concatenation order.)

Now let $G \subset \PSL(2,\bR)$ be a discrete group and set $X_G=\PSL(2,\bR)/G$. The quotient $X_G$ has a natural basepoint given by the coset of the identity which we will write as $IG$. Observe that the universal covering map $p_G: \widetilde \PSL(2,\bR) \to X_G$ (which factors through $\upsilon$ of \eqref{eq:short exact2}) is the universal covering map, and using \eqref{eq:concatenation2} we can observe that we have the following isomorphism to the fundamental group
\begin{equation}
\label{eq:iota2}
\iota_G: 
\begin{array}{ccc}
\upsilon^{-1}(G) & \to & \pi_1(X_G) \\
\tilde g & \mapsto & [p_G \circ \tilde \gamma_{\tilde g}] \\
\end{array},
\end{equation}
where we follow the convention that $[\alpha] [\beta]=[\alpha \ast \beta]$ is the group operation in $\pi_1(X_G)$ where $\ast$ is defined as above.

Now let $H$ be a subgroup of $G$ and define $X_H=\PSL(2,\bR)/H$. Then we get a natural covering map $f:X_H \to X_G$. 

Recall that for a covering map $\phi: Y \to X$, the {\em monodromy action} 
is the action of $\pi_1(X,x_0)$ on the fiber
$\phi^{-1}(x_0)$, 
$M:\pi_1(X,x_0) \times \phi^{-1}(x_0) \to \phi^{-1}(x_0)$.
Here if $[\gamma] \in \pi_1(X,x_0)$ and $y \in \phi^{-1}(x_0)$, then 
$$M([\gamma],y) = \tilde \gamma(1), \quad \text{where $\tilde \gamma$ is the lift
of $\gamma$ to $Y$ so that $\tilde \gamma(0)=y$.}$$
With the conventions above, the monodromy action is a left action. 
It is a standard fact from topology that the monodromy action determines the isomorphism class of the cover since the subgroup of the fundamental group $\pi_1(X,x_0)$ associated to the cover $Y$ with a choice of basepoint $y_0 \in f^{-1}(x_0)$ is the subgroup whose monodromy action stabilizes $y_0$.

Observe that for the covering map $f:X_H \to X_G$, the fiber above the basepoint $f^{-1}(IG)$ is given by the collection of cosets of the form $gH$ with $g \in G$, so
that $f^{-1}(IG) = G / H$. 

\begin{proposition}
\label{prop:monodromy}
The mondromy action of $\pi_1(X_G)$ on the fiber $f^{-1}(IG)$ factors through the natural action of $G$ on $G/H$. 
Concretely,
the action of $\iota_G(\tilde g) \in \pi_1(X_G)$ on $g' H \in G/H$ is 
$$M\big(\iota_G(\tilde g),g'H \big)=g g' H \quad \text{where $g=\upsilon(\tilde g)$.}$$
\end{proposition}
\begin{proof}
By \eqref{eq:iota2}, we see that $\iota_G(\tilde g)$ is represented by the path $p_G \circ \tilde \gamma_{\tilde g}$,
and in particular, $\tilde \gamma_{\tilde g}$ is a lift of the path to $\widetilde \PSL(2,\bR)$. Choose a preimage $\tilde g' \in 
\widetilde \PSL(2,\bR)$ of $g' \in G$. Then $p_H \circ (\tilde \gamma_{\tilde g} \cdot \tilde g')$ is a lift of the path 
$p_G \circ \tilde \gamma_{\tilde g}$ to $X_H$ starting at $p_H(\tilde g')=g' H$. The terminal point of this path is
$p_H(\tilde g \tilde g')=g g' H.$
\end{proof}

\begin{remark}
The convention that $\alpha \ast \beta$ represents $\beta$ followed by $\alpha$ makes Proposition \ref{prop:monodromy}
natural as it appears.
We quotient $X_G$ and $X_H$ on the right because our discrete groups will be Veech groups obtained as stabilizers of a left-action. 
Statements above change when concatenating in the opposite order and when quotienting on the left.
\end{remark}

\section{Cone surfaces with Veech unfoldings}
\label{sect:cone lattice}

In this section, we consider the case of a Euclidean cone surface $D$ which is given by a $k$-differential and has an unfolding $\tilde D$ which is a lattice surface, that is, has an affine automorphism group which is a lattice in $SL(2, \mathbb R)$.

In addition to the Veech Dichotomy (Theorem~\ref{theorem:veech}), we record the following important result of Veech~\cite[\S 3]{Veech89} on lattice surfaces:

\begin{theorem}\label{theorem:finiteorbits} Let $S$ be a lattice surface. The sets of saddle connections and of maximal cylinders decompose into finitely many $\Aff(S)$ orbits. 

\end{theorem}

Since $\tilde D$ is a lattice surface, it is either a branched cover of a primitive lattice surface $\Pi$ of genus larger than one (see Theorem \ref{thm:unique} and Corollary \ref{cor:subgroup}), or it is a branched cover of a torus $\Pi$ branched over one point. In either case, we get the following diagram:
\begin{equation}
\label{eq:diagram}
\begin{tikzcd}
 & \tilde D \arrow{dr}{\pi_{\Pi}} \arrow{dl}[swap]{\pi_D}\\
D   && \Pi
\end{tikzcd}
\end{equation}
and we have $V_\pm(\tilde D) \subset V_\pm(\Pi)$ and both these groups are lattices. As a consequence of the Veech dichotomy, every non-singular geodesic on $D$ is either closed or dense on $D$. Given the setup above, let ${\mathcal A}$ be a collection of objects on $D$ which have well-defined lifts to the unfolding $\tilde D$. Here, we expect multiple choices of the lift which differ by elements of the deck group of $\pi_D$. 

\begin{definition}
Two objects $a,b \in {\mathcal A}$ are \emph{unfolding-symmetric} if $\tilde a$ and $\tilde b$ on $\tilde D$ differ by the action of an element of
$\Aff_\pm(\tilde D)$. 
\end{definition}

Note that the choice of lifts is irrelevant since the deck group of $\pi_D$ is contained in $\Aff_\pm(\tilde D)$. The notion of unfolding-symmetry is an equivalence relation on ${\mathcal A}$. Some possible examples of ${\mathcal A}$ is the collection of saddle connections on $D$, the set of maximal immersed cylinders, or the set of closed saddle connections. Sometimes (but not always) being a closed saddle connection is an unfolding-symmetry invariant notion. A sufficient condition is if the action of the group $\Aff_\pm(\tilde D)$ on the singularities of $\tilde D$ induces an action on the singularities of $D$.

\begin{definition}
A collection of objects ${\mathcal A}$ on $D$ is \emph{unfolding-symmetry admissible} if unfolding-symmetry induces a well-defined equivalence relation on ${\mathcal A}$.
\end{definition}

\begin{lemma}
\label{rem:closed saddle}
If under $\pi_D$, $\Aff_\pm(\tilde D)$ induces a well-defined action on the singularities of $D$ (i.e., if for any two singularities $x_1$ and $x_2$ with $\pi_D(x_1)=\pi_D(x_2)$, for any $f \in \Aff_{\pm}(\tilde D)$ we have $\pi_D \circ f(x_1)=\pi_D \circ f(x_2)$), then 
the collection of closed saddle connections is unfolding-symmetry admissible.
\end{lemma}

\begin{corollary}\label{cor:platonicequiv} If $D$ is the surface of a Platonic solid, then the collection of closed saddle connections is unfolding-symmetry admissible.
\end{corollary}

\begin{proof} Let $D$ be the surface of a Platonic solid. Then $\pi_D$ restricted to the singularities of $\tilde D$ is a bijection to the singularities of $D$ by Lemma \ref{lem:strata computation}.
\end{proof}


By Theorem~\ref{theorem:finiteorbits}, we can enumerate unfolding-symmetry equivalence classes of saddle connections on $D$. Let ${\mathcal A}$ be either of these classes of objects
on $D$, and $\tilde {\mathcal A}$ be the collection of lifts of elements of ${\mathcal A}$ to $\tilde D$.

The projectivization $\bR P^1$ is the space of lines through the origin in $\bR^2$. 
Let ${\mathit Dir}: \tilde {\mathcal A} \to \bR P^1$ send a saddle connection (or cylinder) to
the parallel element of $\bR P^1$. The image ${\mathit Dir}({\mathcal A})$ are the periodic directions. By Veech Dichotomy, $[{\mathbf v}] \in \bR  P^1$ is a periodic direction
if and only if the stabilizer of a representative vector ${\mathbf v}$ of $[{\mathbf v}]$
in $V(S)$ is a non-trivial (necessarily parabolic) subgroup. Each periodic direction is
then associated to a maximal parabolic subgroup of $V(S)$, and the conjugacy classes of such subgroups are naturally associated to cusps of $\bH^2/V(\tilde D)$ of which there are only finitely many since $V(S)$ is a lattice.

The following reduces the problem of finding a representative from each $\Aff_\pm(\tilde D)$-orbit in $\tilde {\mathcal A}$ to looking at finitely many directions and automorphisms preserving those directions.

\begin{lemma}
\label{lem:orbits on cover}
Choose finitely many periodic directions $[{\mathbf v}_1], \ldots, [{\mathbf v}_k] \in \bR P^1$ enumerating the periodic directions up to the action of $V_\pm(\tilde D)$.
Then each $\Aff_\pm(\tilde D)$-orbit in $\tilde {\mathcal A}$ intersects exactly one set of the form ${\mathit Dir}^{-1}([{\mathbf v}_i])$ for $i \in \{1, \ldots, k\}$. 
Furthermore, if for each $i$ we let $P_i \subset \Aff_\pm(\tilde D)$ be the stabilizer of $[{\mathbf v}_i]$, i.e.,
$$P_i = \big\{\phi \in \Aff_\pm(\tilde D)~:~ D(\phi)([{\mathbf v}_i])=[{\mathbf v}_i]\big\},$$
then such non-empty intersection with ${\mathit Dir}^{-1}([{\mathbf v}_i])$ has the form 
$P_i(\tilde \sigma)$ for some $\tilde \sigma \in \tilde {\mathcal A}$ with ${\mathit Dir}(\tilde \sigma)=[{\mathbf v}_i]$.
\end{lemma}
\begin{proof}
The fact that we can take $k$ finite follows from the discussion preceding the lemma.
The remainder can be deduced from naturality of the derivative and the ${\mathit Dir}$-map,
i.e., we have $D\big(\Aff_\pm(\tilde D)\big)=V_\pm(\tilde D)$ and 
${\mathit Dir}\big(\phi(\tilde \sigma)\big)=D\phi \cdot {\mathit Dir}(\tilde \sigma)$ for
all $\phi \in \Aff_\pm(\tilde D)$ and all $\tilde \sigma \in \tilde {\mathcal A}$.
\end{proof}

The following result says that the same works to choose representatives from each unfolding-symmetry class in ${\mathcal A}$:

\begin{proposition}
\label{prop:unfolding-symmetry classes}
Suppose that $\tilde \sigma_1, \ldots, \tilde \sigma_n \in \tilde {\mathcal A}$ gives one representative from every $\Aff_\pm(\tilde D)$-orbit. Let $\sigma_1, \ldots, \sigma_n \in {\mathcal A}$ be the projections to $D$. Then the list $\sigma_1, \ldots, \sigma_n$ has exactly one representative from each unfolding-symmetry class in ${\mathcal A}$.
\end{proposition}
\begin{proof}
A more sophisticated way to say this result is that the covering map $\pi_D$ sends $\Aff_\pm(\tilde D)$-orbits in $\tilde {\mathcal A}$ to unfolding-symmetry classes in ${\mathcal A}$. This is clear from the definition and because $\pi_D:\tilde D \to D$ is a regular covering
with deck group contained in $\Aff_\pm(\tilde D)$.
\end{proof}

\section{Regular polygon decompositions of cone surfaces}
\label{sect:decompositions}

We briefly digress here to discuss a result that holds in great generality.  The reader only interested in the results of this paper can take the surface $\tilde D$ below to be the unfolding of the dodecahedron.  Nevertheless, the general result will play a fundamental role in a subsequent paper~\cite{Aulicino} of the second-named author concerning automorphism groups of regular mappings.

\subsection{Regular polygon decompositions}
Let $D$ be a connected Euclidean cone surface with a choice of a singular set $\Sigma$. For each $n$, we say that a collection $\cP$ of polygons in $D$ is a {\em regular $n$-gon decomposition} if the polygons in $\cP$ are regular $n$-gons, have disjoint interiors, cover $D$, meet edge-to-edge, and the collection of vertices of polygons in $\cP$ coincides with $\Sigma$. For example, the surface of the dodecahedron with its natural Euclidean cone structure has a regular pentagon decomposition consisting of $12$ pentagons.

\subsection{The translation surfaces $\Pi_n$}

There is a family of translation surfaces $\Pi_n$ with regular $n$-gon decompositions which are particularly important to us. If $n$ is even, we define $\Pi_n$ to be a regular $n$-gon with two horizontal sides with opposite edges glued together by translation. If $n$ is odd, we define $\Pi_n$ to be formed from gluing two regular $n$-gons, one with a bottom horizontal side and one with a top horizontal side, and with each edge of the  first polygon glued to the parallel edge of the second polygon. For example, the surfaces $\Pi_3$ and $\Pi_4$, and $\Pi_6$ are the rhombic, square, and hexagonal tori respectively. Figure \ref{fig:pi} shows two higher-genus examples.

\begin{figure}
\begin{center}
\includegraphics[height=1.8in]{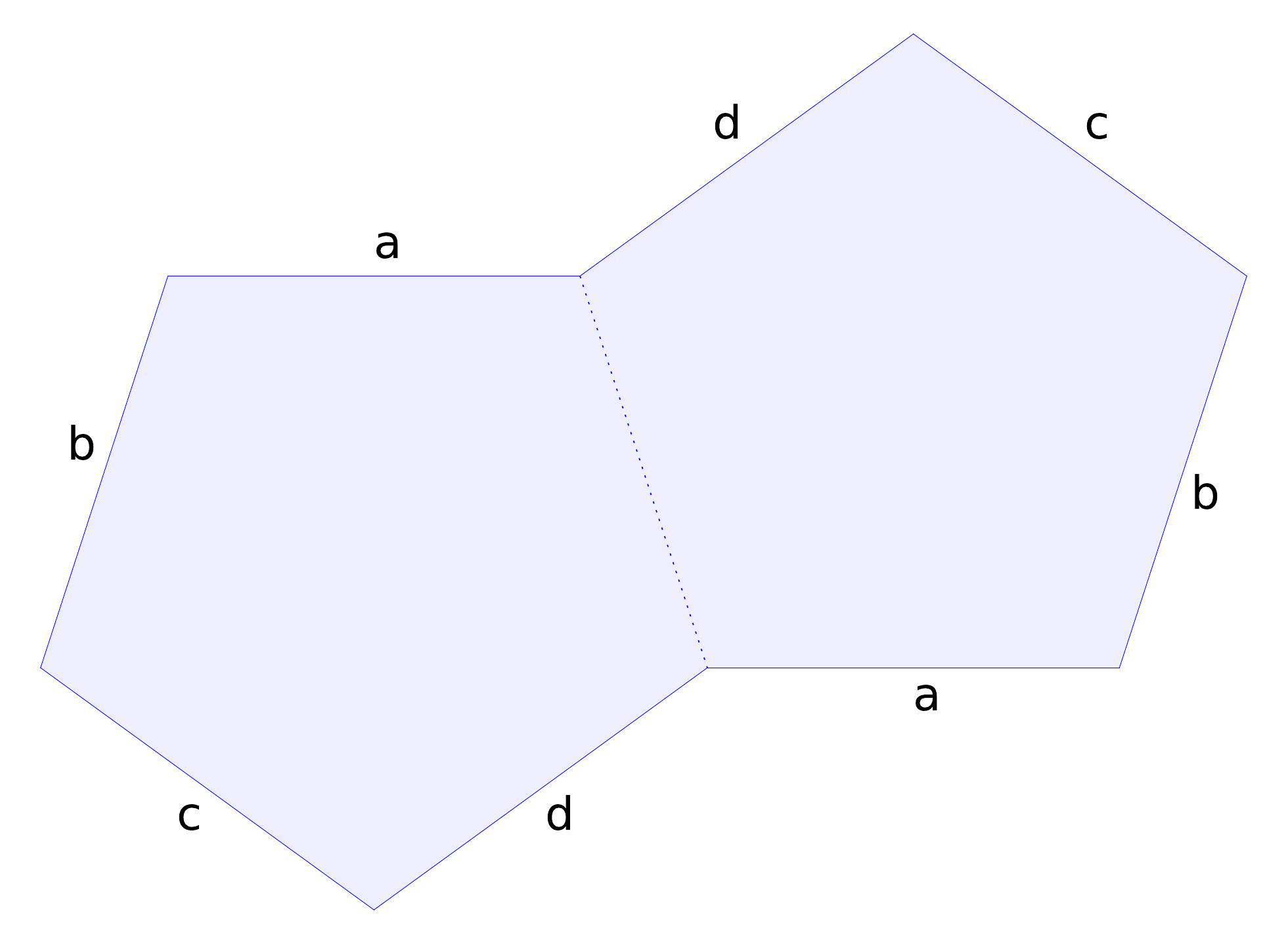}
\hspace{1in}
\includegraphics[height=1.8in]{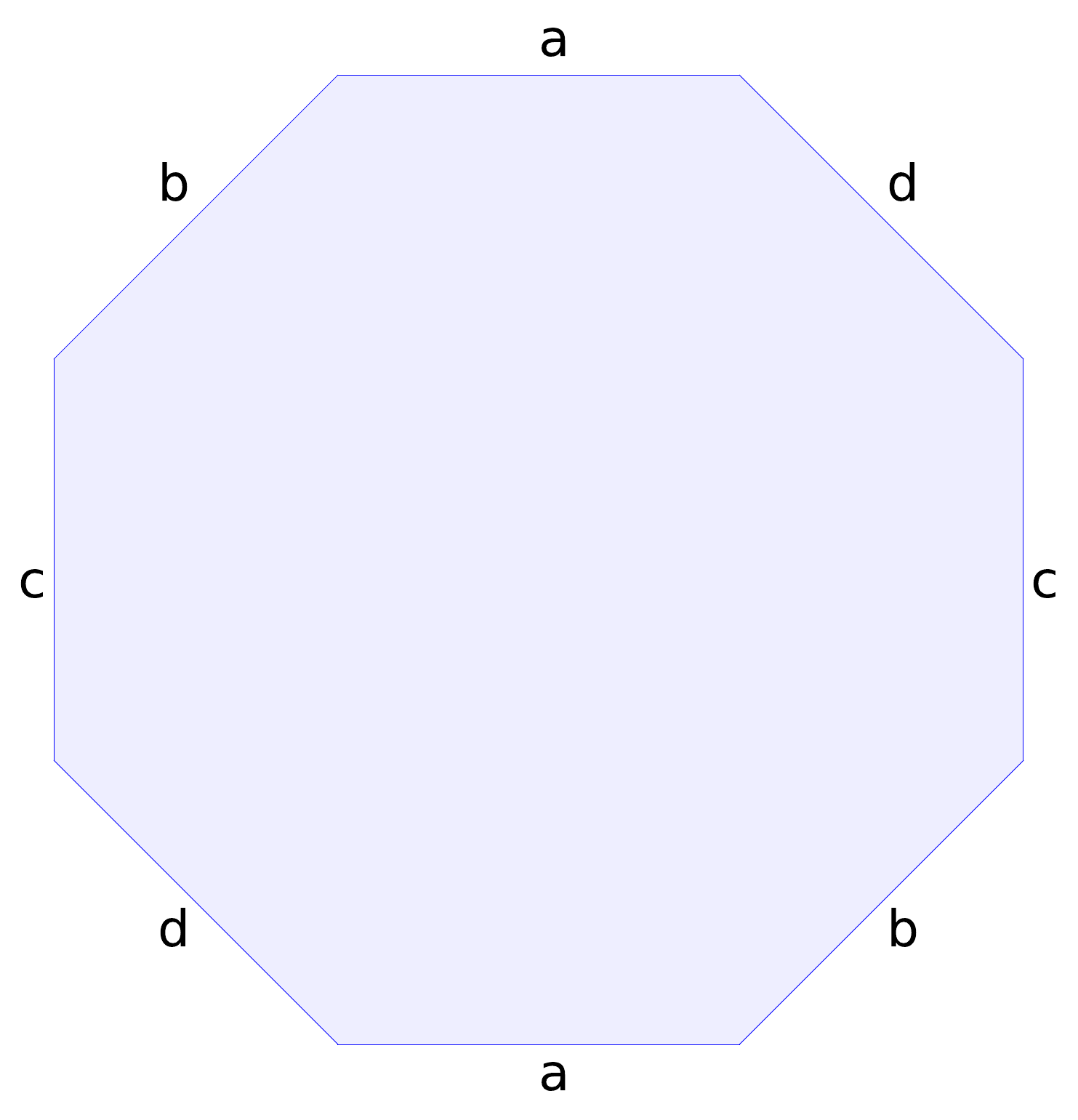}
\end{center}
\caption{The translation surfaces $\Pi_5$ and $\Pi_8$.}
\label{fig:pi}
\end{figure}

The surfaces $\Pi_n$ were first studied by Veech in \cite{Veech89}. It can be observed:

\begin{proposition}
For any $n \geq 3$, the surface $\Pi_n$ lies in the stratum $\cH(n-3)$ if $n$ is odd,
lies in $\cH(\frac{n-4}{2})$ if $n$ is a multiple of four, and lies in $\cH(\frac{n-6}{4},\frac{n-6}{4})$ if $n \equiv 2 \pmod{4}$. 
In particular, $\Pi_n$ is a torus if and only if $n \in \{3,4,6\}$.
\end{proposition}

\begin{proposition}
For any $n \geq 3$, the surface $\Pi_n$ is primitive. In particular,
$\Pi_n$ is arithmetic only when $n \in \{3,4,6\}$.
\end{proposition}

\begin{proposition}[{\cite[\S 5]{Veech89}}]
\label{prop:veech group}
If $n \geq 5$ is odd, then $V(\Pi_n)=\langle R,T \rangle$ and $V_\pm(\Pi_n)=\langle R,T,J\rangle$ where 
$$R = \left(\begin{array}{rr}
\cos \frac{\pi}{n} & -\sin \frac{\pi}{n} \\
\sin \frac{\pi}{n} & \cos \frac{\pi}{n}
\end{array}\right),
\quad
T = \left(\begin{array}{rr}
1 & 2 \cot \frac{\pi}{n} \\
0 & 1
\end{array}\right),
\quad
J = \left(\begin{array}{rr}
1 & 0 \\
0 & -1
\end{array}\right).
$$
If $n \geq 8$ is even, then $V(\Pi_n)=\langle R^2,T \rangle$ and $V_\pm(\Pi_n)=\langle R^2,T,J\rangle$ where $R$, $T$ and $J$ are as above.
\end{proposition}

These results follow for instance from work in \cite{Hooper13} where $\Pi_n$ appears as the surface $(Y_{2,n},\eta_{2,n})$ when $n$ is odd and as $(Y^e_{2,n},\eta^e_{2,n})$ when $n$ is even. 

The importance of $\Pi_n$ to us is that any translation surface with a regular $n$-gon decomposition is a cover of $\Pi_n$:

\begin{proposition}
\label{prop: natural covering}
Fix $n \geq 3$. If $S$ is a translation surface with a regular $n$-gon decomposition $\cP$, then there is a translation covering $\pi:S \to \Pi_n$ whose derivative is a dilation of the plane so that for each $P \in \cP$ the image $\pi(P)$ lies in the $n$-gon decomposition of $\Pi_n$. 
\end{proposition}

\begin{proof}
The map $\pi$ can be obtained by analytically continuing a dilation carrying some $P \in \cP$ to 
a regular polygon in $\Pi_n$. 
\end{proof}

\subsection{Platonic surfaces and unfoldings}

Fix a connected Euclidean cone surface with a regular $n$-gon decomposition $\cP$. 
A {\em flag} of $D$ is a triple $(P, e, v)$ where $P \in \cP$, $e$ is an edge of $P$, and $v$ is an endpoint of $e$.
We call $D$ a \emph{Platonic surface} if the group $\Isom(D)$
acts transitively on flags.

Note that if a Euclidean cone surface $D$ has a regular $n$-gon decomposition $\cP$,
and $\tilde D$ is a cover of $\cP$, then the collection $\tilde \cP$ of lifts of polygons in $\cP$
is a regular $n$-gon decomposition of $\tilde D$. We call $\tilde \cP$ the \emph{lifted $n$-gon decomposition}.

\begin{theorem}
\label{NormCov}
Let $D$ be a Platonic surface with regular $n$-gon decomposition $\cP$. Let $\tilde D$ be the unfolding of $D$ and let $\tilde \cP$ be the lifted $n$-gon decomposition. Then $\tilde D$ is also a Platonic surface,
and the covering $\pi_\Pi:\tilde D \to \Pi_n$ guaranteed by Proposition \ref{prop: natural covering} is regular.
\end{theorem}
\begin{proof}
First we will show that $\tilde D$ equipped with the $n$-gon decomposition $\tilde \cP$ is Platonic.
Let $(\tilde P_1, \tilde e_1, \tilde v_1)$ and $(\tilde P_2, \tilde e_2, \tilde v_2)$ be flags on $\tilde D$. Let
$(P_1, e_1, v_1)$ and $(P_2, e_2, v_2)$ be their projections to $D$ under the covering map $\pi_D:\tilde D \to D$. Since $D$ is Platonic, there is a $g \in \Isom(D)$ carrying $(P_1, e_1, v_1)$ to $(P_2, e_2, v_2)$. Because $g$ is an isometry, we have $PT \circ g_\ast=PT$ where $PT$ is the parallel transport map of \eqref{eq:parallel transport} and $g_\ast$ is the induced action of $g$ on $H_1(D^\ast, \bZ)$. Then $g$ lifts to an isometry $\tilde g \in \Isom(\tilde D)$ satisfying $$\pi_D \circ \tilde g(\tilde P_1, \tilde e_1, \tilde v_1)=(P_2,e_2, v_2).$$ Since $\pi_D$ is a regular cover, there is a deck transformation $\delta:\tilde D \to \tilde D$ of the covering $\pi_D$ so that $\delta \circ \tilde g(\tilde P_1, \tilde e_1, \tilde v_1)=(\tilde P_2, \tilde e_2, \tilde v_2)$. The isometry $\delta \circ \tilde g$ proves $\tilde D$ is Platonic.

To see the covering $\pi_\Pi:\tilde D \to \Pi_n$ is regular, select a flag $(P,e,v)$ in $\Pi_n$. Consider the collection of flags of $\tilde D$ which are preimages of $(P,e,v)$. From the previous paragraph, $\Isom(\tilde D)$ acts transitively on these flags. This is precisely the deck group of the covering $\pi_\Pi$. \end{proof}

\begin{definition}
The \emph{rotation group} of $D$ is the subgroup $\Gamma_{Rot}(D) := \Isom^+(D)$ of orientation preserving isometries of the Platonic surface.
\end{definition}


\section{Weierstrass points and closed saddle connections}
\label{sect:weierstrass}

In this section, we use Weierstrass points (fixed points of hyperelliptic involutions) to show that there are no closed saddle connections on the non-dodecahedron Platonic solids.

We first observe that the surfaces $\Pi_p$ are all hyperelliptic.
If $p$ is even, the hyperelliptic involution rotates about the center of the single $p$-gon making up $\Pi_p$. If $p$ is odd, the hyperelliptic involution swaps the two $p$-gons.
Recall that a {\em Weierstrass point} is a fixed point of the hyperelliptic involution.

Let $D$ be a Platonic surface formed from $p$-gons, and let $\tilde D$ be the unfolding of $D$. Then we have coverings $\pi_D:\tilde D \to D$
and $\pi_\Pi:\tilde D \to \Pi_p$.  (See Diagram \ref{eq:diagram}.)

We call a point on $x \in D$ a {\em virtual-Weierstrass point} if 
$x$ lifts (via $\pi_D^{-1}$) to a point $\tilde x \in \tilde D$ whose image is a Weierstrass point
of $\Pi_p$. The vertices of $D$ are virtual-Weierstrass points
unless $p \equiv 2 \pmod{4}$. The midpoints of edges are always regular virtual-Weierstrass points, while the centers of polygons are regular virtual-Weierstrass points if and only if $p$ is even.

We call a saddle connection $\sigma$ in $D$ a {\em virtual-Weierstrass saddle connection} if
$\sigma$ lifts (via $\pi_D^{-1}$) to a saddle connection $\tilde \sigma$ in $\tilde D$ whose image (under $\pi_{\Pi}$) in $\Pi_p$ is fixed by the hyperelliptic involution. Observe that
a saddle connection $\sigma$ in $D$ is a virtual-Weierstrass saddle connection
if and only if the midpoint of $\sigma$ is a regular virtual-Weierstrass point.

\begin{proposition}
\label{prop:virtual-Weierstrass}
Suppose $D$ is Platonic surface so that no involution in the rotation group  preserves both a vertex and a regular virtual-Weierstrass point. Then $D$ has no closed virtual-Weierstrass saddle connections.
\end{proposition}
\begin{proof}
Suppose to the contrary that $D$ had a closed virtual-Weierstrass saddle connection $\sigma$. The midpoint of the saddle connection is a regular virtual-Weierstrass point  $x$. Consider the involution $\iota:D \to D$ in the rotation group fixing $x$, which rotates about $x$ and preserves $\sigma$, reversing its orientation. Then $\iota$ also fixes the singularity serving as (both) the endpoints of $\sigma$. By assumption, no such involution exists.
\end{proof}

\begin{theorem}
\label{thm:virtual-Weierstrass}
Let $D$ be a Platonic surface constructed from $p$-gons such that there are an odd number of $p$-gons meeting at each vertex. Then $D$ has no closed virtual-Weierstrass saddle connections.
\end{theorem}
\begin{proof}
Since $D$ has no involutions in the rotation group which preserve a vertex, 
Proposition \ref{prop:virtual-Weierstrass} applies.
\end{proof}


\begin{proposition}
\label{prop:all vW}
If $D$ is a Platonic surface constructed from triangles or squares, then every saddle connection on $D$ is virtual-Weierstrass.
\end{proposition}
\begin{proof}
This is because the midpoint of any saddle connection on $\Pi_3$ and on $\Pi_4$ is a Weierstrass point. To see this observe that both $\Pi_3$ and $\Pi_4$ are flat tori with a single marked point, which is fixed by the hyperelliptic involution. A saddle connection is thus a closed flat geodesic through a fixed point and is therefore fixed by the hyperelliptic involution.
\end{proof}

\begin{corollary}
\label{cor:no closed}
Let $D$ be a Platonic surface constructed from triangles or squares so that there are an odd number of faces meeting at each vertex. Then, $D$ has no closed saddle connections.
\end{corollary}
\begin{proof}
Combine Theorem \ref{thm:virtual-Weierstrass} with Proposition \ref{prop:all vW}.
\end{proof}

\begin{corollary}
\label{cor:PlatSolidsBlocking}
The tetrahedron, cube, octahedron and icosahedron have no closed saddle connections.
\end{corollary}

This corollary was obtained for the tetrahedron and the cube by other methods in \cite{DDTY},
and in the remaining cases above in \cite{FuchsArnold}. This result can also be obtained (as in \cite{FuchsArnold}) 
from considering tilings of the plane by triangles and squares, which makes sense in our context because these tilings are the universal covers of $\Pi_3$ and $\Pi_4$, respectively.

\begin{proof}
That the tetrahedron, cube and icosahedron have no closed saddle connections
follows directly from Corollary \ref{cor:no closed}. That the octahedron has no closed saddle connections follows from Proposition \ref{prop:virtual-Weierstrass} since
an involution in the rotation group of the octahedron which fixes a vertex has a fixed point set consisting only of two antipodal vertices.
\end{proof}

\noindent Theorem \ref{thm:virtual-Weierstrass} tells us that the dodecahedron has no closed virtual-Weierstrass saddle connections. We use this in \S\ref{sect:DodecSearchOutput}.

\section{Unfoldings of the Platonic solids}
\label{sect:unfoldings}

The covering $\pi_{\Pi}: \tilde D \rightarrow \Pi_p$ from the previous section can be encoded in terms of elements of the permutation group $S_m$, where $m$ is the degree of $\pi_{\Pi}$.  Consider a basis $\{x_0, \ldots, x_{p-1}\}$ or $\{x_0, \ldots, x_{p/2-1}\}$, depending on the parity of $p$, for the fundamental group $\pi_1(\Pi_p^*, p_0)$.  Number the sheets of the covering space $\tilde D$ by elements of $\{0, \ldots, m-1\}$.  Then each element of the fundamental group induces a permutation in $S_m$ on the fiber over $p_0$, which is the monodromy representation associated to the covering $\pi_{\Pi}$.  In this section we exhibit Sage code that generates these permutations.

While it will be necessary to produce the permutations for the octahedron, cube, icosahedron, and dodecahedron, we will only produce the permutations for the latter three and only include the calculations for the octahedron in the auxiliary Sage notebook.  In fact, the case of the octahedron can be done entirely by hand.

\subsection{Labeling the Polygons}

In this section we construct a ``coordinate system'' adapted to a presentation of the unfolded surface $\tilde D$ to identify every polygon in $\tilde D$.  The unfolded surface can be presented by taking a net of the Platonic surface, and taking $k$ copies of it with each one rotated by $2\pi/k$ from the previous copy.  Our so-called coordinate system will assign a pair $(\verb|sheet|, \verb|poly|)$ to every copy of a $p$-gon in the unfolded surface $\tilde D$.  To do this we orient the presentation of $\tilde D$ so that every polygon contains a horizontal edge.  Label the horizontal edge, or the lower horizontal edge in the case of even-sided polygons, $0$, and continue counter-clockwise around the polygon to $p-1$.  Then the surface is (over)-determined by a function that associates to each polygon a list of $p$ tuples specifying which polygon is incident with edge $j$.  For each solid, we cunstruct a function that takes a polygon with the coordinates $(\verb|sheet|, \verb|poly|)$ as an input and outputs a list of $p$ tuples specifying the coordinates of each polygon to which the input polygon is adjacent.  The index $i$ element of the list will correspond to the polygon incident with edge $i$ with the convention above.

The cube is a valuable example because it is the only classical Platonic solid constructed from a regular polygon with an even number of sides, i.e. the square.  In a sequel~\cite{Aulicino} to this paper, arbitrary regular polygons with an even number of sides are considered and the procedure below will generalize to these cases.

\begin{figure}
\begin{center}
\includegraphics[height=1.4in]{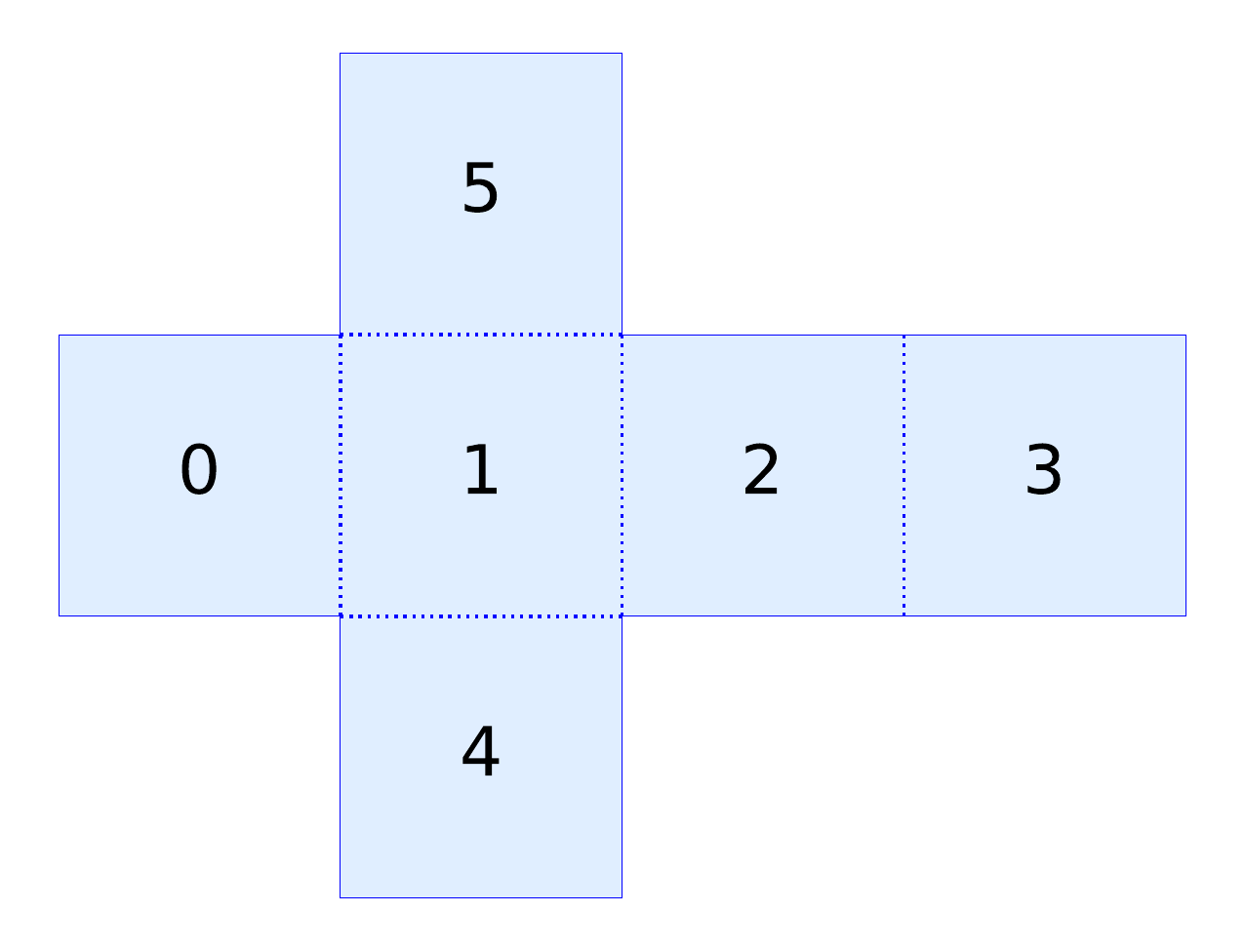}
\end{center}
\caption{The cube with the labeling of squares used in the function \texttt{build\_adj\_cube}.}
\label{fig:cube_net}
\end{figure}

\begin{figure}
\centerline{\includegraphics[width=4in]{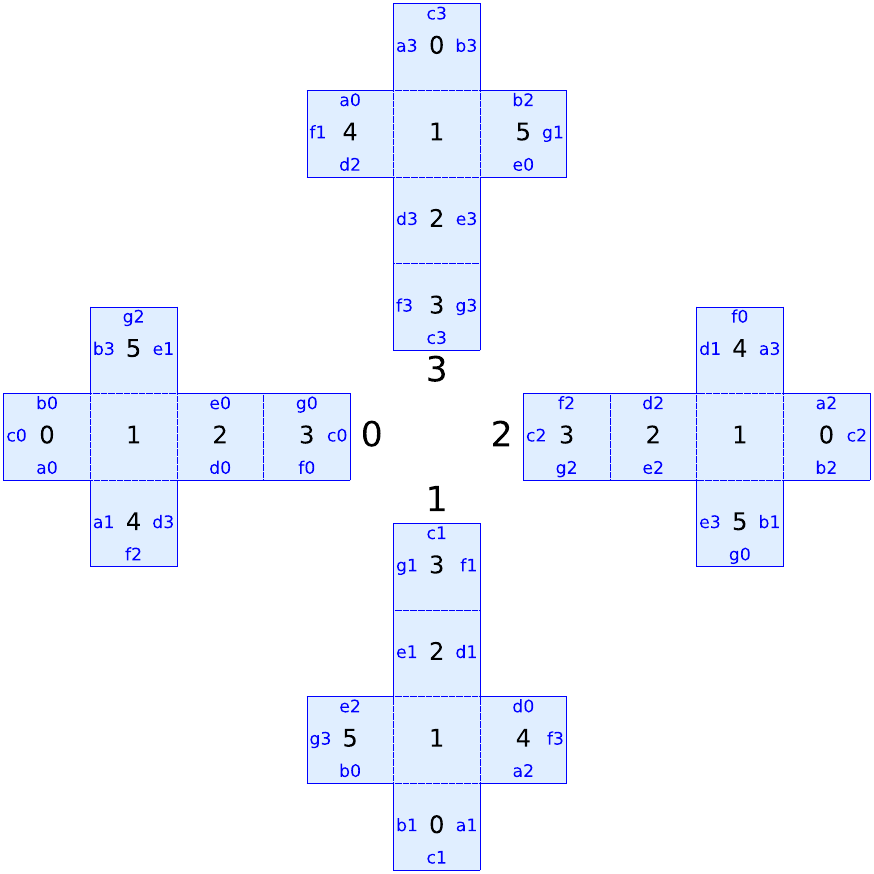}}
\caption{The unfolding of the cube following the conventions of Figure \ref{fig:unfolding}.}
\label{fig:unfolded cube}
\end{figure}

\begin{example}[The cube]
The following code creates a list of lists that outputs the adjacencies of the input square as specified by the tuple \verb|(sheet, square)|.  The variable \verb|square| refers to the labeling of the squares depicted in Figure \ref{fig:cube_net}.  The lower horizontal edge has index zero in the resulting list of four pairs and the adjacencies continue counter-clockwise. The reader may wish to refer to Figure \ref{fig:unfolded cube} which shows the unfolded cube.

\begin{verbatim}
def build_adj_cube(sheet, square):
    i = sheet;
    cube_adj_base = 6*[None]
    cube_adj_base[0] = [(i-1,4), (i,1), (i+1,5), (i,3)]
    cube_adj_base[1] = [(i,4), (i,2), (i,5), (i,0)]
    cube_adj_base[2] = [(i+1,4), (i,3), (i-1,5), (i,1)]
    cube_adj_base[3] = [(i+2,4), (i,0), (i+2,5), (i,2)]
    cube_adj_base[4] = [(i+2,3), (i-1,2), (i,1), (i+1,0)]
    cube_adj_base[5] = [(i,1), (i+1,2), (i+2,3), (i+3,0)]
    prelim_adj = [cube_adj_base[square%6][(k-i)%4] for k in range(4)]
    return [[item[0]%4, item[1]%6] for item in prelim_adj] 
\end{verbatim}

All lines can be checked by inspection with the exception of the one non-trivial line defining \verb|prelim_adj|.  By definition of \verb|cube_adj_base|, the index \verb|[square%6]|
extracts the appropriate square.  However, for square \verb|square| on sheet \verb|i|, the lower horizontal edge is given by $-\verb|i|$.  Each successive edge counter-clockwise is given by adding \verb|k| to this number.  Hence, the formula follows.
\end{example}

\begin{figure}
\begin{center}
\includegraphics[height=1.4in-]{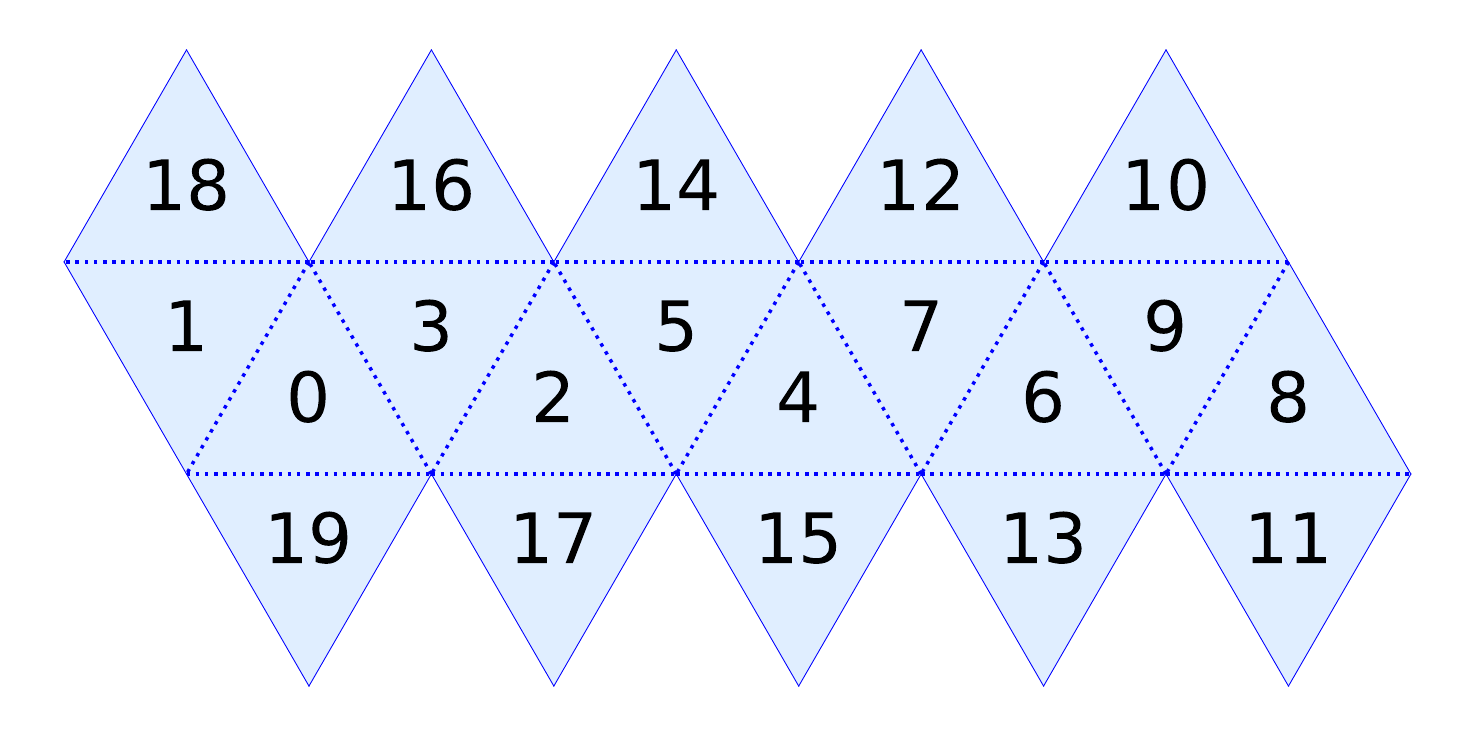}
\end{center}
\caption{The icosahedron with the labeling of triangles used in the function \texttt{build\_adj\_icosa}.}
\label{fig:icosahedron_net}
\end{figure}

\begin{example}[The icosahedron]
As in the previous example, the following code defines the adjacencies among all twenty triangles.  The variable \verb|triangle| refers to the labeling of the triangles depicted in Figure \ref{fig:icosahedron_net}.
\begin{verbatim}
def build_adj_icosa(sheet, triangle):
    i = sheet;
    icos_adj_base = 20*[None]
    icos_adj_base[0] = [(i,19),(i,3),(i,1)]
    icos_adj_base[1] = [(i,18),(i,8),(i,0)]
    icos_adj_base[2] = [(i,17),(i,5),(i,3)]
    icos_adj_base[3] = [(i,16),(i,0),(i,2)]
    icos_adj_base[4] = [(i,15),(i,7),(i,5)]
    icos_adj_base[5] = [(i,14),(i,2),(i,4)]
    icos_adj_base[6] = [(i,13),(i,9),(i,7)]
    icos_adj_base[7] = [(i,12),(i,4),(i,6)]
    icos_adj_base[8] = [(i,11),(i,1),(i,9)]
    icos_adj_base[9] = [(i,10),(i,6),(i,8)]
    icos_adj_base[10] = [(i,9),(i-1,18),(i+1,12)]
    icos_adj_base[11] = [(i,8),(i-1,13),(i+1,19)]
    icos_adj_base[12] = [(i,7),(i-1,10),(i+1,14)]
    icos_adj_base[13] = [(i,6),(i-1,15),(i+1,11)]
    icos_adj_base[14] = [(i,5),(i-1,12),(i+1,16)]
    icos_adj_base[15] = [(i,4),(i-1,17),(i+1,13)]
    icos_adj_base[16] = [(i,3),(i-1,14),(i+1,18)]
    icos_adj_base[17] = [(i,2),(i-1,19),(i+1,15)]
    icos_adj_base[18] = [(i,1),(i-1,16),(i+1,10)]
    icos_adj_base[19] = [(i,0),(i-1,11),(i+1,17)]
    prelim_adj = [icos_adj_base[triangle%20][(k-i)%3] for k in range(3)]
    return [[item[0]%6, item[1]%20] for item in prelim_adj]
\end{verbatim}
As before all lines of this code are easily checked with the exception of the term \verb|[(k-i)%3]|
in the definition of \verb|prelim_adj|.  We see that in the six rotations of the net of the icosahedron, each one rotates the net by $\pi/3$.  Therefore, the horizontal edge returns to the horizontal direction after three rotations, even though the triangle is inverted.  After one rotation of the net the edge with index $1$ is rotated to the (top) horizontal and therefore should have index $0$, and this continues cyclically.  Hence $\verb|k|-\verb|i| \mod 3$ yields the correct formula.
\end{example}

\begin{figure}
\begin{center}
\includegraphics[height=1.4in]{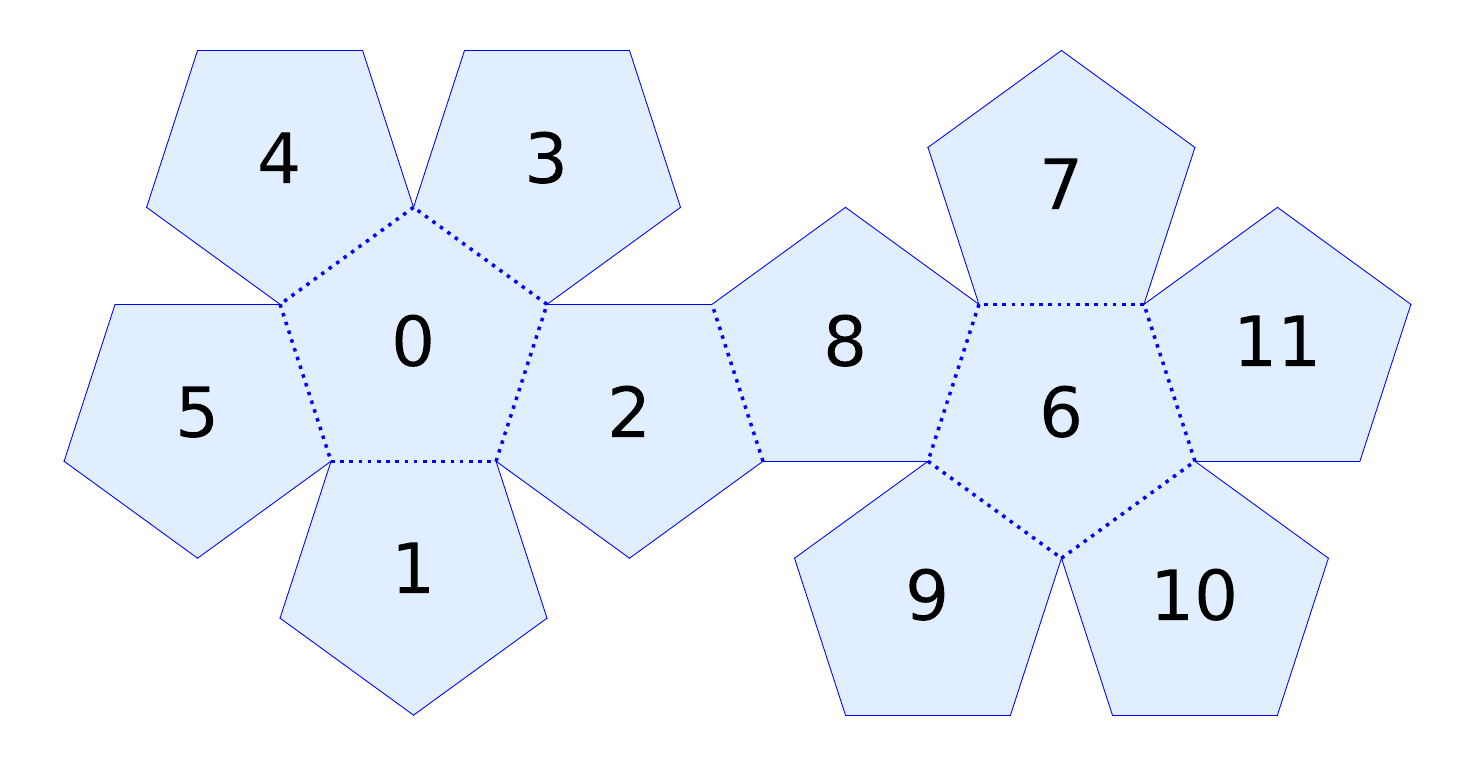}
\end{center}
\caption{A net of the dodecahedron with the labeling of the pentagons used in the function \texttt{build\_adj\_dodec}.}
\label{fig:dodecahedron}
\end{figure}

\begin{example}[The dodecahedron]
\label{ex:dodecahedron}
This is very similar to the previous example.  The variable \verb|pent| refers to the labeling of the pentagons depicted in Figure \ref{fig:dodecahedron}.  In this case, Figure \ref{fig:unfolding} includes a labeling of each of the sheets and each of the pentagons in the unfolded dodecahedron.  We remark that the labeling of the sheets can be cyclically permuted without changing the function below.
\begin{verbatim}
def build_adj_dodec(sheet, pent):
    i = sheet;
    dodec_adj_base = 12*[None]
    dodec_adj_base[0] = [(i,1),(i,2),(i,3),(i,4),(i,5)]
    dodec_adj_base[1] = [(i,0),(i+1,5),(i+4,10),(i-4,9),(i-1,2)]
    dodec_adj_base[2] = [(i-1,3),(i,0),(i+1,1),(i-2,9),(i,8)]
    dodec_adj_base[3] = [(i+4,7),(i-1,4),(i,0),(i+1,2),(i+2,8)]
    dodec_adj_base[4] = [(i-4,7),(i-2,11),(i-1,5),(i,0),(i+1,3)]
    dodec_adj_base[5] = [(i+1,4),(i,11),(i+2,10),(i-1,1),(i,0)]
    dodec_adj_base[6] = [(i,7),(i,8),(i,9),(i,10),(i,11)]
    dodec_adj_base[7] = [(i,6),(i+1,11),(i+4,4),(i-4,3),(i-1,8)]
    dodec_adj_base[8] = [(i-1,9),(i,6),(i+1,7),(i-2,3),(i,2)]
    dodec_adj_base[9] = [(i+4,1),(i-1,10),(i,6),(i+1,8),(i+2,2)]
    dodec_adj_base[10] = [(i-4,1),(i-2,5),(i-1,11),(i,6),(i+1,9)]
    dodec_adj_base[11] = [(i+1,10),(i,5),(i+2,4),(i-1,7),(i,6)]
    prelim_adj = [dodec_adj_base[pent%12][(k+2*i)%5] for k in range(5)]
    return [[item[0]%10, item[1]%12] for item in prelim_adj]
\end{verbatim}
Again we focus on justifying the index \verb|[(k+2*i)%5]|.  
The situation is slightly different here because one rotation by $2\pi/10$ moves the edge with index $2$ into the horizontal.  Hence, the indexing must be shifted by \verb|2*i| to compensate for this change because the sheet index is exactly equal to the number of rotations applied to sheet $0$.  See Figure \ref{fig:unfolding}.
\end{example}

\begin{remark}
Once again, we emphasize that the convenience of this coordinate system is that in every function above the second term is static because it is independent of the unfolding.  This greatly simplifies the construction of these functions.
\end{remark}

\subsection{Labeling the Sheets of $\pi_{\Pi}$}
\label{sect:labeling sheets}
We return to the covering $\pi_{\Pi}: \tilde D \rightarrow \Pi_p$.  We label the sheets of the covering by a set of integers equal to the degree of the covering.  For the cube this is straightforward.

\begin{verbatim}
def squares():
    return list(itertools.product(*[range(4), range(6)]))
\end{verbatim}

\noindent
This code generates a list of all possible square coordinates in the unfolded cube, which using the coordinates from the previous section, is the Cartesian product $\{0,1,2,3\} \times \{0,1,\ldots,5\}$.  The result is a Python list.  The index of the coordinate in the list implicitly numbers the sheets.

On the other hand, the icosahedron covers $\Pi_3$, which is a rhombus, and the dodecahedron covers $\Pi_5$, which is a double pentagon.  In these cases, each polygon has a unique horizontal edge following the convention above.  Therefore, it is essential to identify the subset of polygons with a horizontal edge on their bottom.  This is called the \emph{top} polygon here.

\begin{verbatim}
def double_triangle_top():
    odd_array = [range(1,6,2), range(1,20,2)]
    even_array = [range(0,6,2), range(0,20,2)]
    return list(itertools.product(*odd_array)) + list(itertools.product(*even_array))
\end{verbatim}

\noindent This code generates a Python list containing the elements of
$$\left(\{1,3,5\} \times \{1,3,...,19\}\right) \cup \left(\{0,2,4\} \times \{0,2,...,18\}\right).$$

\begin{verbatim}
def double_pent_top():
    odd_array = [range(1,10,2), range(7,12) + [0]]
    even_array = [range(0,9,2), range(1,7)]
    return list(itertools.product(*odd_array)) + list(itertools.product(*even_array))
\end{verbatim}

Finally, the other half of $\Pi_p$ must be identified using the functions of the previous section to get a pair of polygons, and the implicit numbering of this list is the numbering of the sheets of the covering $\pi_{\Pi}$.

\begin{verbatim}
def double_triangles():
    return [[list(top),build_adj_icosa(top[0],top[1])[0]] for top in double_triangle_top()]
\end{verbatim}

\begin{verbatim}
def double_pent():
    return [[top,build_adj_dodec(top[0],top[1])[4]] for top in double_pent_top()]
\end{verbatim}

\subsection{Constructing the Permutations}
\label{sect:ConstructPerms}

We now have the necessary ingredients to construct the permutations corresponding to the images of the generators of the fundamental group of $\Pi_n$ under the monodromy representation. (Figure \ref{DoublePentConvsFundGrp} has the generators for $\pi_1(\Pi_5,p_0)$.)
Recall that the numbers in each permutation correspond to the sheet with that index in the list of sheets generated in the previous section.

We begin with the cube.  For the cube, $\Pi_4$ is the square torus and 
we take the generators of $\pi_1(\Pi_4,p_0)$ to be a horizontal closed loop 
and a vertical closed loop.  The list of all squares is stored as \verb|sq_list|.  The algorithm below takes the first square \verb|i| from the list and determines the index of the square in \verb|sq_list| sharing edge \verb|abcd| with it.  It stores all of the sheets that have been accounted for in the list \verb|total| and successively chooses the smallest integer not in \verb|total| on which to build the next permutation.  It proceeds until all sheets have been accounted for.

\begin{verbatim}
def perm_sq(abcd):
    sq_list = squares()
    total = []
    i = 0
    perm_a_sub = []
    perm_a = []
    while len(total) < 20:
        total += perm_a_sub
        total.sort()
        if len(total) != 0:
            i_list = [j for j in range(len(total)) if j != total[j]]
            if i_list == []:
                i = len(total)
            else:
                i = i_list[0]
        perm_a_sub = []
        while i not in perm_a_sub:
            perm_a_sub += [i]
            i = sq_list.index(tuple(build_adj_cube(sq_list[i][0], sq_list[i][1])[abcd]))
        perm_a.append(tuple(perm_a_sub))
    return perm_a
\end{verbatim}

The icosahedron and dodecahedron can be addressed in a single function with only one ad hoc modification: the function that determines the adjacent polygon is completely different for each of these solids.  The \verb|plat_solid| variable is introduced to distinguish them.  The other difference from the previous case is that there is a pair of polygons to address and so the sheet that sheet \verb|i| is adjacent to depends on determining the element in \verb|bot_list| that the element in \verb|top_list| is adjacent to.

\begin{verbatim}
def perm_odd(plat_solid, abcd, top_list, double_list, upper_limit):
    bot_list = [tri[1] for tri in double_list]
    total = []
    i = 0
    perm_a_sub = []
    perm_a = []
    while len(total) < upper_limit:
        total += perm_a_sub
        total.sort()
        if len(total) != 0:
            i_list = [j for j in range(len(total)) if j != total[j]]
            if i_list == []:
                i = len(total)
            else:
                i = i_list[0]
        perm_a_sub = []
        while i not in perm_a_sub:
            perm_a_sub += [i]
            if plat_solid == 5:
                i = bot_list.index(build_adj_dodec(top_list[i][0],top_list[i][1])[abcd])
            elif plat_solid == 4:
                i = bot_list.index(build_adj_icosa(top_list[i][0],top_list[i][1])[abcd])
        perm_a.append(tuple(perm_a_sub))
    return perm_a
\end{verbatim}

\section{The Arithmetic Platonic Solids}\label{sect:nondodec}

This section concerns the tetrahedron, octahedron, cube, and icosahedron which we refer to collectively as the \emph{Arithmetic Platonic Solids}, since they are covers of tori branched over one point.  We begin by giving a computer assisted proof of Corollary \ref{cor:PlatSolidsBlocking}.  This is done to show how one could approach the problem computationally in general. Finally we use the data computed to describe the Teichm\"uller curves of these unfoldings.  The proof is effectively the same as that of~\cite{FuchsDedicata} (icosahedron and cube) and \cite{FuchsFuchs} (tetrahedron and octahedron).

We observe that all of these Platonic solids have either square faces, in which case the unfolding is square-tiled and covers $\Pi_4$, or they have triangle faces, in which case they cover the double triangle $\Pi_3$.  Let
$$M = \left(\begin{array}{cc} 1 & \frac{-1}{\sqrt{3}} \\ 0 & \frac{2}{\sqrt{3}} \\ \end{array} \right).$$
Observe that $M(\Pi_3) = \Pi_4$.  Therefore, if $S$ is triangle-tiled, then $M(S)$ is square-tiled.  Throughout this section, we regard all triangle-tiled Platonic surfaces as squ 	are-tiled surfaces.\footnote{As observed in the introduction, this observation was stated in the final paragraph of \cite{FK36} using the terminology of tilings of the plane.}

Recall that for any square-tiled surface, labeling the squares and writing two permutations corresponding to the horizontal and vertical direction suffice to completely determine the surface. The package \verb|surface_dynamics| allows you to define these objects (``origamis'') in SageMath \cite{SurfaceDynamics}. The permutations specifying the unfoldings of the octahedron, cube, and icosahedron are given below.

\begin{example}[The octahedron] The flat metric on the surface of an octahedron is given by a $3$-differential, so the natural cover is tiled by $8 \times 3 = 24$ triangles and $8 \times 3/2 = 12$ rhombi, so our permutations are in $S_{12}$. Our indices start at $0$.
\begin{verbatim}
oct_perm_0 = [(0, 7, 6), (1, 3, 4), (2, 11, 5), (8, 10, 9)]
oct_perm_1 = [(0, 11, 9), (1, 7, 10), (2, 3, 8), (4, 5, 6)]
octahedron = Origami(oct_perm_0, oct_perm_1)
\end{verbatim}
The monodromy group generated by these permutations is of order $12$ and isomorphic to $A_4$.
\end{example}

Using the functions defined in the previous section, we construct the unfolding of the cube as a square-tiled surface as follows.

\begin{example}[The cube] The flat metric on the surface of a cube is given by a $4$-differential, so the natural cover is tiled by $6 \times 4 = 24$ squares so our permutations are in $S_{24}$. Our indices start at $0$.
\begin{verbatim}
cube_perm_0 = perm_sq(0)
cube_perm_1 = perm_sq(1)
cube = Origami(cube_perm_0, cube_perm_1)
\end{verbatim}

For reference, the values of these permutations are as follows

\begin{verbatim}
cube_perm_0 = [(0, 22, 14, 11), (1, 4, 15, 5), (2, 10, 12, 23),
(3, 16, 13, 17), (6, 9, 8, 7), (18, 19, 20, 21)]
cube_perm_1 = [(0, 1, 2, 3), (4, 20, 17, 6), (5, 8, 16, 18), 
(7, 10, 21, 11), (9, 22, 19, 23), (12, 15, 14, 13)]
\end{verbatim}
The group generated by these permutations is of order $24$ and isomorphic to $S_4$.
\end{example}

\begin{example}[The icosahedron]
The flat metric on the surface of an icosahedron is given by a $6$-differential, so the natural cover is tiled by $20 \times 6 = 120$ triangles and $20 \times 6/2 = 60$ rhombi, so our permutations are in $S_{20}$. Our indices start at $0$.
\begin{verbatim}
icos_perm_1 = perm_odd(4, 1, double_triangle_top(), double_triangles())
icos_perm_2 = perm_odd(4, 2, double_triangle_top(), double_triangles())
icosahedron = Origami(icos_perm_1, icos_perm_2)
\end{verbatim}

The permutations above are as follows.

\begin{verbatim}
icos_perm_1 = [(0, 45, 13, 17, 9), (1, 49, 14, 16, 8), 
(2, 48, 10, 15, 7), (3, 47, 11, 19, 6), (4, 46, 12, 18, 5), 
(20, 21, 22, 23, 24), (25, 30, 38, 56, 53), (26, 34, 39, 57, 52), 
(27, 33, 35, 58, 51), (28, 32, 36, 59, 50), (29, 31, 37, 55, 54), 
(40, 44, 43, 42, 41)]
icos_perm_2 = [(0, 4, 3, 2, 1), (5, 33, 36, 48, 40), 
(6, 32, 37, 49, 44), (7, 31, 38, 45, 43), (8, 30, 39, 46, 42), 
(9, 34, 35, 47, 41), (10, 19, 27, 23, 55), (11, 18, 26, 24, 59), 
(12, 17, 25, 20, 58), (13, 16, 29, 21, 57), (14, 15, 28, 22, 56), 
(50, 51, 52, 53, 54)]
\end{verbatim}
The group generated by these permutations is of order $60$ and isomorphic to $A_5$.

\end{example}

The close relationship between the monodromy groups and the rotation groups of the Platonic solids is not coincidental (the monodromy group of the dodecahederon covering, generated by 4 permutations, is also $A_5$) and will be further explored in~\cite{Aulicino}.  


\subsection{The Topology of the Teichm\"uller Curves}\label{sect:topology} Here, we describe the Teichm\"uller curves associated to the octahedron, cube, and icosahedron. Note that for the tetrahedron the cover is a torus, so the Veech group is $\SL(2, \bZ)$, and the Teichm\"uller curve is the modular curve, which is of genus 0, has one cusp of width $1$ at infinity, and has two orbifold points of order $2$ and $3$ respectively. The \emph{geometric cusp width} of a cusp of a Fuchsian group is the length of the longest embedded horocycle on the quotient surface homotopic to the cusp. We will be considering finite index subgroups $\Gamma'$ of a fixed Fuchsian group $\Gamma$ (either $\SL(2, \bZ)$ or the $\Delta(2, 5, \infty)$), and the \emph{normalized} cusp width of a cusp of $\Gamma'$ is the ratio of the geometric cusp width of $\Gamma'$ to that of $\Gamma$. This is always an integer. In the case of $\SL(2, \bZ)$, the cusp width of the only cusp at $\infty$ is $1$, so the geometric and normalized cusp widths for finite index subgroups are identical.

As above, we are using the Sage package \verb|surface_dynamics| which computes Veech groups of square-tiled surfaces using the monodromy, and allows us to compute cusps, cusp widths, genus, and index.

Since each of these computations only rely on the Veech group, we compute each of the relevant Veech groups at once.  Define the Veech groups with the following commands.

\begin{verbatim}
H_octa = octahedron.veech_group()
H_cube = cube.veech_group()
H_icos = icosahedron.veech_group()
\end{verbatim}

Let
$$\verb|H_plat| \in \{\verb|H_octa|, \verb|H_cube|, \verb|H_icos| \}.$$
For each of the three Platonic solids above we compute:
\begin{verbatim}
H_plat.index() #index in SL(2, Z)
H_plat.cusps() #list of cusps
[H_plat.cusp_width(cusp) for cusp in H_plat.cusps()]
H_plat.nu2() #Number of orbifold points of order 2
H_plat.nu3() #Number of orbifold points of order 3
H_plat.genus()
\end{verbatim}
For the convenience of the reader, we collect the output of this code in Table \ref{ArithPlatSolidsTeichCurveTable}\footnote{Since there are infinitely many choices for a fundamental domain of a discrete group, Sage can change its choice with each computation.  This will lead to different points in the cusp list from those listed in Column 3.  The cusps listed here represent one possible choice.}.
The widths of individual cusps in the Teichm\"uller curves appear as the numbers $n(p, q)$ in the work of D. Fuchs~\cite{FuchsDedicata} and D. Fuchs and K. Fuchs~\cite{FuchsFuchs}.

\begin{table}
  \centering
\begin{tabular}{l|c|c|c|c|c|c}
Platonic Solid & Index in $SL(2, \bZ)$ & Cusps & Cusp Widths & Order $2$ & Order $3$ & Genus \\
 \hline
Tetrahedron & $1$ & $\{\infty\}$ & $\{1\}$ &  $1$ & $1$ & $0$ \\
Octahedron & $4$ & $\{\infty, 1\}$ &$\{3, 1 \}$& $0$ & $1$ & $0$ \\
Cube & $9$ & $\{\infty, 1/2, 1\}$ &$\{4, 2, 3\}$ & $1$ & $0$ & $0$ \\
Icosahedron & $10$ & $\{\infty, 3/2, 2\}$ & $\{5, 2, 3\}$ & $0$ & $1$ & $0$ \\
\end{tabular}
\caption{Quantities associated to the Teichm\"uller curves of the arithmetic Platonic solids: Column two provides the index of the Veech group in $\SL(2,\bZ)$, Columns $5$ and $6$ provide number of orbifold points of orders $2$ and $3$, respectively.}
\label{ArithPlatSolidsTeichCurveTable}
\end{table}

\subsection{Blocking}
\label{sect:blocking}
Every vertex of a Platonic solid lifts to a unique zero on its unfolding.  Therefore it suffices to prove that there is no closed saddle connection on the associated translation surface. We first give the proof for the tetrahedron (\S\ref{sect:tetra}), as a model for the computer-assisted proof for the other three solids (\S\ref{sect:comp}). Finally, we discuss the relationship with more general blocking problems (\S\ref{sect:blocking cardinality}).

\subsubsection{The tetrahedron}\label{sect:tetra}

We state the proof for the tetrahedron without the aid of a computer.  The proof is well-known to experts on translation surfaces.

\begin{proof}[Proof for the Tetrahedron]
The Teichm\"uller curve of the torus with all $2$-torsion points marked has a single cusp.  Therefore, it suffices to observe that the horizontal direction on the square-torus does not contain a horizontal trajectory passing through exactly one $2$-torsion point.
\end{proof}

\subsubsection{Computer-assisted proof}\label{sect:comp}

This section relies only on readily available functions for square-tiled surfaces in the Sage package \verb|surface_dynamics|.  We use Sage to determine the number of cusps of each of the resulting Teichm\"uller curves.  Then we find distinct directions by inspection that correspond to each of these cusps.  Finally, we observe that none of these directions admit a saddle connection from a zero to itself.  In all cases below the proof is reduced to checking that on each cylinder in the three figures, there does not exist a vertex on the top and bottom of any individual cylinder with the same color.  For this reason it simplifies the verification of the claims to avoid choosing the horizontal direction in the proofs below.  We refer the reader to the auxiliary file \verb|code_from_the_article.ipynb|
to verify the claims of this section.

\begin{proof}[Proof for the Octahedron]
We compute the permutations \verb|oct_perm_1| and \verb|oct_perm_0*oct_perm_1| to give the permutations of the cylinders in the vertical and slope one direction, respectively.  Since there are two distinct cusps by Table \ref{ArithPlatSolidsTeichCurveTable} and every vertical cylinder is given by a permutation that is a product of $3$-cycles, while in the directions with slope one, every cylinder is given by a permutation that is a product of $2$-cycles, we see that these two directions correspond to distinct periodic directions.  By inspection of the colors on the top and bottom of every cylinder in Figure \ref{Octahedron3CoverFig}, we see that none of them connect a zero to itself.  (Simply observe that no zero appears on the top and bottom of every cylinder by comparing the colors on each boundary of the cylinder.)  Therefore, there are no trajectories from a vertex to itself on the octahedron.
\end{proof}

\begin{figure}[ht]
\centering
\includegraphics[scale=.45]{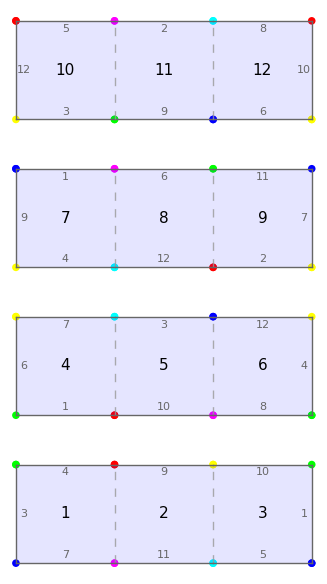}
\caption{The canonical $3$-cover of the octahedron after acting by an element of $GL(2,\bR)$.}
\label{Octahedron3CoverFig}
\end{figure}

\begin{proof}[Proof for the Cube]
The corresponding Teichm\"uller curve has three cusps by Table \ref{ArithPlatSolidsTeichCurveTable}.  The cylinders in the vertical direction are permutations with $4$-cycles by computing \verb|cube_perm_1|, in the direction with slope one they are permutations with $3$-cycles by computing \verb|cube_perm_0*cube_perm_1|, and in the direction with slope $1/2$, they are permutations with $2$-cycles by computing \verb|cube_perm_0*cube_perm_0*cube_perm_1|.  Therefore, these directions account for all three cusps.  Again, by comparing colors on each cylinder in Figure \ref{Cube4CoverFig}, it is clear that there is no closed saddle connection.
\end{proof}

\begin{figure}[ht]
\centering
\includegraphics[scale=.5]{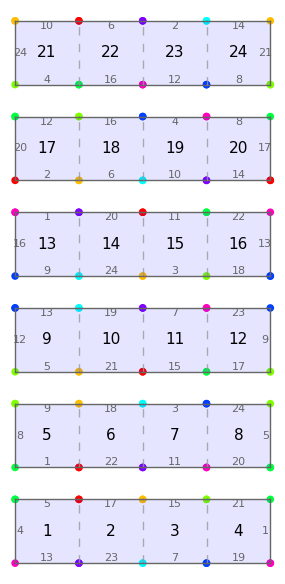}
\caption{The canonical $4$-cover of the cube.}
\label{Cube4CoverFig}
\end{figure}

\begin{proof}[Proof for the Icosahedron]
The corresponding Teichm\"uller curve has three cusps by Table \ref{ArithPlatSolidsTeichCurveTable}.  Computing the three permutations
$$\{\verb|icos_perm_1|, \verb|icos_perm_0*icos_perm_0*icos_perm_0*icos_perm_1|,$$ $$\verb|icos_perm_0*icos_perm_0*icos_perm_1|\},$$
which correspond to directions with slope vertical, $1/3$, and $1/2$, and yield cylinders given by permutations that are products of $5$-cycles, $3$-cycles, and $2$-cycles, respectively, we see that these three directions correspond to each of the three cusps.  As above, it suffices to compare colors within each cylinder in Figure \ref{Icosahedron6CoverFig} to conclude.
\end{proof}

\begin{figure}
\centering
\includegraphics[scale=.9]{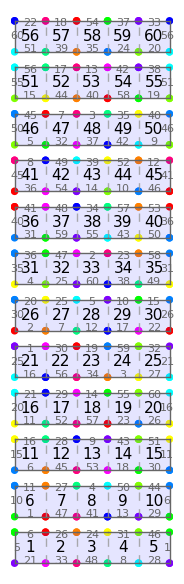}
\caption{The canonical $6$-cover of the icosahedron after acting by an element of $GL(2,\bR)$.}
\label{Icosahedron6CoverFig}
\end{figure}




\subsubsection{Blocking Cardinality}\label{sect:blocking cardinality} We recall from Leli\`evre, Monteil, and Weiss~\cite{LMW} that a pair of points $(x, y)$ on a translation surface $S$ is \emph{finitely blocked} if there exists a finite set  which does not contain $x$ or $y$ and intersects every straight-line trajectory connecting $x$ and $y$. A set with this property is called an \emph{blocking set} for $(x, y)$. The \emph{blocking cardinality} of $(x, y)$ is the minimal size of a blocking set. The above result can be interpreted as giving an upper bound on the blocking cardinality of $(x, x)$ for any singularity $x$ on the translation surfaces arising from unfolding the arithmetic Platonic solids.

Leli\`evre, Monteil, and Weiss~\cite{LMW} show that the property of every pair of points being finitely blocked characterizes arithmetic translation surfaces. While this does not imply our results, it fits well together; closed saddle connections are blocked by other singularities on the arithmetic Platonic solids, but not on the one non-arithmetic one. 

\section{Computing Veech groups of covers}
\label{sect:computing veech groups}

In this section we assume $p:S \to P$ is a branched covering map between translation surfaces. Define
\begin{equation}
\label{eq:X}
Y_S=\SL(2,\bR)/V(S) \quad \text{and} \quad Y_P=\SL(2,\bR)/V(P),
\end{equation}
where we consider these spaces to have basepoints at the coset containing the identity. The closely related projectivizations 
$$X_S = \PSL(2,\bR)/PV(S) \quad \text{and} \quad X_P = \PSL(2,\bR)/PV(P)$$
can be thought of as unit tangent bundles to the the quotients of the Teichm\"uller disk by the Veech groups (called the Teichm\"uller curves in the lattice case). We give an algorithm which gives an understanding of $Y_S$ under some hypotheses on $P$ and the covering map $S \to P$.

Our algorithm is somewhat similar to the algorithm of Schmith\"usen \cite{S04} which does the same for square-tiled surfaces
and an algorithm of Finster~\cite{Finster}
which does the same for translation surfaces with a regular $n$-gon decomposition. 


The algorithm we present is geometric (whereas \cite{S04} and \cite{Finster} take a more algebraic viewpoint). In particular, 
we assume that the computer can store exact presentations of translation surfaces defined over an algebraic number field $F$ defined over $\bQ$, that the computer can compute the image of a translation surface under the action of an element of $\SL(2,F)$, and
the computer can tell if two translation surfaces are the same. \verb|FlatSurf| \cite{flatsurf} can do these things, and
we give a brief explanation of how \verb|FlatSurf| works in Appendix \ref{appendix:flatsurf}. In \S \ref{sect:dodecunfolding},
we implement the algorithm using \verb|FlatSurf| to study the unfolding of the dodecahedron.

\subsection{Hypotheses and conclusions of the algorithm}

Let $p:S \to P$ be a branched covering map and define $Y_S$ and $Y_P$ as in \eqref{eq:X}.
We will assume
\begin{enumerate}
\item[(H0)] The subgroup $V(S)$ is a subgroup of $V(P)$. 
\end{enumerate}

Hypothesis (H0) implies that there is a covering map $\chi: Y_S \to Y_P$.
The fiber over the base point of $Y_P$ is given by the cosets in $V(P)/V(S)$. Observe
that this fiber is in bijective correspondence with images of $S$ under $V(P)$:
\begin{proposition}
\label{prop:injective}
Let $\cH$ be the stratum of $S$. 
The map $V(P)/V(S) \to \cH$, defined so that for each $A \in V(P)$ the coset 
$A V(S)$ maps to $A(S)$, is injective.
\end{proposition}
\begin{proof}
First observe the map is well defined. If $A V(S)=A' V(S)$, then $A^{-1}A' \in V(S)$ and so 
$A^{-1} A'(S)=S$. Applying $A$ yields $A'(S)=A(S)$.

Now we will show the map is injective. Let $A V(S)$ and $B V(S)$ be cosets and suppose $A(S)=B(S)$. Then $A^{-1} B \in V(S)$,
and $A V(S) = B V(S)$ follows.
\end{proof}

We will also assume
\begin{enumerate}
\item[(H1)] The subgroup $V(S)$ is finite index in $V(P)$. 
\item[(H2)] The group $V(P)$ comes with a finite generating set $\{M_1, \ldots, M_m\}$. 
\end{enumerate}
Statement (H1) means that we can enumerate the cosets in $V(P)/V(S)$, while (H2) ensures we can examine all of the generators.
(Note that some Veech groups are infinitely generated \cite{McM03,HS04}.) 

Our algorithm will produce words $C_1, \ldots, C_n$ in the generators $\{M_1, \ldots, M_m\}$ 
and surfaces $S_1,\ldots, S_n$ so that 
\begin{enumerate}
\item[(C0)] $S_1=S$ and $C_1=I$. 
\item[(C1)] $C_i(S)=S_i$ for all $i \in \{1, \ldots, n\}$. 
\item[(C2)] We have $[V(P):V(S)]=n$ and the collections $\{C_i V(S)\}$ and $\{S_i\}$ enumerate
$V(P)/V(S)$ and the $V(P)$-orbit of $S$, respectively.
\end{enumerate}
Finally, we will compute the left-action of each generator $M_i$ on the cosets and surfaces. For each $M_j$ we will
compute a permutation $m_j$ of $\{1, \ldots, n\}$ satisfying
\begin{enumerate}
\item[(C3)] For each $i$, 
$M_j(S_i)=S_{m_j(i)}$ and 
$M_j C_i V(S) = C_{m_j(i)} V(S)$.
\end{enumerate}
It follows that the permutations $m_j$ can be used to determine the monodromy action of $\pi_1(X_P)$ on 
the fiber $PV(P)/PV(S)$; see Proposition \ref{prop:monodromy}.
In addition these conclusions give access to a finite generating set for $V(S)$:

\begin{proposition}
With input matrices $\{M_j\}$ as above and output matrices $\{C_i\}$ satisfying the conclusions above, $V(S)$ is generated by
$$\{C_{m_j(i)}^{-1} M_j C_i:~ \text{$1 \leq i \leq n$ and $1 \leq j \leq m$}\}.$$
\end{proposition}
\begin{proof}
First of all observe that any $C_{m_j(i)}^{-1} M_j C_i \in V(S)$ as a consequence of (C3). Now suppose $g \in V(S)$.
Since $V(S) \subset V(P)$ and $\{M_1, \ldots, M_m\}$ generates $V(P)$, there is a choice of $g_0, \ldots, g_K \in V(P)$
so that $g_0=I$, $g_K=g$ and 
$$g_k = M_{j(k)} g_{k-1} \quad \text{or} \quad g_k=M_{j(k)}^{-1} g_{k-1}
\quad \text{for some $j(k)$ for each $k=1, \ldots, K$}.$$
Consider the path on $V(S)/V(P)$ given by $k \mapsto g_k V(S)$
for $k \in \{1, \ldots, K\}$.
From conclusion (C3) we know for each $k$ there is an $i$ so that
$g_k V(S)=C_{i} V(S)$. The index $i$ associated to $k$ can be determined inductively using the permutations $m_j$. Set $i(0)=1$ and inductively define
$$i(k)= \begin{cases}
m_j \circ i(k-1) & \text{if $g_k = M_{j(k)} g_{k-1}$} \\
m_j^{-1} \circ i(k-1) & \text{if $g_k = M_{j(k)}^{-1} g_{k-1}$} \\
\end{cases}
\quad 
\text{for $k \in \{1, \ldots, K\}$.}
$$
By induction observe with this definition we have that $g_k V(S)=C_{i(k)} V(S)$ for all $k$ using (C3). Now define the sequence of group elements of $V(S)$ inductively using the claimed generators by the rules that
$h_0=I$ and 
$$h_k = 
\begin{cases}
C_{i(k)}^{-1} M_j C_{i(k-1)} h_{k-1} & \text{if $g_k = M_{j(k)} g_{k-1}$} \\
C_{i(k)}^{-1} M_j^{-1} C_{i(k-1)} h_{k-1} & \text{if $g_k = M_{j(k)}^{-1} g_{k-1}$} \\
\end{cases}
\quad 
\text{for $k \in \{1, \ldots, K\}$,}
$$
noting that we are left multiplying $h_{k-1}$ by one of the generators in our list or its inverse.
Now by induction we can see that $h_k=C_{i(k)}^{-1} g_k$ for each $k$. Since $g=G_K \in V(S)$, we know
that $g_K(S_1)=S_1$ and so $i(K)=0$. Specializing to this case we see $h_K=C_1^{-1} g_K=g_K$ since $C_1=I$. 
Thus our expression for $h_K$ gives $g$ as a product of generators and their inverses.
\end{proof}

\subsection{The algorithm}
\label{sect:algorithm}

\noindent{\bf Algorithm 1:}
\begin{itemize}
\item Let $n=1$, let $S_1=S$, and $C_1=I$.\hfill $1$
\item Let $NF=\{1\}$ be a modifiable set of indices of newly found surfaces. \hfill $2$
\item Define $m_1, \ldots, m_k$ to be dictionaries (finite editable maps). \hfill $3$
\item While $NF$ is non-empty:\hfill $4$
\begin{itemize}
\item Remove an index $i$ from $NF$.\hfill $5$
\item For each $j \in \{1, \ldots, m\}$:\hfill $6$
\begin{itemize}
\item Set $S'=M_j(S_i)$.\hfill $7$
\item If $S'$ appears in the list of surfaces $S_1, \ldots, S_n$:\hfill $8$
\begin{itemize}
\item Set $k$ so that $S'=S_k$.\hfill $9$
\item Define $m_j(i)=k$.\hfill $10$
\end{itemize}
\item Otherwise:\hfill $11$
\begin{itemize}
\item Increment $n$ and define $S_n=S'$ and set $C_n=M_j C_i$.\hfill $12$
\item Add $n$ to the set $NF$.\hfill $13$
\item Define $m_j(i)=n$. \hfill $14$
\end{itemize}
\end{itemize}
\end{itemize}
\end{itemize}

\begin{proposition}
If the hypotheses (H0)-(H2) are satisfied then the algorithm terminates and conclusions (C0)-(C3) are satisfied.
\end{proposition}
\begin{proof}
To see that the algorithm terminates, observe there is at most one index created for each surface in the $V(P)$ orbit of $S$.
Thus the total number of indices created is finite by (H1). 
Each index is then placed in $NF$ at most once, so the total number of times the while loop is run is finite. The inner loop is also clearly finite by (H2), so the algorithm terminates in finite time. 

Statement (C0) is set to hold by the first statement in the algorithm. 

Statement (C1) can be observed by induction. 
Statement (C0) gives the base case, while whenever we define a new $C_n$ and $S_n$ in line $12$ we have
$S_n=S'=M_j(S_i)$ and $C_n=M_jC_i$ providing the inductive step.

Once the algorithm terminates, we will have iterated over every created index exactly once. The inner loop will then have iterated over every $i \in \{1, \ldots, n\}$ and every $j \in \{1, \ldots, m\}$. The inner loop always defines $m_j(i)$ either in line 10 or line 14. In either case, we define $m_j(i)$ so that it satisfies $S'=S_{m_j(i)}=M_j S_i$. This verifies the first part of (C3).
To see the second part observe that using (C1) we have that $M_j(S_i)=S_{m_j(i)}$ implies $M_j C_j(S) = C_{m_j(i)}(S)$ and so
$C_{m_j(i)}^{-1} M_j C_j \in V(S)$ and therefore $M_j C_i V(S) = C_{m_j(i)} V(S)$.

To see statement (C2) first observe that the surfaces $S_1, \ldots, S_n$ are distinct. This is because we never introduce a new surface (in line $12$) unless it is distinct from the previously indexed surfaces. Proposition \ref{prop:injective} combined with (C1) then guarantees that the cosets $C_1 V(S), \ldots, C_n V(S)$ are all distinct. 

We will now show that $C_1 V(S), \ldots, C_n V(S)$ enumerates $V(P)/V(S)$. 
Observe that because $V(P)/V(S)$ is finite, the semigroup generated by $\{M_j\}$ acts transitively on 
$V(P)/V(S)$. (The left action of $V(P)$ acts transitively on $V(P)/V(S)$ and gives a group homomorphism $\rho: V(P) \to \mathrm{Perm}\big(V(P)/V(S)\big)$. Since $\{M_j\}$ generates $V(P)$
and $\mathrm{Perm}\big(V(P)/V(S)\big)$ is a finite group, the image under $\rho$ of the generated semigroup coincides with $\rho\big(V(P)\big)$ and is therefore also transitive.)
Statement (C3) implies that the collection of cosets $\{C_i V(S):~0 \leq i \leq n\}$ is 
invariant under each of the generators $M_j$, so by the transitivity of the semigroup action, we see that $\{C_i V(S)\}=V(P)/V(S)$,
proving our statement about coset enumeration. Since we have enumerated the cosets and $S_i=C_i(S)$ for all $i$,
the set $\{S_1, \ldots, S_n\}$ enumerates the $V(P)$-orbit of $S$. This proves that (C2) holds.
\end{proof}

We remark that by Corollary \ref{cor:finite index}, hypotheses (H0)-(H2) hold if $P$ is primitive, of genus larger than one,
has Veech's lattice property, the singular set of $P$ is the set of zeros of the associated holomorphic $1$-form, and $p:S \to P$ is a translation covering.
In particular, our algorithm can be applied to any translation surface admitting an regular $n$-gon decomposition with $n \not \in \{3,4,6\}$. Here the covering $S \to \Pi_n$ guaranteed by Proposition \ref{prop: natural covering}
satisfies the hypotheses of the algorithm; see \S \ref{sect:decompositions} and \S \ref{sect:additional}.
The remaining cases of $n \in \{3,4,6\}$ are the arithmetic cases and the hypotheses of the algorithm are still true if $P$ is taken 
as in Theorem \ref{thm:square tiled coverings}.

\section{The dodecahedron unfolding}
\label{sect:dodecunfolding}

In this section, we carry out the algorithm described in \S \ref{sect:computing veech groups} to understand the Veech group and Teichm\"uller curve of the unfolding of the dodecahedron. We make use of \verb|FlatSurf| which once installed following directions provided in
\cite{flatsurf} can be started in SageMath using the command:
\begin{verbatim}
from flatsurf import *
\end{verbatim}

\subsection{The number field}

The regular pentagon with side length two, with one vertex at the origin, and with a horizontal side has coordinates which are algebraic integers in the number field $F={\mathbb Q}(s)$ where $s=2 \sin \frac{\pi}{5}$.

The following code defines $s$ as an algebraic number (as \verb|s_AA|), defines
the number field $F$ and $s$ as a member of this number field.
\begin{verbatim}
s_AA = AA(2*sin(pi/5))
F.<s> = NumberField(s_AA.minpoly(), embedding=s_AA)
\end{verbatim}

\subsection{The pentagons}

The {\em top regular pentagon} is the one whose bottom side is horizontal. Edges are labeled by $\{0,1,2,3,4\}$ with edge $0$ horizontal. The following code defines the top pentagon with side length two:
\begin{verbatim}
pentagon_top = 2 * polygons.regular_ngon(5, field=F)
\end{verbatim}
The bottom pentagon with side length two may be defined by
\begin{verbatim}
pentagon_bottom = (-1) * pentagon_top
\end{verbatim}

\subsection{The double pentagon}

The double pentagon $\Pi_5$ is the translation surface built from two regular pentagons depicted in Figure \ref{DoublePentConvs}. As above we normalize the geometry of $\Pi_5$ so that all edges of the regular pentagons have length two.

As stated in Proposition \ref{prop:veech group}, the Veech group of $\Pi_5$ has the form:
\begin{equation}
\label{eq:VeechGroupGenerators}
V(\Pi_5)=\langle R, T \rangle \quad \text{and} \quad V_\pm (\Pi_5)=\langle R, T, J\rangle
\end{equation}
where $R$, $T$ and $J$ denote the matrices
\begin{equation}
\label{eq:RTJ}
R = \left(\begin{array}{rr}
\cos \frac{\pi}{5} & -\sin \frac{\pi}{5} \\
\sin \frac{\pi}{5} & \cos \frac{\pi}{5}
\end{array}\right),
\quad
T = \left(\begin{array}{rr}
1 & 2 \cot \frac{\pi}{5} \\
0 & 1
\end{array}\right),
\quad_{}
J = \left(\begin{array}{rr}
1 & 0 \\
0 & -1
\end{array}\right).
\end{equation}
The matrices satisfy the identities
\begin{equation}
\label{eq:identities}
R^5=-I, \quad (RT^{-1})^2=-I, \quad \quad J^2=I, \quad JR=R^{-1}J, \quad JT = T^{-1}J,
\end{equation}
and these relations together with the generators listed in \eqref{eq:VeechGroupGenerators}
describe $V(\Pi_5)$ and $V_\pm(\Pi_5)$ as abstract groups (noting that $-I$ is an involution lying in the center of the Veech groups).

We define the matrices $R$, $T$ and $J$ into SageMath using:
\begin{verbatim}
R = Matrix(F, [[ cos(pi/5), -sin(pi/5) ],
               [ sin(pi/5),  cos(pi/5) ]] )
T = Matrix(F, [[ 1, 2*cot(pi/5) ],
               [ 0,           1 ]] )
J = Matrix(F, [[ 1,  0 ],
               [ 0, -1 ]] )
\end{verbatim}

\begin{figure}
\centering
\includegraphics[height=2in]{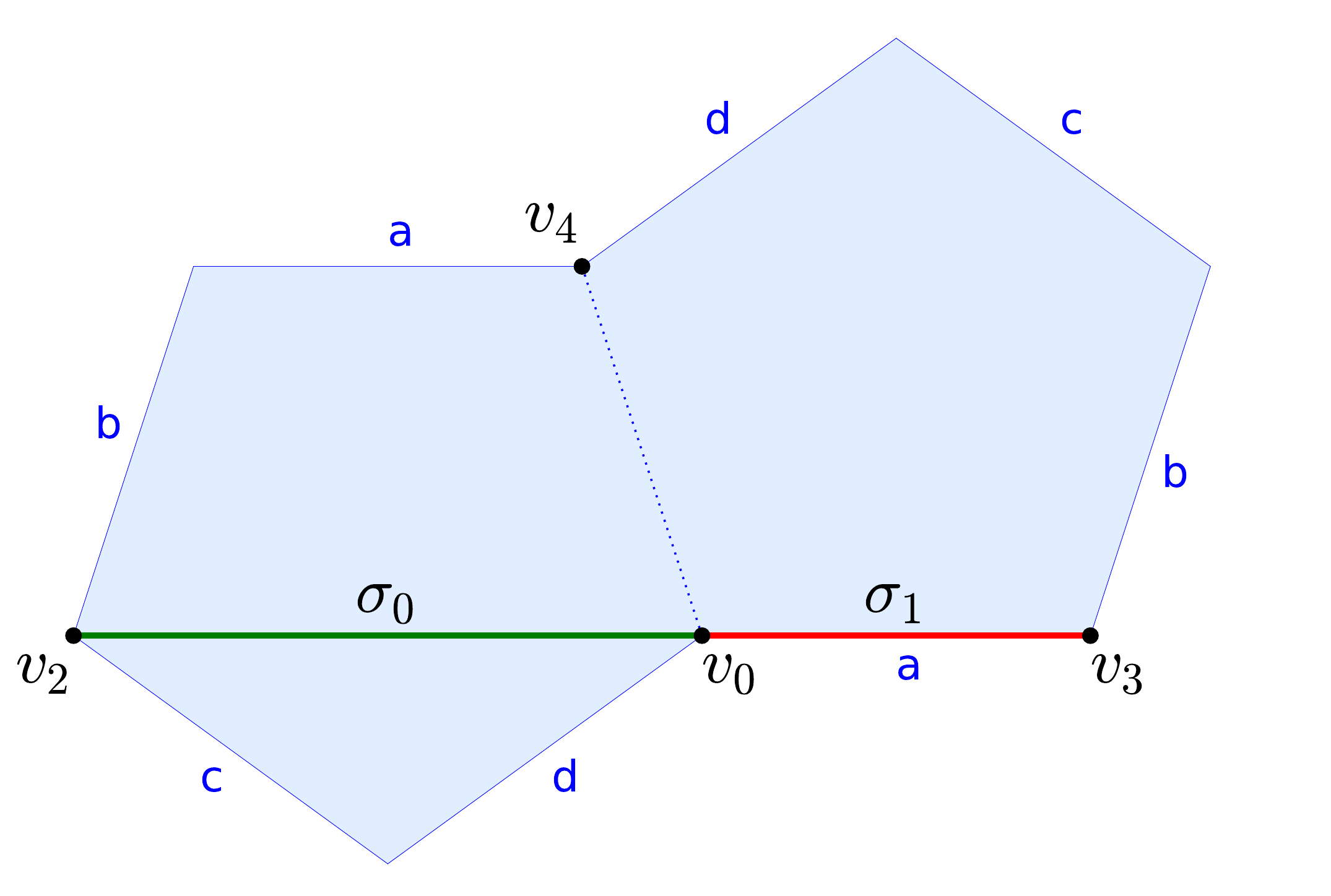}
\caption{The double pentagon $\Pi_5$ together with the saddle connections $\sigma_0$ and
$\sigma_1$.}
\label{DoublePentConvs}
\end{figure}

Observe that because $\Pi_5$ has no translation automorphisms, the derivative map $D:\Aff_\pm(\Pi_5)\to \GL(2,\bR)$
is a group isomorphism onto its image $V_\pm(\Pi_5)$. The projectivized derivative map $P D:\Aff_\pm(\Pi_5) \to \PGL(2,\bR)$ is 
a $2$-$1$ map onto its image $P V_\pm(\Pi_5)$.

\begin{proposition}
\label{SaddConnReductionProp}
Let $\sigma_0$ and $\sigma_1$ be the saddle connections depicted in Figure \ref{DoublePentConvs} and let $\sigma$ be any saddle connection on $\Pi_5$.
We view these saddle connections as lacking orientation. For $i \in \{0,1\}$ define 
$$\Aff_\pm(\sigma_i, \sigma)=\{f \in \Aff_\pm(\Pi_5):~f(\sigma_i)=\sigma\}.$$
Then precisely one of these two sets is non-empty. Furthermore if $\Aff_\pm(\sigma_i, \sigma)$ is non-empty, then
$$P D\big(\Aff_\pm(\sigma_i, \sigma)\big) \in PV_\pm(\Pi_5)/ \langle T, J \rangle.$$
\end{proposition}
\begin{proof}
Fix $\sigma$. Because $V(\Pi_5)$ is a lattice, $\sigma$ must be stabilized by an affine automorphism $g:\Pi_5 \to \Pi_5$ with parabolic derivative. Since the quotient 
$\bH^2/PV(\Pi_5)$ has one cusp (represented by the conjugacy class of $T$ in our generating set), we have $D(g)= \pm M T^n M^{-1}$ for some $n \in \bZ$
and some matrix $M$ in $V(\Pi_5)$. Then there is an $f \in \Aff(\Pi_5)$ with $Df=M$ which sends the horizontal direction to the direction of $\sigma$. 
Observe $\sigma'=f^{-1}(\sigma)$ is a horizontal saddle connection. By inspection we see that $\sigma' \in \{\sigma_0, \sigma_1, -\sigma_0\}$
where $-\sigma_0$ denotes the image of $\sigma_0$ under the hyperelliptic involution. If $\sigma'=\sigma_0$ we have shown 
$M \in D \Aff_\pm(\sigma_0, \sigma)\big)$, if $\sigma'=\sigma_1$ we have shown 
$M \in D \Aff_\pm(\sigma_1, \sigma)\big)$ and if $\sigma'=-\sigma_0$ we have shown 
$-M \in D \Aff_\pm(\sigma_0, \sigma)\big)$. In any case, one of the sets $\Aff_\pm(\sigma_i, \sigma)\big)$ is non-empty. 

To see that only one of the sets $\Aff_\pm(\sigma_i, \sigma)\big)$ is non-empty, suppose to the contrary that
$M_i \in D \Aff_\pm(\sigma_i, \sigma)\big)$ exists for $i \in \{0,1\}$. Then $M_1 \circ M_0^{-1}$ has a horizontal eigenvector with eigenvalue $\pm \varphi=\pm \frac{1+\sqrt{5}}{2}$. But then the subgroup $\langle M_0 \circ M_1^{-1}, T \rangle \subset V_\pm(\Pi_5)$ is not discrete, which is a contradiction.

Consider the final statement. From the first paragraph we have some $\pm M \in D\big(\Aff_\pm(\sigma_i, \sigma)\big)$ for some $i$.
Let $\mathit{Stab}(\sigma_i)$ denote the set of all $f \in \Aff_\pm(\Pi_5)$ so that $f(\sigma_i)=\sigma_i$ as a non-oriented saddle connection. Observe that
$$
D \circ \mathit{Stab}(\sigma_0)=\langle T, -J\rangle
\quad \text{and} \quad 
D \circ \mathit{Stab}(\sigma_1)=\langle T, J, -I \rangle.
$$
Then since the projectivizations of these subgroups coincide we have
$$P D\big(\Aff_\pm(\sigma_i, \sigma)\big) = M \cdot P D\big( \mathit{Stab}(\sigma_i) \big) = M \cdot \langle T, J \rangle \in  PV_\pm(\Pi_5)/ \langle T, J \rangle.$$
\end{proof}

\subsection{The unfolding}
We will explain how to enter the unfolding $\tilde D$ of the dodecahedron in \verb|FlatSurf| using the description of $\tilde D$ provided in \S \ref{sect:unfoldings}.

The following code tells \verb|FlatSurf| that we will build a combinatorial surface
by gluing together edges of polygons whose vertices lie in $F^2$:
\begin{verbatim}
S = Surface_dict(F)
\end{verbatim}

The polygons making up $\tilde D$ are top and bottom pentagons
indexed by pairs in $\{0, \ldots, 9\} \times \{0, \ldots, 11\}$ and the list produced by \verb|double_pent_top()| tells us which indices correspond to
top pentagons; see \S \ref{sect:labeling sheets}. We now associate the correct pentagon to each index:
\begin{verbatim}
top_indices = double_pent_top()
for sheet in range(10):
    for pent in range(12):
        if (sheet, pent) in top_indices:
            S.add_polygon(pentagon_top, label=(sheet,pent))
        else:
            S.add_polygon(pentagon_bottom, label=(sheet,pent))
\end{verbatim}
Also \verb|FlatSurf| requires our surfaces to have a base label:
\begin{verbatim}
S.change_base_label((0,0))
\end{verbatim}

The function \verb|build_adj_dodec| of Example \ref{ex:dodecahedron} describes how the edges of the pentagons should be glued. The code tells \verb|FlatSurf| about these gluings:
\begin{verbatim}
for sheet in range(10):
    for pent in range(12):
        adj = build_adj_dodec(sheet,pent)
        for edge in range(5):
            S.change_edge_gluing((sheet,pent), edge, tuple(adj[edge]), edge)
\end{verbatim}
This completes the definition of the combinatorial surface $S$. To make it so it can no longer be changed:
\begin{verbatim}
S.set_immutable()
\end{verbatim}
To define $\tilde D$ (or \verb|Dtilde|) to be a translation surface based on the gluings described by $S$ we enter:
\begin{verbatim}
Dtilde = TranslationSurface(S)
\end{verbatim}

\subsection{Monodromy}
\label{sect:dodecahedron monodromy}
We will carry out the algorithm described in \S \ref{sect:computing veech groups} in the case of the covering $\tilde D \to \Pi_5$.

We begin by checking the hypotheses. Observe that $\tilde D \to \Pi_5$ is a translation covering and $\Pi_5$ is a primitive lattice surface so that by Corollary \ref{cor:finite index}, $V(\tilde D)$ is a finite index subgroup of $V(\Pi_5)$. This verifies hypotheses (H0) and (H1). As remarked earlier we have $V=\langle R,T\rangle$ verifying (H2). Thus, we can run the algorithm of \S \ref{sect:computing veech groups}.

As in \S \ref{sect:computing veech groups}, set 
$$Y_{\Pi_5}=\SL(2,\bR)/V(\Pi_5) \quad \text{and} \quad Y_{\tilde D}=\SL(2,\bR)/V(\tilde D)$$
and recall that we get a covering map $Y_{\tilde D} \to Y_{\Pi_5}$. The group $V(\Pi_5)$ acts on the fiber over the base point of $Y_{\tilde D} \to Y_{\Pi_5}$. In this special case, we see that $V_\pm(\Pi_5)$ also acts on this fiber and we can explicitly find the action of $J \in V_\pm(\Pi_5)$ from the information about the action of $T$ and $R$:
\begin{proposition}
\label{prop:j}
Let $\varphi:V(\Pi_5) \to V(\Pi_5)$ be the automorphism defined so that
$\varphi(T)=T^{-1}$ and $\varphi(R)=R^{-1}$. Then for any $A \in V(\Pi_5)$ we have
$$J A(\tilde D) = \varphi(A)(\tilde D).$$
In particular, $V(\Pi_5)(\tilde D)$ is also the orbit of $\tilde D$ under $V_{\pm}(\Pi_5) \subset \GL(2,\bR)$. \end{proposition}
\begin{proof}
In light of \eqref{eq:identities} we see that $\varphi$ describes the action of $J$ by conjugation on $V(\Pi_5)$. Theorem \ref{NormCov} tells us that $\tilde D$ is Platonic, and it follows that $J(\tilde D)=\tilde D$. We conclude
$$J A(\tilde D) = J A J^{-1}(\tilde D)=\varphi(A)(\tilde D).$$
\end{proof}

Let $N$ denote the number of elements of $V(\Pi_5)(\tilde D)$. We will use \verb|FlatSurf| to follow the algorithm in \S \ref{sect:computing veech groups} in order to
\begin{enumerate}
\item Compute $N$, the index of $V(\tilde D)$ in $V(P)$.
\item Enumerate $V(\Pi_5)(\tilde D)=\{E_1, E_2, \ldots, E_{N}\}$ with $E_1=\tilde D$.
\item Compute permutations $r, t$ of $\{1,\ldots, N\}$ satisfying 
$$R(E_i) = E_{r(i)} \quad \text{and} \quad  T(E_i) = E_{t(i)}$$
for all $i \in \{1, \ldots, N\}$.
\end{enumerate}

\begin{remark}[Indexing of surfaces]
\label{rem:indexing}
Consider the free monoid of words in the alphabet $\{`R\textrm', `T\textrm'\}$ equipped with the operation of concatenation.
We equip this monoid with a total ordering $\prec$. We declare that $w_1 \prec w_2$ if the word $w_1$ is shorter than $w_2$.
If $w_1$ and $w_2$ have the same length we define $w_1 \prec w_2$ to agree with the lexicographical ordering read from {\em right to left}, i.e.,
$w_1 \prec w_2$ if in the first position from the right at which they differ an $`R\textrm'$ appears in $w_1$ and a $`T\textrm'$ appears in $w_2$ (e.g., we have $`TRR \prec RTR\textrm'$).
Words in $`R\textrm'$ and $`T\textrm'$ act on surfaces through the matrix action of $R$ and $T$. Observe that for any surface, $E \in V(\Pi_5)(\tilde D)$ there is a smallest word $w=w(E)$ in the ordering so that $w(\tilde D) = E$. Our algorithm indexes the surfaces in  $V(\Pi_5)(\tilde D)$ so that
$$w(E_1) \prec w(E_2) \prec \ldots \prec w(E_N).$$
This condition uniquely determines the indexing of surfaces.
\end{remark}

We will also:
\begin{enumerate}
\item[4.] Compute the words $\verb!w[i]!=w(E_i)$ following the remark above.
\end{enumerate}
The following code computes these four quantities. Compare \S \ref{sect:algorithm}.

\begin{verbatim}
gens = [(R,'R'), (T,'T')] 
for gen,letter in gens:
    gen.set_immutable()
gens_action = {gen:{} for gen,letter in gens} 
Dtilde_canonicalized = Dtilde.canonicalize()
leaves = [Dtilde_canonicalized]
surface_to_index = {Dtilde_canonicalized: 1}
w = {1: ''}
while len(leaves)>0: 
    old_leaves = leaves
    leaves = []
    for leaf in old_leaves:
        leaf_index = surface_to_index[leaf]
        for gen, letter in gens:
            image_surface = (gen*leaf).canonicalize()
            if image_surface in surface_to_index:
                image_index = surface_to_index[image_surface]
            else:
                image_index = len(surface_to_index)+1
                surface_to_index[image_surface] = image_index
                leaves.append(image_surface)
                w[image_index] = letter + w[leaf_index]
            gens_action[gen][leaf_index] = image_index
N = len(surface_to_index)
r = Permutation( [ gens_action[R][i] for i in range(1,N+1) ] )
t = Permutation( [ gens_action[T][i] for i in range(1,N+1) ] )
\end{verbatim}

Running this code we see that $N=2106$. The computed permutations $r$ and $t$ of the set $\{1, \ldots, 2106\}$ are given in the auxiliary file
\verb|code_from_the_article.ipynb|\footnote{This answer and the list of cosets was independently checked by code produced by Myriam Finster through personal communication.}.
By Proposition \ref{prop:j}, the permutation $j$ of $\{1, \ldots, 2106\}$
satisfying $J(E_i)=E_{j(i)}$ can be computed using the following code:
\begin{verbatim}
rinv = ~r  # The inverse permutation of r
tinv = ~t  # The inverse permutation of t
def index_of_phi_of_word(word):
    # Given a word in the letters "R" and "T", this function computes
    # the surface index of the surface phi(word)(tildeD).
    index = 1
    for letter in reversed(word):
        if letter == "R":
            index = rinv(index)
        elif letter == "T":
            index = tinv(index)
    return index
j = Permutation([index_of_phi_of_word(w[i]) for i in range(1,2107)])
\end{verbatim}

To look for closed saddle connections in $\S$ \ref{sect:closed}, we will need to understand the $<T,J>$-orbits of the surfaces $\{E_1, \ldots, E_N\}$.
The action of $T$ and $J$ are given by the permutations $t$ and $j$. 
The $\langle t,j\rangle$-orbits partition $\{1,\ldots, N\}$ into equivalence classes. We compute this partition as \verb|equivalence_classes| below. Then we produce a list of words called \verb|wordlist| which contains a single word 
in \verb|"R"| and \verb|"T"| for each equivalence class. The word $w$ associated to an equivalence class is the smallest $w$ (in the sense of Remark \ref{rem:indexing}) so that $w(\tilde D)$ is a surface $E_i$ with index in the equivalence class.

\begin{verbatim}
equiv = {i:{i} for i in xrange(1,2107)} 
# equiv will be a map from indices to their equivalence classes.
for index1 in xrange(1,2107):
    index2 = t(index1)
    for index3 in equiv[index2]:
        equiv[index1].add(index3)
        equiv[index3] = equiv[index1]
    index2 = j(index1)
    for index3 in equiv[index2]:
        equiv[index1].add(index3)
        equiv[index3] = equiv[index1]
equivalence_classes = { frozenset(equiv[i]) for i in xrange(1,2107) }
indexlist = [min(eqclass) for eqclass in equivalence_classes]
indexlist.sort()
wordlist = [w[i] for i in indexlist]
\end{verbatim}

We observe that \verb|wordlist| consists of $211$ words with the command \verb|len(wordlist)|. Recall that we claimed that there were $422$ unfolding-symmetry classes of maximal cylinders on $D$. This follows:

\begin{proof}[Proof of Theorem \ref{thm:422}]
Since the Teichm\"uller curve of the double pentagon has only one cusp, every periodic direction on the double pentagon looks like the horizontal one and in particular there are two maximal cylinders in any such direction. These two cylinders are geometrically very different: for example, the ratio of their circumferences is irrational. Since the covering $\pi_\Pi:\tilde D \to \Pi_5$ is regular, the periodic directions on $\tilde D$ and $\Pi_5$ coincide and there are only two maximal cylinders in such a direction on $\tilde D$ up to the action of the deck group of $\pi_\Pi$. These two cylinders cover those of $\Pi_5$ and are thus not in the same $\Aff_\pm(\tilde D)$ orbit because the circumferences are not commensurable. The calculation above tells us that there are $211$ cusps (or maximal parabolic subgroups of $V(\tilde D)$) up to conjugation in $V_\pm(\tilde D)$. From the above and using Lemma \ref{lem:orbits on cover}, each such cusp gives rise to two $\Aff_\pm(\tilde D)$-orbits of maximal cylinders enumerating those orbits, giving a total of $422$ orbits of maximal cylinders in $\tilde D$. By Proposition \ref{prop:unfolding-symmetry classes}, we have the same number of unfolding-symmetry classes of maximal cylinders on $D$.
\end{proof}

\subsection{Topology of the Teichm\"uller curve}
\label{sect:teich curve dodecahedron}

We explain how knowledge of the permutations $r$ and $t$ yields a description of the Teichm\"uller curve of $\tilde D$ as a cover of the Teichm\"uller curve of $\Pi_5$. We utilize the relationship between hyperbolic geometry and coverings described in \S \ref{sect:hyperbolic}.

The Teichm\"uller curve and its unit tangent bundles of $\tilde D$ are given by
$$\cT(\tilde D) =PV(\tilde D) \bs \PSL(2,\bR)/ \PSO(2) 
\quad \text{and} \quad
U \cT(\tilde D) =  PV(\tilde D) \bs \PSL(2,\bR).$$
We have similar definitions for the double pentagon $\Pi_5$. Since $V(\tilde D) \subset V(\Pi_5)$, we get a covering map
$\pi_U:U \cT(\tilde D) \to U \cT(\Pi_5)$. 
Recall that we computed permutations $r$ and $t$ in \S \ref{sect:dodecahedron monodromy}.

\begin{proposition}
The monodromy group of the covering map $\pi_{U}$ is $\langle r, t\rangle$ where
$r$ and $t$ were the permutations found in \S \ref{sect:dodecahedron monodromy}. In particular,
lifts of $R$ and $T$ to $\pi_1\big(U \cT(\Pi_5)\big)$ which is identified
with the preimage $\upsilon^{-1}(PV(\Pi_5)\big)\subset \widetilde{PSL}(2,\bR)$ act by the permutations $r$ and $t$ respectively.
\end{proposition}
\begin{proof}
Since $PV(\Pi_5)= \langle R, T \rangle/\pm I$, we see that the monodromy group is generated
by the action of $R$ and $T$ on the fiber $\pi_U^{-1}(\Pi_5)$. Consequently,  
$\pi_U^{-1}(\Pi_5) = PV(\Pi_5)(\tilde D)$. Note that because $-I \in V(\tilde D)$, $PV(\Pi_5)=V(\Pi_5)/\pm I$ has a well defined action on
the orbit $V(\Pi_5)(\tilde D)$ and that action agrees with the $V(\Pi_5)$-action. We obtained the permutations $r$ and $t$
using the action of $R$ and $T$ on the orbit $V(\Pi_5)(\tilde D)$.
The conclusion follows from Proposition \ref{prop:monodromy}.
\end{proof}

Now we can prove Theorem \ref{TeichDodec} describing the topology of $\cT(\tilde D)$. 
\begin{proof}[Proof of Theorem \ref{TeichDodec}]
Since $U \cT(\tilde D) \to U \cT(\Pi_5)$ is a cover of degree $N=2106$, we know that
$$\Area~\cT(\tilde D) = 2106 \cdot \Area~\cT(\Pi_5) = 2106 \cdot \frac{3 \pi}{5} = \frac{6318 \pi}{5}.$$

The Teichm\"uller curve $\cT(\Pi_5)$ has exactly one cusp. A loop around this cusp in $\pi_1\big(\cT(\Pi_5)\big)$
is represented by a lift of $T$ to $\widetilde {PV}(\Pi_5)$, and the monodromy action of such an element
was given by $t$. The  number of  cusps of $\cT(\tilde D)$ is then the number of $t$-orbits. We can compute this to be $362$ using the 
sage command \verb!len(t.cycle_type())!.

The Teichm\"uller curve $\cT(\Pi_5)$ has one cone singularity of cone angle $\frac{2 \pi}{5}$. As before
the corresponding monodromy action is given by the $r$-action. Using the command \verb!r.cycle_type()! we observe that $r$ has a single $1$-cycle (corresponding to $\tilde D)$ and $421$ five-cycles. This means that all points in the fiber of $P$ other than $\tilde D$ are flat (regular) points on $\cT(\tilde D)$, i.e., $\cT(\tilde D)$ has exactly one singularity of cone angle $\frac{2 \pi}{5}$.

The Teichm\"uller curve $\cT(\Pi_5)$ has one cone singularity of cone angle $\pi$.
Recall that $(R T^{-1})^2=-I$. This means that a lift of $RT^{-1}$ corresponds to the loop in $U \cT(\tilde D)$ around the cone singularity with cone angle $\pi$. The monodromy action is then given by 
$r \circ t^{-1}$ and the cycle type can be computed with
\begin{verbatim}
r.left_action_product(~t).cycle_type()
\end{verbatim}
This reveals that there are $1044$ two-cycles and $18$ one-cycles. Consequently there are $18$ cone singularities with cone angle $\pi$ on $\cT(\tilde D)$. 

By the Orbifold Gauss-Bonnet formula \cite[Proposition 7.9]{FarbMargalit11}, we have

$$6318\textstyle\frac{\pi}{5} = \Area~\cT(\tilde D) = -2 \pi\left((2-2g)-18\left(1- \textstyle \frac{1}{2}\right)-1\left(1-\textstyle \frac{1}{5}\right) - 362\right).$$
where $g$ is the genus of $\cT(\tilde D)$. Reducing this, we obtain $20g = 2620$ and so $g=131$.
\end{proof}

\section{Closed saddle connections on the dodecahedron}
\label{sect:closed}

In this section we take the reduced list of cosets \verb|wordlist| from Section \ref{sect:dodecahedron monodromy} and perform the search for saddle connections from a zero to itself. We prove Theorem \ref{PlatVertDodecRough} which says there are $31$ unfolding-symmetry classes of closed saddle connections on the dodecahedron.

\subsection{The action of the Veech group on the monodromy permutations}

Consider the four elements of the fundamental group $\pi_1(\Pi_5,p_0)$ denoted $x_i$, which is a closed loop based at $p_0$ and passing through edge $i$ of each pentagon. See Figure \ref{DoublePentConvs}
which is based on \cite[Figure 3.2]{Finster}. Then each of these four elements give rise to the monodromy permutations derived in Section \ref{sect:ConstructPerms}.  They are given by
\begin{verbatim}
x = [python_to_sage(perm_odd(5, i, double_pent_top(), double_pent())) for i in range(4)]
\end{verbatim}

\begin{figure}
\centering
\includegraphics[height=2in]{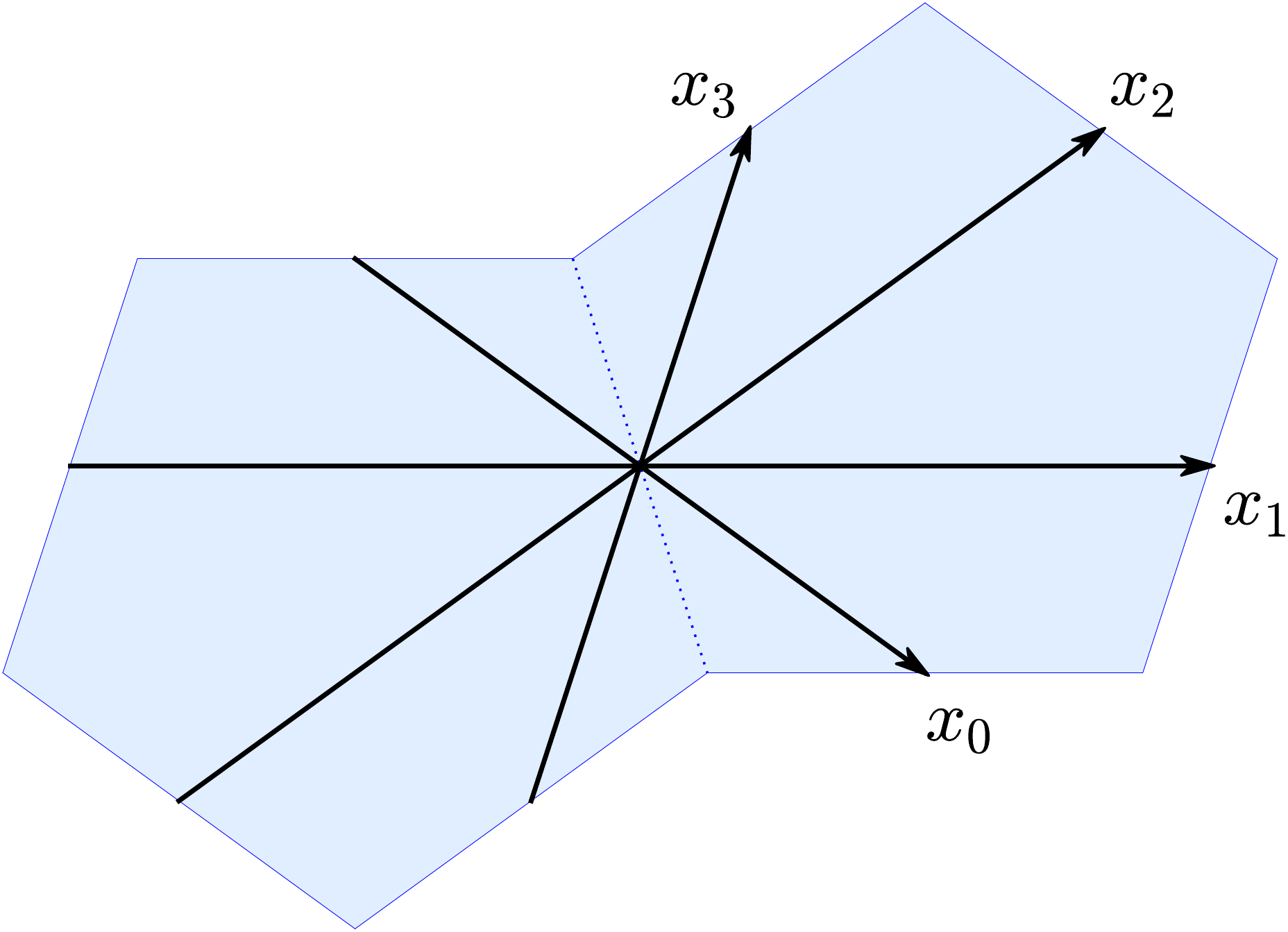}
\caption{The double pentagon $\Pi_5$ with generating loops $x_0$, $x_1$, $x_2$, and $x_3$ for
$\pi_1(\Pi_5,p_0)$, where $p_0$ denotes the common intersection point.}
\label{DoublePentConvsFundGrp}
\end{figure}

Then the action by the generators $R$ and $T$, which generate the Veech group of the double pentagon, induces an action on \verb|x|.  Formulas for the action by each element were derived in \cite{Finster} for the case of arbitrary $p$-gons and double $p$-gons.  However, the reader can easily check the double pentagon case by hand.  The actions on the permutation elements are given by functions \verb|gam_r(x)| and \verb|gam_t(x)| for $R$ and $T$, respectively.

\begin{verbatim}
def gam_r(x):
    return (x[2]*x[3]**(-1), x[2], x[2]*x[0]**(-1), x[2]*x[1]**(-1))
\end{verbatim}

\begin{verbatim}
def gam_t(x):
    return (x[0]*x[1]**(-1), x[1], x[1]*x[2]*(x[3]**(-1))*x[2], x[1]*x[2])
\end{verbatim}

The function \verb|sn_n(perm, elt, power=1)| is the canonical action of $\verb|perm| \in S_n$ on $\verb|elt| \in \{1, 2, \ldots, n\} = S$ by multiplying by \verb|perm| to the power $\verb|power| = \pm 1$.  The function outputs the resulting integer in the set $S$.

\subsection{Saddle connection search}
\label{sect:saddle connection search}
In the Teichm\"uller curve of $\tilde D$, let $\tilde D'$ be any surface corresponding to a point in the fiber over the double pentagon $\Pi_5$.  Since $\tilde D'$ covers $\Pi_5$, it is completely determined by the quadruple of monodromy permutations as $\tilde D$ is.  Let \verb|x| be the quadruple of permutations defining $\tilde D'$.  The function \verb|vert_cycle(n,x)| takes the point $v_0$ incident with sheet \verb|n|, in the sense depicted in Figure \ref{DoublePentConvs}, and on the surface $\tilde D'$ given in the form of the tuple \verb|x| and travels counter-clockwise recording every double pentagon sheet it sees as it does this until it closes.  Since the canonical $10$-cover of the dodecahedron is totally ramified and it lies in the stratum $\cH(8^{20})$, this loop will see a total cone angle of $(8)(2\pi) + 2\pi = 18\pi$.  The corner angle of a pentagon is $3\pi/5$, which implies that this loop passes through $30$ pentagons before closing because $30\left(\frac{3\pi}{5}\right) = 18\pi$.  However, the output of \verb|vert_cycle(n,x)| is a $24$-tuple and \emph{not} a $30$-tuple.  This is clear by observing that $v_0$ and $v_4$ in Figure \ref{DoublePentConvs} each correspond to two pentagons, yet only represent a single sheet of the cover.  Therefore, it outputs a list of $24$ sheets that $v_0$ is incident with.

\begin{verbatim}
def vert_cycle(n,x):
    sheet_cycle = [n];
    end = -1;
    n8 = n
    while end != n:
        n1 = sn_n(x[3],n8,-1)
        sheet_cycle.append(n1)
        n2 = sn_n(x[2],n1)
        sheet_cycle.append(n2)
        n3 = sn_n(x[1],n2,-1)
        sheet_cycle.append(n3)
        n4 = sn_n(x[0],n3)
        sheet_cycle.append(n4)
        n5 = sn_n(x[3],n4)
        sheet_cycle.append(n5)
        n6 = sn_n(x[2],n5,-1)
        sheet_cycle.append(n6)
        n7 = sn_n(x[1],n6)
        sheet_cycle.append(n7)
        n8 = sn_n(x[0],n7,-1)
        if n8 != n:
            sheet_cycle.append(n8)
        end = n8
    return sheet_cycle
\end{verbatim}

As proven in Proposition \ref{SaddConnReductionProp}, it suffices to search in each of the translation surfaces corresponding to coset representative among lifts of the two saddle connections $\sigma_0$ and $\sigma_1$ to see if there is a saddle connection from a vertex to itself.

The long saddle connection $\sigma_0$ connects $v_0$ to $v_2$ in Figure \ref{DoublePentConvs}.  Therefore, the tuple outputted by \verb|vert_cycle(n,x)| will agree between the $0$ and $2~(\textrm{mod}~8)$ place in the tuple exactly when $v_0$ and $v_2$ represent the same zero.  This is exactly what the function \verb|vert_to_self_long(vert)| determines, where \verb|vert| is a $24$-tuple outputted by \verb|vert_cycle(n,x)|.  The function \verb|vert_to_self_short(vert)| does the corresponding search among the short saddle connections.

\begin{verbatim}
def vert_to_self_long(vert):
    vert0_to_vert2 = [(0,i%8,vert[i]) for i in range(2,24,8) if vert[0] == vert[i]]
    return vert0_to_vert2

def vert_to_self_short(vert):
    vert0_to_vert3 = [(0,i%8,vert[i]) for i in range(3,24,8) if vert[0] == vert[i]]
    return vert0_to_vert3
\end{verbatim}

The function \verb|sage_mon_gen_fcn(coset_rep)| takes a coset representative \verb|coset_rep| in the form of a string, and outputs a Python function that is an appropriate concatenation of the functions \verb|gam_r(x)| and \verb|gam_t(x)| above.  Finally, the functions 
\begin{center}
\verb|traj_to_self_search_short(sheet = 1)| and \verb|traj_to_self_search_long(sheet = 1)|
\end{center}
\noindent combine all of the functions above to search through every point in the fiber over the double pentagon for short and long saddle connections on sheet \verb|sheet| of the cover, respectively.

\begin{verbatim}
def traj_to_self_search_short(sheet = 1):
    coset_traj_to_self = []
    for coset in coset_reps:
        y = sage_to_python(str(eval(sage_mon_gen_fcn(coset))))
        coset_search = vert_to_self_short(vert_cycle(sheet,y))
        if len(coset_search) > 0:
            coset_traj_to_self.append(
                (coset_reps.index(coset),coset,coset_search))
    return coset_traj_to_self_short

def traj_to_self_search_long(sheet = 1):
    coset_traj_to_self = []
    for coset in coset_reps:
        y = sage_to_python(str(eval(sage_mon_gen_fcn(coset))))
        coset_search = vert_to_self_long(vert_cycle(sheet,y))
        if len(coset_search) > 0:
            coset_traj_to_self.append(
                (coset_reps.index(coset),coset,coset_search))
    return coset_traj_to_self_long
\end{verbatim}

Since $\tilde D$ is a regular cover of $\Pi_5$ by Theorem \ref{NormCov}, it suffices to run the search with the default value of \verb|sheet| since all answers will be equivalent up to an element of the deck group of $\tilde D$.

\subsubsection{Output of the search}
\label{sect:DodecSearchOutput}

Prior to executing the search above, the reader will observe that the midpoint of the short saddle connection on $\Pi_5$ contains a Weierstrass point and by Theorem \ref{thm:virtual-Weierstrass} no closed saddle connection on the unfolded dodecahedron can be renormalized to it.  Indeed, the function \verb|traj_to_self_search_short()| returns the empty list.


On the other hand, \verb|traj_to_self_search_long()| returns a list with $31$ elements.  The list of all coset representatives with other data is stored in the variable \verb|traj_to_self_coset_list|.  This completes the proof of Theorem \ref{PlatVertDodecRough}.

\subsection{Generating the Trajectories}
\label{sect:generating}

The coset representatives of in \verb|traj_to_self_coset_list| are words in $\{R,T\}$.
Fix such a word $w$. It represents a coset $[w]$ of $\langle T, J \rangle \bs V_\pm(\Pi_5) / V_\pm(\tilde D)$ by viewing $w \in \langle R, T \rangle \subset V_\pm(\Pi_5)$. The word $w$ was returned by \verb|traj_to_self_search_long()| because a lift of $\sigma_0$ (depicted in Figure \ref{DoublePentConvs}) to $w (\tilde D)$ is closed (i.e., the endpoints are the same singularity). Let $\phi:\tilde D \to w(\tilde D)$ be an affine homeomorphism with derivative $w$, we see that $\phi^{-1}(\sigma)$ is a closed saddle connection on $\tilde D$. In particular, $\tilde D$ has a closed saddle connection with holonomy $w^{-1} \hol(\sigma_0)$. By the regularity of the covering $\tilde D \to \Pi_5$ and the fact that $-I \in V(\tilde D)$, we have that in fact every
saddle connection with holonomy $w^{-1} \hol(\sigma_0)$ is closed.
Furthermore, since $R \in V(\tilde D)$, there are closed saddle connections with holonomy $R^k w^{-1} \hol(\sigma_0)$ for every $k \in \{0, \ldots, 9\}$ and every saddle connection on $\tilde D$ with such a holonomy is closed. Note also that all of these saddle connections are unfolding-symmetric.

A list of holonomies $R^k w^{-1} \hol(\sigma_0)$ of closed saddle connections obtained by words in \verb|traj_to_self_coset_list| is provided in the table in Appendix \ref{appendix:list}, with one choice of $w$ and $k$ for each unfolding-symmetry class.
In the auxiliary file \verb|Figures.ipynb|, we draw a closed saddle connection corresponding to each $w$ used in the table. 

\appendix
\section{The work of St\"ackel and Rodenberg\\ \smallskip by Anja Randecker}
\label{appendix:History}

We describe two early 20th century papers which introduce the problem of studying geodesics on the surfaces of polyhedra, written by two colleagues from Hannover, Paul St\"ackel~\cite{Stackel} and Carl Rodenberg~\cite{Rodenberg}. We also give very brief biographical sketches of St\"ackel and Rodenberg.

These papers are cited in the 1936 paper of Fox-Kershner~\cite{FK36}, and  have the same name: \textit{``Geod\"atische auf Polyederfl\"achen''} -- geodesics on surfaces of polyhedra. Rodenberg~\cite{Rodenberg} builds on the work of St\"ackel~\cite{Stackel},  and it seems likely that they had discussed their results in advance. While Stäckel prepares the topic by introducing concepts and giving definitions, Rodenberg applies these concepts to specific examples.

\subsection{St\"ackel and Rodenberg}
Carl Rodenberg (1851--1933) studied mathematics in Karlsruhe and Göttingen, graduated in Göttingen in 1874, and was first a high school teacher in Plauen and a professor of mathematics in Darmstadt and then a professor for descriptive geometry in Hannover until retiring. He is today perhaps best known for designing and building plaster models for a series of more than 20 cubic surfaces~\cite{GotMod}.

Paul St\"ackel (1862--1919) studied mathematics and physics in Berlin and graduated in 1885 under the direction of Kronecker. He wrote his habilitation thesis in Halle and was a professor in Königsberg, Kiel, Hannover, Karlsruhe, and Heidelberg. He is best known for his contributions in integrable systems, differential geometry and complex analysis and for his work on the history of mathematics, in particular on non-Euclidean geometry~\cite{vonRenteln}.

\subsection{Extending geodesics and unwinding}
Motivated by singularities of differential equations defining geodesics, St\"ackel studies straight lines on a polyhedron and explains that a straight line can uniquely be continued over an edge. St\"ackel describes the concept of \emph{unwinding} by fixing a face $F$ of the polyhedron and considering a geodesic segment on $F$ that hits an edge $E$. The other face $F'$ adjacent to $E$ is rotated about $E$ until it lies in the same plane as $F$. Then the geodesic segment can be continued in this plane and hence it can be continued on the polyhedron over the edge.

Rodenberg defines a geodesic as a ``Linienzug, welcher nach einer ebenen Aneinanderreihung der von ihm durchlaufenen Fl\"achen zu einer Geraden wird'' (p.~108), that is, a sequence of line segments which becomes a straight line when stringing together the faces that are traversed by the segments. In modern terminology, we would say that the concept of unwinding defines a developing map. This idea anticipates the concept of unfolding developed by Fox-Kershner~\cite{FK36}, where additionally translation-equivalent faces are identified. 

While St\"ackel and Rodenberg agree on the general definition of a geodesic, they strongly disagree on how to deal with geodesic segments that hit a singularity. St\"ackel says that in some cases (such as the diagonals of a cube) it is impossible to define a continuation of a geodesic segment, Rodenberg is of the opinion that every geodesic segment that hits a corner can be continued in two different ways: there are families of parallel geodesics to the geodesic hitting the corner on either side of it which can be continued over the edge they hit. Taking limits of the continuations of segments `above' and `below' the singular segment one can potentially obtain two possible continuations. 

\subsection{Dihedra and Billiards} St\"ackel and Rodenberg also differ in regard to which polyhedra are considered. Rodenberg uses a broader definition that in particular includes dihedra. He defines a dihedron as a polyhedron where ``der von zwei benachbarten Fl\"achen gebildete Winkel Null ist'' (p.~108), that is, adjacent faces form an angle of $0$. One obtains a general dihedron from gluing all edges of two copies of a polygonal face.
Rodenberg states that a geodesic on a dihedron is ``die Bahn eines von den elastisch gedachten Kanten reflektierten materiellen Punktes oder Strahls'' (p.~108), that is, the orbit of a point or ray that is reflected by the edges which are thought of as being elastic. This is essentially the modern method of describing orbits of polygonal billiards. In this context, Rodenberg conjectures that a generic geodesic is dense on a polyhedron. This would imply in particular that a generic path on a polygonal billiard table is dense, a problem still open for polygonal billiards with irrational angles. 

\subsection{Dichotomies, tilings, and closed geodesics} Both St\"ackel and Rodenberg are interested in the question of understanding different behaviors of geodesics, in particular both closed and non-closed geodesics. They do not consider the problem of singular closed geodesics that is discussed in the current paper.

St\"ackel considers the cube in depth, by considering a tiling of the Euclidean plane by squares. Every unwinding along a geodesic is a ``Kette'' (chain) of squares in this tiling. 
From this St\"ackel deduces that geodesics behave differently depending on whether their slope is rational or irrational. He states that for a slope $\frac{p}{q}$ with $p,q\in \mathbb{Z}$, there is a ``Schar von Achsen'' (family of axes; in today's language, we would say a cylinder) of width~$\frac{1}{\sqrt{p^2+q^2}}$. In particular, these cylinders contain closed geodesics. For an irrational slope, the geodesic ``kommt jedem gegebenen Werte beliebig nahe'' (p.~145), that is, it approximates every given point. This is a version of Kronecker's dichotomy for rotations in a circle: irrational rotations have dense orbits, and rational rotations periodic orbits. 

St\"ackel sketches how a tiling of the Euclidean plane by equilateral triangles can be used to find \emph{closed} geodesics on the tetrahedron, the octahedron, and the icosahedron.
For the dodecahedron, this method cannot be immediately applied as there is no tiling of the plane by regular pentagons. St\"ackel writes ``es scheint, als ob hier nur eine endliche Anzahl von Scharen geschlossener geod\"atischer Linien vorhanden sei'' (p.~150), that is, it seems as if there exists only a finite number of families of closed geodesics. That this conjecture of Stäckel is false was shown by Masur~\cite{Masur} for general abelian differentials, and directly by Theorem \ref{thm:422} (since each of the 422 equivalence classes in that result is infinite).

St\"ackel's discussion on closed geodesics on polyhedra other than the cube starts with an acknowledgement that Rodenberg pointed out to Stäckel that closed geodesics occur on regular polyhedra. Rodenberg~\cite{Rodenberg2} makes clear that he additionally pointed out that by unwinding the dodecahedron, one has a dense set of images of the corners instead of a tiling of the plane by regular pentagons. From this, he deduces the opposite statement: that there should be infinitely many closed geodesics on the dodecahedron. This conjecture is in line with Theorem \ref{thm:422}.

He picks up on this point in~\cite{Rodenberg} and studies how to construct closed geodesics on the dodecahedron. The idea of a tiling of the plane is replaced by a coordinate system defined by a side and a diagonal of a regular pentagon. Rodenberg claims that every ray that goes through an integer point of this coordinate system is contained in a chain of pentagons that comes from unwinding and hence this ray defines a closed geodesic. He mentions that he found this statement through a ``lange Versuchsreihe'' (p.~116), that is, a long series of tests. A rigorous proof seems not to be given here (see also Steinitz's Zentralblatt review~\cite{Steinitz} of~\cite{Rodenberg}).


Rodenberg then studies a procedure to find the integer points that lie on a given ray and a criterion for such a ray to be parallel to a family of geodesics. At this point, he emphasizes that the closed geodesics which are parallel to an edge are the only ones known so far. He then introduces involved notation to describe straight lines from a corner to another corner. Using this notation, he gives a few explicit directions and four families of directions of closed geodesics. His idea is to use a known cylinder of closed geodesics and consider the infinitely many straight lines from one boundary of a cylinder to the other. That the other families have all disjoint directions and are indeed defining closed geodesics, is not elaborated by Rodenberg. 

Finally, Rodenberg claims (supported by computations) the existence of closed billiard trajectories in the regular $n$-gon for any $n$. His computations are presented in some detail for $n=7$ and $n=4p$.

\section{The inner workings of \texttt{FlatSurf}}
\label{appendix:flatsurf}

Here we give a brief explanation of how \verb|FlatSurf| works with translation surfaces. For more details
see the \verb|FlatSurf| documentation \cite{flatsurf}.

\subsection{Basic geometry}

\subsubsection*{Number fields}
SageMath can do exact arithmetic and comparisons between algebraic numbers in a fixed real algebraic field $\bQ(\alpha)$ where $\alpha$ is an algebraic number. For purposes of this exposition we fix an $F=\bQ(\alpha) \subset \bR$.

\subsubsection*{Polygons}
Throughout this appendix all our polygons will convex, i.e., all interior angles are smaller than $\pi$. The \verb|ConvexPolygon| class represents a convex $n$-gon $P$ with vertices in $F^2$ in terms of their vertices $v_0^P, v_1^P, \ldots, v_{n-1}^P$. We write $n^P$ for the number of vertices of $P$. The edge $e_i^P$ of $P$ is the edge joining $v_i^P$ to $v_{i+1}^P$ with addition taken modulo $n$. The vertices are always oriented counter-clockwise around $P$, i.e., as we move along $e_i^P$ from $v_i^P$ to $v_{i+1}^P$, the polygon $P$ is on the left.

\subsubsection*{Surfaces}
\verb|FlatSurf| has a \verb|ConeSurface| class and a \verb|TranslationSurface| class among other geometric surfaces supported.
A surface $S$ is represented as a collection $\{P(i)\}$ of polygons with indices $i$ in some labeling set $\Lambda_S$,
a {\em base label} $i_0^S \in \Lambda_S$ and gluing data. The {\em gluing data} is an involution $\epsilon$ of the set of edges
$$\cE_S = \{(i,j):~\text{$i \in \Lambda_S$ and $0\leq j \leq n^{P(i)}-1$}\}.$$
Write $e(i,j)=e^{P(i)}_j$ for the edge represented by $(i,j)$ which is an edge of $P(i)$ in the notation above.

If the surface is a cone surface and $(i',j')=\epsilon(i,j)$, then $e(i,j)$ and $e(i',j')$ will differ by a rotation
which we use to glue the polygons $P(i)$ and $P(i')$ along these edges. In the translation surface case,
the edges will differ by translation.

\subsubsection*{Affine action}
SageMath has support for matrices with entries in $F$. If $M$ is a $2 \times 2$ matrix with $det(M)>0$ and entries in $F$ and $P$ is a polygon, then $M*P$ will return the polygon $Q$ whose vertices are $v^Q_i=M v^P_i$.

If $S$ is a translation surface defined using polygons $\{P_i:~i \in \Lambda\}$ and gluing data $\epsilon:\cE_S \to \cE_S$
and $M$ is as above, then \verb|FlatSurf| defines $M*S$ to be the surface whose polygons are $\{Q_i=M*P_i:~i \in \Lambda\}$
and uses the same gluing data $\epsilon$. 

\subsection{Canonicalization of translation surfaces}

\subsubsection*{Polygons}
Call a (convex) polygon $P$ {\em standard} if $v_0=(0,0)$ and the other vertices all lie in 
$$\{(x,y)~:~y>0\} \cup \{(x,y)~:~\text{$y=0$ and $x>0$}\}.$$
Observe every polygon differs from a unique standard polygon by translation and cyclic permutation of vertex indices.

The {\em coordinate string} of a polygon $P$ is the list of real numbers:
$$(x_0,y_0, x_1, y_1, \ldots, x_{n-1}, y_{n-1}) \quad \text{where} \quad
\text{$(x_i,y_i)=v_i^P$ for all $i$}.$$

We utilize a total ordering $\prec$ on the collection of all polygons. We say $P \prec Q$ if $P$ has fewer sides than $Q$.
If $P$ and $Q$ have the same number of sides, then we say $P \prec Q$ if the coordinate string of $P$
appears before the coordinate string of $Q$ in the lexicographical ordering.

\subsubsection*{Canonical indexing}

Let $S$ be a connected surface with polygons $\{P(i)~:~ i \in \Lambda\}$, base label $i_0$, and gluing data $\epsilon$. 
The {\em breadth first indexing} extends $i_0$ to an enumeration of $\Lambda$ of the form $\Lambda=\{i_0, i_1, \ldots, i_{n-1}\}$,
which can be produced by the algorithm below. Recall that a {\em queue} is a finite sequence of objects where
elements can be appended and {\em popped} (removed from front).

\noindent
{\bf Algorithm 2:}
\begin{itemize}
\item Set $m=0$.
\item Let $I=[i_0]$ be a queue of indices to check.
\item Let $V=\{i_0\}$ be a set of visited indices.
\item While $I$ is non-empty:
\begin{itemize}
\item Pop an index $i_\ast$ from $I$.
\item For each $j$ with $0 \leq j \leq n^{P(i_\ast)}-1$:
\begin{itemize}
\item Let $(i',j')=\epsilon(i_\ast,j)$. 
\item If $i'$ not in $V$:
\begin{itemize}
\item Increment $m$.
\item Set $i_m=i'$.
\item Add $i'$ to $V$ and append $i'$ to $I$. 
\end{itemize}
\end{itemize}
\end{itemize}
\end{itemize}

We call a surface $S$ {\em canonically indexed} if the index set is $\Lambda=\{0, \ldots, n-1\}$ for some $n \geq 1$,
if the base label is $0$, and if the breadth first indexing satisfies $i_m=m$ for all $m$.

\subsubsection*{Delaunay decomposition}
Let $S$ be a cone surface with singular set $\Sigma$ and let $S^\ast=S \smallsetminus \Sigma$. 
The boundary of a largest immersed ball in $S^\ast$ touches the singular set in at least two points, and there will be immersed balls touching at three or more points. Within such a largest immersed ball, one can take the convex hull of the points in $\Sigma$ touched. This gives rise to the {\em Delaunay decomposition} of $S^\ast$, which is canonical decomposition into polygons.

If $S$ is given as a collection of polygons with edge identifications so that the identified vertex set is $\Sigma$,
then the Delaunay decomposition can be computed by a triangle flipping algorithm \cite{ILTC, BS07}.

\subsubsection*{An ordering on translation surfaces with base polygon}
We will describe  total ordering $\lhd$ on translation surfaces which are given in terms of their Delaunay decomposition,
are canonically indexed, and so that all polygons are standard. Let $S_1$ and $S_2$ be such surfaces. Then we declare $S_1 \lhd S_2$
if $S_1$ is defined with fewer polygons. Now assume they have the same number of polygons, $n$. Since the surfaces are canonically
indexed, the index set is $\Lambda=\{0, \ldots, n-1\}$. Let $P^1_i$ and $P^2_i$ for $i \in \Lambda$ denote the polygons of the respective surfaces. We declare $S_1 \lhd S_2$
if the smallest $i \in \Lambda$ for which $P^1_i \neq P^2_i$ we have $P^1_i \prec P^2_i$, where $\prec$ is our ordering on polygons.
If this did not differentiate the surfaces, we have $P^1_i=P^2_i$ for all $i$, but the gluing data might differ.
Let $\epsilon_1$ and $\epsilon_2$ be the gluing data for the respective surfaces. If the gluing data differs,
then there is a minimal $i \in \Lambda$ for which the gluing of the $i$-th polygons differ and a minimal edge index $j$ for
which the gluing differs. We declare $S_1 \prec S_2$ if $\epsilon_1(i,j)$ precedes $\epsilon_2(i,j)$ lexicographically.

\subsubsection*{A canonicalized translation surface}

Let $S$ be a translation surface with its Delaunay decomposition and so that all polygons are standard.
Let $n$ be the number of Delaunay polygons.
For each polygon $P$ in the Delaunay decomposition, let $S_P$ denote $S$ with polygons canonically indexed so that the base label $0$ refers to $P$. The {\em canonicalization} of $S$ is the smallest $S_P$ in the ordering $\lhd$ where $P$ varies over all polygons in the Delaunay decomposition. 

Observe that two translation surfaces $S$ with singular set consisting of vertices of polygons are equal if and only if their canonicalizations are identical. Furthermore, the canonicalization of a translation surface presented by gluing polygons together
can be computed by an algorithm.

\begin{remark}
This algorithm also gives access to the translation automorphism group of $S$: The translation automorphism group acts simply transitively on the set of polygons $P$ so that $S_P$ is minimal in the $\lhd$ order.
\end{remark}


\section{List of closed saddle connections on the dodecahedron}
\label{appendix:list}

We will now explain how to produce a saddle connection from a word $w$ representing a coset of closed saddle connections.
In Table~\ref{table:list}, we include one row for each equivalence class of closed saddle connections on $\tilde D$. Each row then also represents one of the unfolding-symmetry classes of closed saddle connections on the dodecahedron. Figure \ref{fig:sc_0} shows the saddle connections described in the first three rows of the table.

Following the notation of \S \ref{sect:generating}, each row contains information corresponding to a closed saddle connection on $\tilde D$ with holonomy ${\mathbf v} = R^k w^{-1} (2 \varphi,0)$, where $w \in \langle R, T\rangle$ is one of the coset representatives returned by the code in \S \ref{sect:saddle connection search}, $k \in \{0, \ldots, 9\}$ and $\varphi=\frac{1+\sqrt{5}}{2}$. (We are working with pentagons with sidelength $2$ so $2 \varphi$ is the length of a diagonal.) The word $w$ in the table is the minimal representative of the coset with respect to the ordering on the monoid described in Remark \ref{rem:indexing}. The integer $k$ was selected so that if $L>0$ and $\theta \in {\mathbb R}/2\pi {\mathbb Z}$ are defined such that ${\mathbf v}=L \cdot (\cos \theta, \sin \theta)$, then we have:
\begin{enumerate}
    \item The angle $\theta$ lies in $[0, \frac{3 \pi}{5})$.
    \item The saddle connection on $\Pi_5$ leaving the left endpoint of a horizontal edge of a pentagon and making an angle of $\theta$ with this edge has holonomy ${\mathbf v}$ and bounds a closed cylinder with circumference $L$ on the left. (This saddle connection leaves vertex $v_0$ of Figure \ref{DoublePentConvs} and travels in direction $\theta$ into the top pentagon.)
\end{enumerate}
It then follows that any separatrix $\sigma$ on $D$ or $\tilde D$ which leaves a singularity while making an angle of $\theta$ with an edge of a pentagon is in fact a closed saddle connection of length $L$.

\begin{remark}
There are three saddle connections in $\Pi_5$ leaving the left endpoint of a horizontal edge, making an angle in $[0, \frac{3 \pi}{5})$ with this edge and traveling parallel to $R^k w^{-1} (2 \varphi,0)$ for some $k$. One is a short saddle connection and two are long. One of the long ones bounds a cylinder with length $L=|w^{-1} (2 \varphi,0)|$ on the left, and the other bounds a cylinder with length $L$ on the right. This follows from an elementary argument using the fact that the Veech group contains the rotation by angle $\frac{\pi}{5}$ and Veech group elements can be used to send any periodic direction to the horizontal direction.
\end{remark}

\begin{remark}
The {\em combinatorial length} of a saddle connection on $\Pi_5$, $D$ or $\tilde D$ is the number of interiors of pentagons passed through counting multiplicity.
Note that combinatorial length is preserved under the natural covering maps between these surfaces. Suppose $\sigma$ is a long saddle connection on $\Pi_5$ with 
$$\hol~\sigma=L \cdot (\cos \theta, \sin \theta) = (a+bs^2,cs+ds^3)
\quad \text{with} \quad  
\theta \in [0, \tfrac{3\pi}{5})
\quad \text{and} \quad  
s=2 \sin \frac{\pi}{5}= \sqrt{\frac{5-\sqrt{5}}{2}}.$$
Suppose also that $\sigma$ bounds a cylinder of length $L$ on the left. Then the combinatorial length of $CL(\sigma)$ is 
$$CL(\sigma) = \begin{cases}
- b - c - 4d & \text{if $\theta \in [0,\frac{\pi}{5})$,} \\
a + 3b - d - 1 & \text{if $\theta \in [\frac{\pi}{5},\frac{2\pi}{5}),$} \\
2c + 6d - 1 & \text{if $\theta \in [\frac{2\pi}{5},\frac{3\pi}{5})$}.
\end{cases}$$
This can be derived from Davis and Leli\`evre's work on the combinatorial length of closed geodesics on the double pentagon \cite[Theorem  3.15]{davis2018periodic} by comparing the combinatorial length of the saddle connection and the cylinder it bounds on the left. The code from this article explicitly verifies that the formula works for the saddle connections in Tables \ref{table:list} and \ref{table:shortest}.
\end{remark}

\newpage

\begin{table}
  \centering
\begin{tabular}{lcccc}
 & Coset Representative & Exact Vector & 
 \begin{tabular}{@{}c@{}}Approximate \\ Length\end{tabular} &
 \begin{tabular}{@{}c@{}}Approximate \\ Vector\end{tabular} \\
 \hline
1 & $RRRTT$ & $ \left(-3 s^{2} + 12,\,-5 s^{3} + 19 s\right) $ & 16.2386 & (7.85410, 14.2128) \\
2 & $RRRTTRRT$ & $ \left(-21 s^{2} + 76,\,-5 s^{3} + 19 s\right) $ & 49.0816 & (46.9787, 14.2128) \\
3 & $RTTTRT$ & $ \left(-19 s^{2} + 72,\,-7 s^{3} + 25 s\right) $ & 49.1630 & (45.7426, 18.0171) \\
4 & $RRTRTTTT$ & $ \left(-16 s^{2} + 59,\,-33 s^{3} + 120 s\right) $ & 94.9181 & (36.8885, 87.4567) \\
5 & $RTRRTTRRRT$ & $ \left(-29 s^{2} + 106,\,-27 s^{3} + 99 s\right) $ & 98.0031 & (65.9230, 72.5173) \\
6 & $RRTTRTRRRT$ & $ \left(-30 s^{2} + 109,\,-27 s^{3} + 98 s\right) $ & 98.2417 & (67.5410, 71.3418) \\
7 & $RRTRRTTTT$ & $ \left(12 s^{2} - 41,\,-41 s^{3} + 148 s\right) $ & 110.117 & (-24.4164, 107.376) \\
8 & $RTTRTTTT$ & $ \left(-41 s^{2} + 152,\,-23 s^{3} + 81 s\right) $ & 111.521 & (95.3394, 57.8554) \\
9 & $RTRRTTTTT$ & $ \left(-18 s^{2} + 67,\,-39 s^{3} + 142 s\right) $ & 111.810 & (42.1246, 103.572) \\
10 & $RTRTRTT$ & $ \left(-23 s^{2} + 84,\,-39 s^{3} + 141 s\right) $ & 114.941 & (52.2148, 102.396) \\
11 & $RRRTTRTTTT$ & $ \left(17 s^{2} - 56,\,-53 s^{3} + 191 s\right) $ & 142.196 & (-32.5066, 138.430) \\
12 & $RRRTTTTRTT$ & $ \left(13 s^{2} - 42,\,-55 s^{3} + 199 s\right) $ & 146.570 & (-24.0344, 144.586) \\
13 & $RRTTRRTRT$ & $ \left(-68 s^{2} + 247,\,-21 s^{3} + 76 s\right) $ & 162.687 & (153.026, 55.2268) \\
14 & $RRTRTRRTT$ & $ \left(-66 s^{2} + 239,\,-27 s^{3} + 98 s\right) $ & 164.109 & (147.790, 71.3418) \\
15 & $RTRRTRTRRT$ & $ \left(-57 s^{2} + 208,\,-63 s^{3} + 229 s\right) $ & 211.047 & (129.228, 166.856) \\
16 & $RRRTTRTTRT$ & $ \left(-87 s^{2} + 318,\,-29 s^{3} + 105 s\right) $ & 211.985 & (197.769, 76.3215) \\
17 & $RRRTRTTTTRT$ & $ \left(-86 s^{2} + 316,\,-30 s^{3} + 108 s\right) $ & 212.102 & (197.151, 78.2237) \\
18 & $RTRRRTTTRTT$ & $ \left(-94 s^{2} + 345,\,-43 s^{3} + 154 s\right) $ & 242.130 & (215.095, 111.180) \\
19 & $RRTRTRRRTRTT$ & $ \left(21 s^{2} - 74,\,-113 s^{3} + 409 s\right) $ & 300.613 & (-44.9787, 297.229) \\
20 & $RTRRTRTTTTTT$ & $ \left(42 s^{2} - 145,\,-127 s^{3} + 458 s\right) $ & 343.284 & (-86.9574, 332.087) \\
21 & $RRRTTRTTTTRT$ & $ \left(-153 s^{2} + 560,\,-53 s^{3} + 191 s\right) $ & 375.042 & (348.559, 138.430) \\
22 & $RTTRTRRRTTRT$ & $ \left(-103 s^{2} + 374,\,-135 s^{3} + 487 s\right) $ & 422.378 & (231.658, 353.182) \\
23 & $RTTTTTTRTRTT$ & $ \left(-88 s^{2} + 319,\,-149 s^{3} + 536 s\right) $ & 435.359 & (197.387, 388.041) \\
24 & $RRRTTRTRTTRT$ & $ \left(-150 s^{2} + 544,\,-198 s^{3} + 716 s\right) $ & 619.524 & (336.705, 520.038) \\
25 & $RRTTTRTRRTRT$ & $ \left(-158 s^{2} + 572,\,-198 s^{3} + 716 s\right) $ & 628.894 & (353.649, 520.038) \\
26 & $RTRTRRTRRTTTT$ & $ \left(-132 s^{2} + 479,\,-263 s^{3} + 952 s\right) $ & 752.761 & (296.581, 691.874) \\
27 & $RTTRTRTRTRT$ & $ \left(-362 s^{2} + 1312,\,-114 s^{3} + 412 s\right) $ & 865.091 & (811.728, 299.131) \\
28 & $RTRTRTRTTRT$ & $ \left(-380 s^{2} + 1377,\,-125 s^{3} + 452 s\right) $ & 912.920 & (851.853, 328.283) \\
29 & $RTTRRTTTRRTRRT$ & $ \left(-15 s^{2} + 56,\,-375 s^{3} + 1357 s\right) $ & 986.655 & (35.2705, 986.025) \\
30 & $RTRRTRTRRTTRT$ & $ \left(27 s^{2} - 94,\,-423 s^{3} + 1531 s\right) $ & 1114.04 & (-56.6869, 1112.59) \\
31 & $RRTRTRTRTRRTT$ & $ \left(84 s^{2} - 302,\,-518 s^{3} + 1874 s\right) $ & 1374.11 & (-185.915, 1361.48) \\
\end{tabular}
\vspace{1em}
\caption{Representatives of closed saddle connections in each unfolding-symmetry class as described in Appendix \ref{appendix:list}.}
\label{table:list}
\end{table}

\newpage 

\section{Shortest closed saddle connections}
\label{appendix:shortest}

Table \ref{table:shortest} lists a shortest closed saddle connection in each equivalence class on $\tilde D$. Rows are numbered to match rows of Table~\ref{table:list}, and we omit rows where a shortest saddle connection already appears in Table~\ref{table:list}.

We will briefly explain how we arrived at this list. FlatSurf \cite{flatsurf} has the ability to form a list of the saddle connections of length less than $L$ leaving a singularity on a translation surface. We apply this to the double pentagon to obtain all saddle connections of length less than $622$ up to rotation. 

Let $\arg:{\mathbb R}^2 \to {\mathbb R}/2 \pi {\mathbb Z}$ be the map sending a vector to its direction.
There is an analog of the slow continued fraction algorithm that takes the holonomy vector ${\mathbf v} \in \arg^{-1}\big([0,\tfrac{3 \pi}{5})\big)$ of a saddle connection of $\Pi_5$ and produces an element $w$ of the semigroup $\langle R^{-1}, T^{-1}\rangle$ that sends $w$ to either $(2,0)$ if the saddle connection is short and $(2 \phi, 0)$ if the saddle connection is long.  One step of such an algorithm is given by
$$\Phi:\arg^{-1}\big([0,\tfrac{4 \pi}{5})\big) \to \arg^{-1}\big([0,\tfrac{4 \pi}{5})\big); \quad
{\mathbf v} \mapsto \begin{cases}
T^{-1} {\mathbf v} & \text{if $\arg({\mathbf v}) \in [0, \frac{\pi}{5})$,} \\
R^{-1} {\mathbf v} & \text{otherwise.} \\
\end{cases}$$
Observe that all holonomies of saddle connections converge as suggested because the collection of all holonomies of saddle connections is a discrete set.  When $T^{-1}$ is applied, it shrinks non-horizontal vectors, while $R^{-1}$ preserves vector lengths, but can only be applied at most three times in a row. This idea is due to Rosen \cite{rosen1954class}.
This algorithm also produces the word $w$, which determines a coset in $\langle T, J \rangle \bs V_\pm(\Pi_5) / V_\pm(\tilde D)$ associated to the saddle connection.

Further, we can tell using work in this paper if a saddle connection $\sigma$ with holonomy ${\mathbf v}=\hol~\sigma$ is closed. First it must be a long saddle connection, so iterative application of $\Phi$ should bring ${\mathbf v}$ to $(2 \phi, 0)$. The algorithm also produces $w \in \langle R^{-1}, T^{-1}\rangle$. When viewed as a matrix, we see there is an affine homeomorphism $\phi:\tilde D \to w(\tilde D)$ with derivative $w$ which carries $\sigma$ to a long horizontal cylinder on $w(\tilde D)$. Thus $\sigma$ is closed if and only if $w(\tilde D)$ represents a surface with closed long horizontal saddle connections. 
The surface index of $w(\tilde D)$ can be calculated by applying the sequence of permutations $r^{-1}$ and $t^{-1}$ corresponding to $w$, where $r$ and $t$ are the permutations found in \S \ref{sect:dodecahedron monodromy}. We know the orbits of surface indices under $\langle T,J\rangle$ and which orbits correspond to closed long horizontal saddle connections, so we can determine whether $\sigma$ is closed, and if so, which equivalence class it lies in.

\begin{table}
  \centering
\begin{tabular}{lcccc}
 & 
 Coset Representative &
 Exact Vector & 
 \begin{tabular}{@{}c@{}}Approximate \\ Length\end{tabular} &
 \begin{tabular}{@{}c@{}}Approximate \\ Vector\end{tabular} \\
 \hline
10 & $\big(T^2(T^2R)^2\big)^{-1}$ & $ \left(-52 s^{2} + 183,\,-3 s^{3} + 12 s\right) $ & 111.521 & (111.138, 9.23305) \\
16 & $\big(TR^3T^3(TR^3)^2TR\big)^{-1}$ & $ \left(-62 s^{2} + 227,\,-89 s^{3} + 326 s\right) $ & 277.350 & (141.318, 238.647) \\
17 & $\big(T^2(T^2R)^2R^2\big)^{-1}$ & $ \left(-48 s^{2} + 174,\,-34 s^{3} + 126 s\right) $ & 142.196 & (107.666, 92.8855) \\
18 & $\big(T^2RT(RTR^2)^2\big)^{-1}$ & $ \left(-88 s^{2} + 319,\,-75 s^{3} + 272 s\right) $ & 279.518 & (197.387, 197.910) \\
19 & $\big(T^2RT(TR^3)^3\big)^{-1}$ & $ \left(-64 s^{2} + 234,\,-54 s^{3} + 198 s\right) $ & 205.478 & (145.554, 145.035) \\
20 & $\big(T(TR)^2R(RT)^2R^2\big)^{-1}$ & $ \left(-94 s^{2} + 341,\,-81 s^{3} + 294 s\right) $ & 300.613 & (211.095, 214.025) \\
21 & $\big(T^2RT^3R^3TR\big)^{-1}$ & $ \left(-109 s^{2} + 390,\,-13 s^{3} + 49 s\right) $ & 242.130 & (239.366, 36.4832) \\
22 & $\big(TRT^3(TR)^2R^2\big)^{-1}$ & $ \left(-93 s^{2} + 334,\,-19 s^{3} + 71 s\right) $ & 212.102 & (205.477, 52.5981) \\
23 & $\big(T(R^2TR)^4R^2T^2R^2\big)^{-1}$ & $ \left(-66 s^{2} + 239,\,-89 s^{3} + 322 s\right) $ & 276.716 & (147.790, 233.944) \\
24 & $\big((TR)^2RT^2R^3TR\big)^{-1}$ & $ \left(-83 s^{2} + 302,\,-91 s^{3} + 331 s\right) $ & 305.441 & (187.297, 241.275) \\
25 & $\big(T^2(T^2R^3)^2TR^2\big)^{-1}$ & $ \left(-159 s^{2} + 576,\,-13 s^{3} + 47 s\right) $ & 357.899 & (356.267, 34.1320) \\
26 & $\big(T^6R^3T^2R^2\big)^{-1}$ & $ \left(-116 s^{2} + 418,\,-6 s^{3} + 22 s\right) $ & 258.195 & (257.692, 16.1150) \\
27 & $\big(T^2(T^2R)^2TR^3\big)^{-1}$ & $ \left(-165 s^{2} + 592,\,-11 s^{3} + 41 s\right) $ & 365.237 & (363.976, 30.3278) \\
28 & $\big(TRT^2(T^2R)^2R^2\big)^{-1}$ & $ \left(-164 s^{2} + 590,\,-34 s^{3} + 126 s\right) $ & 375.042 & (363.358, 92.8855) \\
29 & $\big(T^3(TR)^2R^2T^2R^3\big)^{-1}$ & $ \left(-117 s^{2} + 424,\,-85 s^{3} + 311 s\right) $ & 347.229 & (262.310, 227.512) \\
30 & $\big(T^2R^2(TR^3T)^2R^3\big)^{-1}$ & $ \left(-143 s^{2} + 522,\,-23 s^{3} + 83 s\right) $ & 329.919 & (324.379, 60.2066) \\
31 & $\big(T(TR^3TR)^2RTR^2\big)^{-1}$ & $ \left(-272 s^{2} + 984,\,-48 s^{3} + 174 s\right) $ & 621.137 & (608.105, 126.569) \\
\end{tabular}
\vspace{1em}
\caption{Shortest saddle connections in cosets corresponding to Table \ref{table:list}. Saddle connections in cosets 1--9 and 11--15 were not included
because Table \ref{table:list} already lists a shortest representative. Conventions from Appendix \ref{appendix:list} are followed here.}
\label{table:shortest}\end{table}

\begin{figure}[!ht]
    \centering
    \includegraphics[width=0.62\textwidth]{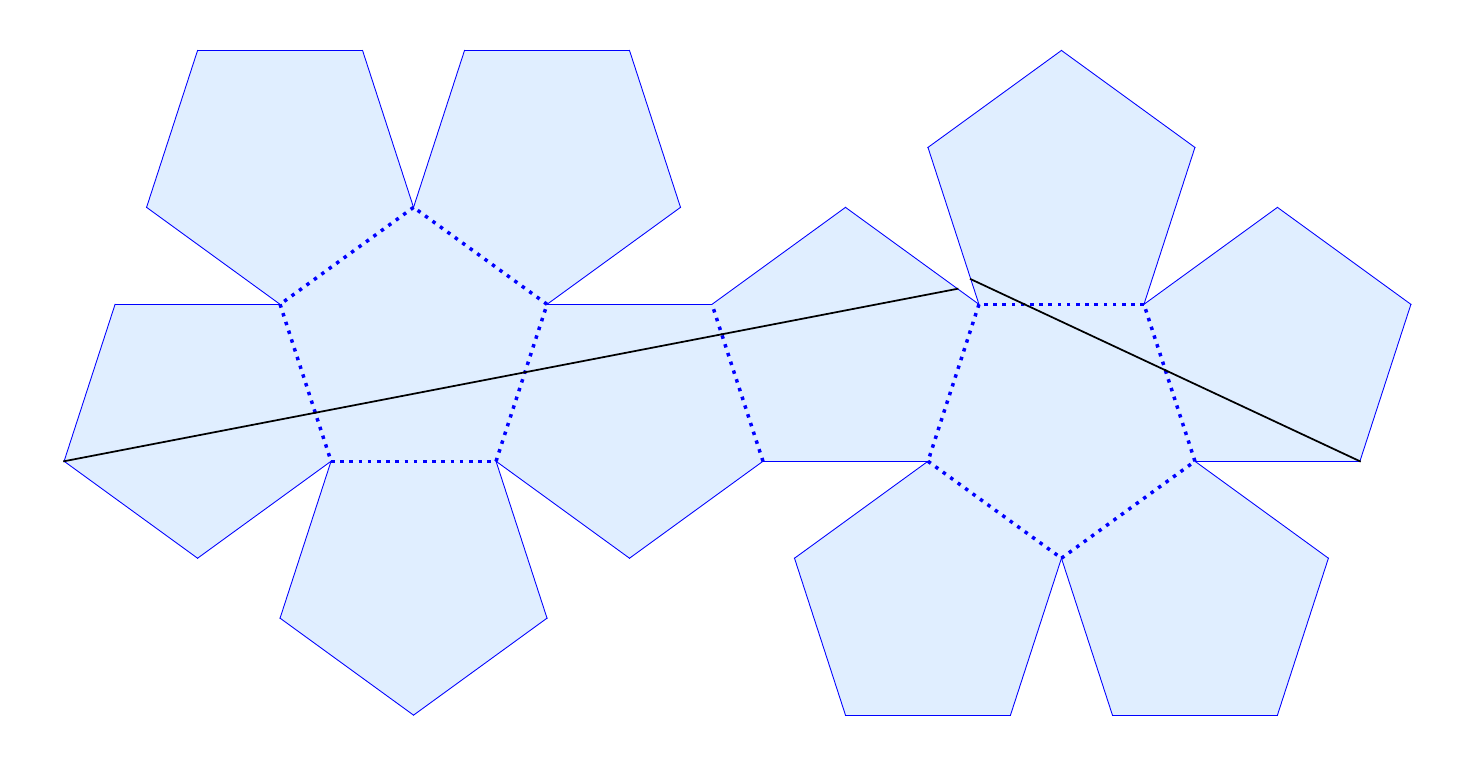}
    \includegraphics[width=0.62\textwidth]{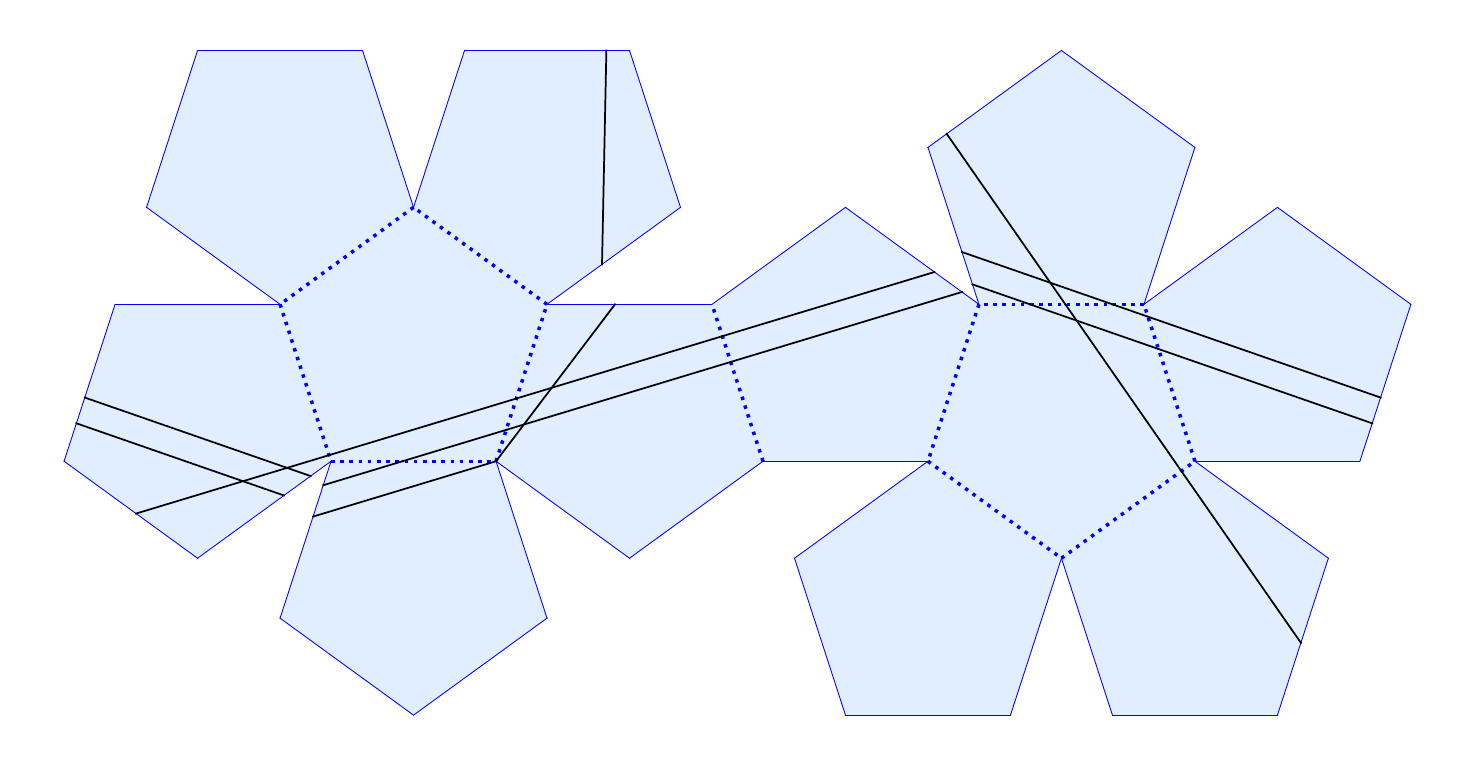}
    \includegraphics[width=0.62\textwidth]{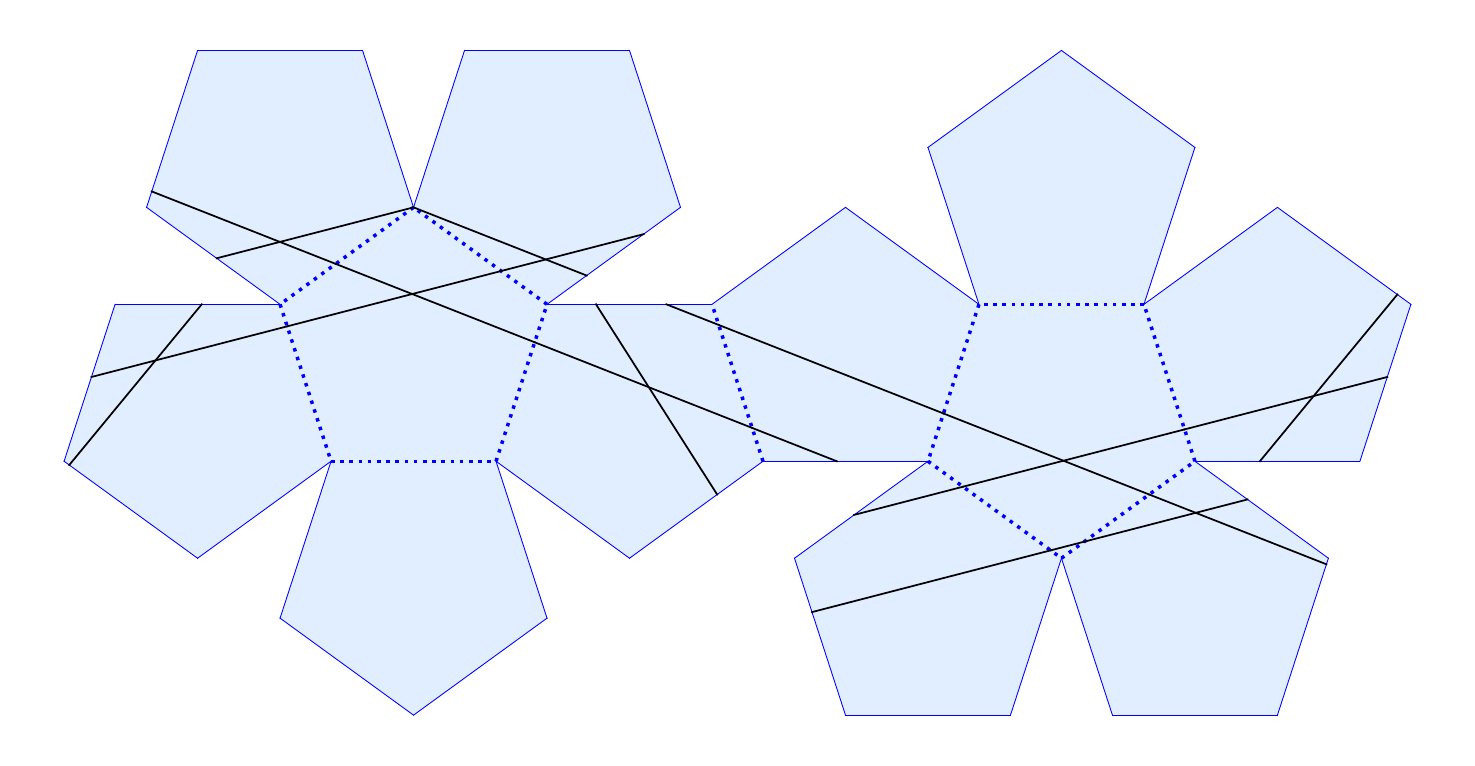}
    \caption{Representatives of the first three equivalence classes from Table~\ref{table:list}.}\label{fig:sc_0}
\end{figure}

\newpage

\bibliographystyle{amsalpha}
\bibliography{bibliography}

\end{document}